\documentclass[11pt]{amsart}
\usepackage[top=1in, bottom=1in, left=1in, right=1in]{geometry}
\usepackage[foot]{amsaddr}
\makeatletter
\renewcommand{\email}[2][]{%
  \ifx\emails\@empty\relax\else{\g@addto@macro\emails{,\space}}\fi%
  \@ifnotempty{#1}{\g@addto@macro\emails{\textrm{(#1)}\space}}%
  \g@addto@macro\emails{#2}%
}
\makeatother

\makeatletter
\@namedef{subjclassname@2020}{\textup{2020} Mathematics Subject Classification}
\makeatother

\usepackage{float,longtable}
\usepackage{amsthm,amsmath,amssymb,enumerate,hyperref,fontenc,graphicx}
\usepackage{stmaryrd}
\usepackage{tikz,tikz-cd}
\usetikzlibrary{matrix,arrows,decorations.pathmorphing,decorations.markings}
\newcommand{\midarrow}{\tikz \draw[-triangle 45] (0,0) -- +(.1,0);}
\usepackage{pgfplots}
\pgfplotsset{
        compat=1.16,
    }

\theoremstyle{plain}
\newtheorem{theorem}{Theorem}[section]
\newtheorem{lemma}[theorem]{Lemma}
\newtheorem{proposition}[theorem]{Proposition}
\newtheorem{corollary}[theorem]{Corollary}
\newtheorem{conjecture}[theorem]{Conjecture}

\theoremstyle{definition}
\newtheorem{definition}[theorem]{Definition}

\newtheorem{remark}[theorem]{Remark}
\newtheorem{algorithm}[theorem]{Algorithm}
\newtheorem{construction}[theorem]{Construction}

\newcommand{\mbbC}{{\mathbb C}}

\newcommand{\mbbN}{{\mathbb N}}
\newcommand{\mbbR}{{\mathbb R}}
\newcommand{\mbbZ}{{\mathbb Z}}

\newcommand{\mcalC}{{\mathcal C}}
\newcommand{\mcalH}{{\mathcal H}}
\newcommand{\mcalK}{{\mathcal K}}

\newcommand{\mcalP}{{\mathcal P}}
\newcommand{\mcalT}{{\mathcal T}}
\newcommand{\mcalW}{{\mathcal W}}
\newcommand{\mcalX}{{\mathcal X}}

\newcommand{\mfrako}{{\mathfrak{o}}}

\newcommand{\va}{{\mathbf a}}
\newcommand{\vb}{{\mathbf b}}
\newcommand{\ve}{{\mathbf e}}

\newcommand{\vE}{{\mathbf E}}
\newcommand{\vDelta}{{\mathbf \Delta}}
\newcommand{\vUpsilon}{{\mathbf \Upsilon}}

\newcommand{\bmcalC}{\overline{\mathcal C}}
\newcommand{\bmcalP}{\overline{\mathcal P}}

\def\dunderline#1{\underline{\underline{#1}}}

\DeclareMathOperator{\Hom}{\operatorname{Hom}}
\DeclareMathOperator{\spec}{\operatorname{Spec}}

\DeclareMathOperator{\imag}{\operatorname{Im}}
\DeclareMathOperator{\ab}{\operatorname{ab}}
\DeclareMathOperator{\vol}{\operatorname{vol}}
\DeclareMathOperator{\Jac}{\operatorname{Jac}}

\numberwithin{equation}{section}

\begin{document}

\title{Homological Spectral Graph Theory and Weighted Cycle Counting}
\author{Ye Luo} 
\address[A1]{School of Informatics, Xiamen University, Xiamen, Fujian 361005, China}
\email[A1]{luoye@xmu.edu.cn}

\author{Arindam Roy}
\address[A2]{Department of Mathematics and Statistics, University of North Carolina at Charlotte, 9201 University City Blvd., Charlotte, NC 28223, USA}
\email[A2]{arindam.roy@charlotte.edu}



\date{}

\begin{abstract}
Twisting a finite graph by its homology characters produces families of generally non-Hermitian weighted vertex and edge adjacency matrices.  We show that these families are organized by the canonical $2$-torsion character and, when it exists, a period character.  The canonical character $\theta$ translates every complex character to an anti-isospectral partner, for every positive directed weight.  On the unitary character torus, the untwisted matrices maximize spectral radius strictly apart from topologically forced exceptions; for edge adjacency, the only possible additional maximizer is the period character $\eta$.  For symmetric weights, these results locate the two outer spectral edges of maximal-abelian and periodic Bloch families.  Fourier inversion of twisted edge traces then yields weighted circuit and prime-cycle formulas in integral and finite-quotient homology classes, with $\theta$ and $\eta$, when present, governing the resulting parity effects.  Thus spectral antisymmetry, spectral-edge extremality, and weighted cycle counting arise from the same homological structure.
\end{abstract}

\maketitle

\section{Introduction}\label{S:intro}
Prime cycles on a finite graph carry several layers of information.  Their
lengths are encoded by the Ihara zeta function, their homology classes are
resolved by characters, and a directed edge weight assigns a multiplicative
size to each traversal.  These layers do not automatically coexist well.
Homological twisting introduces phases, while asymmetric directed weights
destroy self-adjointness and need not respect reversal of a cycle.  One might
therefore expect the resulting family of twisted matrices to have little
global spectral organization.  The central phenomenon of this paper is that a
topological organization nevertheless survives.  Canonical translation
controls spectral sign reversal, while a possible period character gives the
complete equality locus for the edge spectral radius; both are independent
of the weights.

The prime number theorem and its Dirichlet refinement provide the arithmetic
model.  They exhibit two mechanisms that recur in geometric counting
problems: an explicit formula obtained from a zeta or $L$-function, and Fourier
inversion over a character group.  In their classical form,
$$
\pi(x)\sim\frac{x}{\log x},
\qquad
\pi(x;q,a)\sim\frac{1}{\phi(q)}\frac{x}{\log x}
\quad ((a,q)=1).
$$
Analogues based on Dedekind and Selberg zeta functions lead to prime-ideal and
prime-geodesic theorems.  This paper develops a weighted, homology-resolved
version of the corresponding prime-cycle problem for finite graphs.  The
spectral input is not merely an analogue of Fourier analysis in arithmetic:
the graph has a distinguished half-period, invisible in the classical
Dirichlet model, which pairs the twisted spectra and creates a parity bias in
the counting problem.

Throughout, $G$ is a connected finite graph with vertex set $V(G)$, $n$
vertices, $m$ unoriented edges, and genus $g=m-n+1$; multiple edges and loops
are allowed.  As usual, $G$ is \emph{bipartite} if its vertex set admits a
partition $V(G)=V_1\sqcup V_2$ such that every edge has one endpoint in each
part.  In particular, a bipartite graph has no loops; any graph with a loop is
non-bipartite.
An orientation $\mfrako$ chooses one oriented edge above each unoriented edge;
the chosen set is $\vE_{\mfrako}(G)$ and the set of all $2m$ oriented edges is
$\vE(G)$.  For $\va\in\vE(G)$, its initial and terminal vertices are
$\va(0)$ and $\va(1)$, and $\va^{-1}$ denotes the reverse orientation.  A
\emph{positive directed weight} is a function
$W:\vE(G)\to\mbbR_{>0}$; no relation between $W(\va)$ and
$W(\va^{-1})$ is imposed.

A walk $\Delta=\va_1\cdots\va_l$ has length $\tau(\Delta)=l$ and satisfies
$\va_j(1)=\va_{j+1}(0)$.  Its inverse is
$\Delta^{-1}=\va_l^{-1}\cdots\va_1^{-1}$; a sub-walk is a contiguous
subsequence, and a path is a walk with no repeated vertices.  A
closed walk has a \emph{backtrack} if
$\va_{j+1}=\va_j^{-1}$ and a \emph{tail} if
$\va_l=\va_1^{-1}$.  A \emph{circuit} is a closed walk with neither, and a
circuit is \emph{prime} if it is not a proper power.  Every circuit has a
unique expression $C=P^r$ with $P$ prime; write $r(C)=r$.  Quotienting
circuits by cyclic change of base vertex gives \emph{cycles}, with prime
cycles defined similarly.  We denote the sets of circuits and prime circuits
by $\mcalC$ and $\mcalP$, and the corresponding cycle sets by $\bmcalC$ and
$\bmcalP$.  The weight of a walk is
$W(\Delta)=\prod_{j=1}^lW(\va_j)$, and $W\equiv1$ recovers ordinary counting.

The graph prime-counting function is
$$
\pi_G(l):=\#\{[P]\in\bmcalP:\tau(P)=l\}.
$$
Its Euler product is the Ihara zeta function
$$
z_G(u)=\prod_{[P]\in\bmcalP}(1-u^{\tau(P)})^{-1}.
$$
When $g\geq2$, if $R_G$ is its radius of convergence and
$\nu_G=\gcd\{\tau(P):[P]\in\bmcalP\}$, then, along admissible lengths,
$$
\pi_G(l)\sim\frac{\nu_G}{lR_G^l};
$$
see, for example, Hashimoto~\cite{H1989Zeta,H1992Artin} and
Horton--Stark--Terras~\cite{HST2006What}.  The associated
\emph{edge adjacency matrix} (or nonbacktracking matrix) $B$ is indexed by
$\vE(G)$, with $B_{\va\vb}=1$ exactly when
$\va(1)=\vb(0)$ and $\vb\neq\va^{-1}$; in this case we say that $\va$
\emph{feeds into} $\vb$ and write $\va\shortrightarrow\vb$.  Its trace gives the explicit formula
$$
N(l)=\operatorname{tr}(B^l)
=\sum_{\lambda\in\spec B}\lambda^l
=\sum_{d\mid l}d\,\pi_G(d),
$$
where $N(l)$ is the number of length-$l$ circuits.

Our refinement asks how these circuits and prime cycles are distributed after
their abelianizations are prescribed in $H_1(G,\mbbZ)$ or in a finite
quotient.  Characters provide the Fourier variables, while twisted weighted
vertex and edge adjacency matrices provide the spectral transforms.  A
distinguished $2$-torsion character, the \emph{canonical character}, creates
an antisymmetry absent from the classical Dirichlet model.  When it exists, a
second $2$-torsion invariant, the \emph{period character}, refines the edge
extrema and the parity in the counting formulas.  We refer to this
interaction between homology, characters, and twisted spectra as
\emph{homological spectral graph theory}.

Three questions are intertwined.  First, does there exist a spectral
involution that pairs every character twist, even when the weights are
asymmetric and the matrices are non-Hermitian?  Second, in the Hermitian
specialization, does that involution locate distinguished edges of the
Floquet spectrum, including when a periodic graph samples only part of the
character torus?  Third, can the same mechanism be seen arithmetically in the
distribution of weighted prime cycles among homology classes?  Our answers
begin with the canonical character $\theta$ but reveal a second invariant.
Canonical translation $\xi\mapsto\theta+\xi$ negates both the vertex and edge
spectra.  The vertex extrema occur at the trivial and canonical characters.
For the edge matrix, normalized cycle-length parity defines at most one
additional \emph{period character} $\eta$, and the complete equality locus
is $\{0\}$ or $\{0,\eta\}$.  In the non-bipartite case $\eta=\theta$;
in the bipartite case it may not exist.  The same period character determines
which finite-quotient homology classes vanish and when the main counting term
doubles.  Thus spectral antisymmetry, edge extremality, and weighted
prime-cycle distribution are different faces of one homological mechanism.

\subsection{Related work}
The closest geometric precedent is the distribution of closed geodesics in
homology classes.  The theory of trace formulas and prime geodesics goes back to
Delsarte, Huber, and Hejhal~\cite{D1942Sur,H1961Zur,H1976Selberg}, with
arithmetic refinements by Sarnak, Iwaniec, and Soundararajan--Young
\cite{S1982Class,I1984Prime,SY2013The}; Luo--Sarnak's quantum-ergodicity work
is another important arithmetic application of this circle of ideas
\cite{LS1995Quantum}.  For a compact
negatively curved manifold, Adachi--Sunada proved that the exponential growth
rate in every fixed homology class equals the topological entropy
\cite{AS1987Homology}; sharper results were obtained in
\cite{PS1987Geodesics,KS1987Homology,L1989Closed,P1991Homology}.  For a
hyperbolic surface of genus $g$, in particular,
$$
\pi(\alpha,x)\sim\frac{(g-1)^g e^x}{x^{g+1}},
$$
so prime geodesics are asymptotically uniform across fixed homology classes.
Following this geometric precedent, Liu obtained homology-class asymptotics
for metrized finite graphs under weak-mixing and Diophantine
assumptions~\cite{L2004Asymptotic}; his weak-mixing hypothesis requires the
cycle lengths to generate a dense subgroup of $\mbbR$ and therefore excludes
the combinatorial setting, where all cycle lengths are integers.  In this
discrete setting, our asymptotic constants admit a direct interpretation
through twisted weighted edge spectra, and the canonical-character
antisymmetry of these spectra governs parity-dependent vanishing and the
possible doubling of the leading term.

Ihara introduced his zeta function for discrete subgroups of $p$-adic groups
\cite{I1966Padic}, and Serre emphasized its graph-theoretic meaning
\cite{S1980Trees}.  Hashimoto obtained the nonbacktracking determinant and
prime-cycle theorem and proved a graph-theoretic Chebotarev density theorem
\cite{H1989Zeta,H1992Artin}, while Bass removed regularity restrictions in the
vertex determinant~\cite{B1992Ihara}.  Multivariate, weighted, and Artin-type
extensions were developed in
\cite{ST1996Zeta,ST2000Zeta,TS2007Zeta,HST2006What,MS2004Weighted}; see also
\cite{HFGO1999Emerging,KS2000Zeta,T2011Zeta}.  Twisted Perron--Frobenius theory
and its associated $L$-functions were developed by
Adachi--Sunada~\cite{AS1987Twisted}.  Thus character- and gain-twisted
nonbacktracking transfer operators themselves have substantial precedents in
the zeta and covering literature.  The signed special case also appears
explicitly through sign assignments on Hashimoto matrices and their
nonbacktracking transitions
\cite{GRW2017Balance,LHW2024SignedZeta}.  Hasegawa--Saito studied prime-cycle
lengths in arithmetic progressions~\cite{HS2015Gen}, and weighted
prime-geodesic theorems in a broader operator setting appear in
Deitmar~\cite{D2023Weighted}.

Our spectral input belongs to several neighboring theories.  Classical
spectral graph theory studies graph structure through adjacency and Laplacian
spectra~\cite{C1997Spectral,CRS2010Introduction,BH2012Spectra}.  A full complex
character is represented by a $\mbbC^\times$-valued gain, and switching
classes of such gains with fixed underlying graph form the full complex
character torus.  Unit gains give its compact unitary subtorus and the
familiar complex unit gain graphs~\cite{Z1989Biased,R2012Spectral};
balancedness, sign-symmetry, and spectral symmetry in that specialization
have been studied in~\cite{MKS2022On,WvD2023Symmetry}.  General complex gain
graphs also admit a discrete line-bundle interpretation
\cite{K2011Spanning,BPT2015Matrix}.  Our approach uses all complex switching
classes simultaneously rather than one class at a time.  It identifies one
topological translation that pairs every class with an anti-isospectral
partner.  The positive directed weight $W$ is independent of the gain and
need not be symmetric or reciprocal.  Thus the pairing holds over the full
complex character torus for arbitrary $W$, while its spectral-radius
consequences are confined to the unitary subtorus.

For symmetric weights the unitary vertex family belongs to magnetic graph
theory: gauge-equivalent potentials are unitarily equivalent, and zero field
is globally extremal by the one-particle diamagnetic inequality
\cite{LL1993Fluxes}.  The local Morse and nodal behavior at signings, critical
submanifolds, and local extrema are studied in
\cite{B2013Nodal,CV2013Magnetic,AG2023Morse,ABG2025Smooth}.  In the simple
unweighted unit-gain setting, the vertex spectral-radius equality case also
follows from~\cite[Theorem~4.4]{MKS2022On}.  All of these results concern
vertex operators.  On the edge side, the
preceding zeta, covering, and signed-graph works provide operator-level
precedents, so our novelty claim is not the introduction of a gain-twisted
Hashimoto matrix.  Rather, we study the entire character family and identify
the canonical translation that negates its spectra, then classify the
spectral-radius equality locus.  The resulting genus-one behavior and the
period-character classification for the generally non-Hermitian edge
operator are not supplied by the existing zeta or signed-operator analyses,
nor by the magnetic vertex arguments.  We further allow full complex
characters, directed weights, and weighted homology-resolved counting.

The unitary tori are also Bloch tori for abelian covers~\cite{K2016Overview}.
A general cover samples an affine subtorus, while the maximal abelian cover
samples the whole character torus.  The spectral edges conjecture predicts
generic isolation, non-degeneracy, and single-band attainment of band extrema
\cite{BCCM2022Local,FS2025Corners}.  Our complementary theorem identifies the
two outermost edges and when an affine Bloch subtorus inherits them, with
uniqueness and simplicity when inheritance occurs; it makes no claim about
every internal band edge or critical point.

\subsection{Our contributions and paper organization}
The paper has four principal contributions.

\medskip
\noindent\textbf{(1) Canonical and period characters: spectral antisymmetry and extremality.}
Theorem~\ref{T:SpecDuality} is the organizing result of the paper.  It
introduces the canonical character $\theta$, which records closed-walk parity
and pairs every full complex character with an anti-isospectral partner for
arbitrary positive directed weights.  On the unitary character torus, the
same theorem proves that the untwisted operators maximize spectral radius and
classifies every equality case.  For edge adjacency in genus at least two,
the equality problem exhibits a subtle topological alternative: besides the
trivial character there is at most one period character $\eta$.  It equals
$\theta$ for non-bipartite graphs, whereas for bipartite graphs it may or may
not exist.

These two distinguished characters connect the other parts of the paper.
For symmetric weights, the canonical pairing and strict extremality locate
the outer spectral edges of the associated Bloch families.  For weighted
edge traces and cycle counting, the same extremal characters govern the
vanishing, parity, and leading-term phenomena in homology classes.  The
strictness proof combines coherent-walk criteria with Gelfand's formula, and
its independence from the numerical values of the weight explains the stable
maximizing loci in Figure~\ref{F:SpecRad}.

\medskip
\noindent\textbf{(2) Outer spectral edges.}
For symmetric weights the vertex matrices are Hermitian.
Theorem~\ref{T:outer-spectral-edges} then identifies the upper and lower outer edges
of the maximal-abelian-cover spectrum with the trivial and canonical
characters.  More generally, it determines when these two edges are inherited
by an affine Bloch subtorus, including a fixed real sign gain.  This is an
outer-edge result, not a claim about all internal band edges.  In gain-graph
language, the theorem organizes all switching classes over a fixed underlying
graph.  An outer edge is inherited exactly when the affine Bloch subtorus
contains the appropriate trivial or canonical translate; its attaining
character is then unique and the fiber eigenvalue is simple.

\medskip
\noindent\textbf{(3) Weighted trace and Fourier formulas.}
For every positive directed weight, Theorem~\ref{T:Fourier} identifies
$\operatorname{tr}(B_{\omega,W}^l)$ with the Fourier transform of the
$W$-weighted circuit sums in integral homology and in finite-quotient
homology.  Its inverse formulas isolate individual classes, and
Theorem~\ref{T:trace-formula} extends the trace identity from powers to
analytic test functions.  No symmetry or reciprocity of $W$ is needed.
The canonical character also yields exact parity-vanishing laws.  Setting
$W\equiv1$ recovers ordinary circuit counts.

\medskip
\noindent\textbf{(4) Weighted prime cycles in homology classes.}
We derive exact prime-power identities for weighted prime-cycle and cycle
sums.  Repetition replaces $W$ by the pointwise power $W^r$, a feature absent
from literal counting.  A strict powered-weight spectral gap makes proper
powers exponentially negligible and yields Theorem~\ref{T:counting}.  The
period character produces parity vanishing and a factor $2$ on admissible
classes.  In the non-bipartite case this is precisely the canonical-character
parity effect; in the bipartite case it occurs exactly when normalized length
parity descends to homology.

\medskip
\noindent\textbf{Organization.}
Section~\ref{S:prelim} collects the preliminary material, with proofs deferred
to Appendix~\ref{B:prelim-proofs}.  Section~\ref{S:main} states the main
spectral, Floquet, trace, and counting results, and
Section~\ref{S:examples} gives weighted and periodic examples.
Section~\ref{S:proof-antisymmetry} proves the principal spectral theorem.
Appendix~\ref{A:notation-guide} is a notation guide, and
Appendix~\ref{B:prelim-proofs} contains the deferred proofs of the preliminary
facts.  Appendix~\ref{C:regular-specializations} records the regular-graph
specializations, Appendix~\ref{D:supplementary-cotree} gives a complementary
cotree treatment of edge spectral-radius rigidity, and
Appendix~\ref{E:weighted-transforms} develops the weighted zeta and
$L$-function transforms.

\section{Preliminaries} \label{S:prelim}
The material in this section fixes notation and records standard or elementary facts used later. We keep the statements here so that the main results can be read without interruption; proofs of the preliminary facts are collected in Appendix~\ref{B:prelim-proofs}.

\subsection{Orthogonal decomposition} \label{SS:ortho-decomp}

Let $\mcalT(G)$ be the real linear space on $\vE(G)$ such that $1\cdot \ve^{-1} = -1\cdot \ve$ for any oriented edge $\ve$. Alternatively, by fixing an orientation on the edges, we may also consider $\mcalT(G)$ as a real linear space on $\vE_{\mfrako}(G)=\{\ve_1,\cdots,\ve_m\}$ where  $\ve_1,\cdots,\ve_m$ are the positively oriented edges. Geometrically, one may consider an element of $\mcalT(G)$ as a \emph{``vector field''} on $G$.  Then a \emph{``$1$-form''} $\omega$ is an element of  the linear dual $\Omega(G)$ of $\mcalT(G)$. We call $\mcalT(G)$ the \emph{tangent space} of $G$, and $\Omega(G)$ the \emph{cotangent space} of $G$, sometimes also written as $\mcalT$ and $\Omega$ respectively for simplicity when $G$ is provided.  For $\alpha\in \mcalT$ and $\omega\in \Omega$, we use $\omega(\alpha)$ to represent the value of the pairing between $\omega$ and $\alpha$. For $i=1,\cdots,m$, let $d\ve_i$ be a $1$-form such that $d\ve_i(\ve_j)=\delta_{ij}$ for $j=1,\cdots,m$. Then $\{d\ve_1,\cdots,d\ve_m\}$ is a basis of $\Omega$. In addition, the \emph{standard inner product} on $\mcalT$ with respect to the basis $\{\ve_1,\cdots,\ve_m\}$ and the \emph{standard inner product} on $\Omega$ with respect to $\{d\ve_1,\cdots,d\ve_m\}$ are defined by $\langle\ve_i,\ve_j\rangle=\langle d\ve_i,d\ve_j\rangle := d\ve_i(\ve_j)=\delta_{ij}$ for $i,j=1,\cdots,m$.

For a walk $\Delta=\va_1\cdots\va_N$, let the \emph{abelianization} of $\Delta$ be $\Delta^{\ab}=\va_1+\cdots+\va_N\in \mcalT$, and the integral of a $1$-form $\omega$ along $\Delta$ is $\int_\Delta \omega:=\omega(\Delta^{\ab})=\omega(\va_1)+\cdots+\omega(\va_N)$.

Let $R$ be a commutative ring with multiplicative identity. Let $C_0(G,R)$ be the free $R$-module on $V(G)$ and $C_1(G,R)$ be the free $R$-module on $\vE_{\mfrako}(G)$. Let $C^0(G,R)=\Hom_R(C_0(G,R),R)$ and $C^1(G,R)=\Hom_R(C_1(G,R),R)$. By convention, an element of $C_0(G,R)$ (resp.\ $C_1(G,R)$) is called a \emph{$0$-chain} (resp.\ \emph{$1$-chain}) with coefficients in $R$, and an element of $C^0(G,R)$ (resp.\ $C^1(G,R)$) is called a \emph{$0$-cochain} (resp.\ \emph{$1$-cochain}) with coefficients in $R$. Note that $C_1(G,\mbbZ)$ is a full-rank lattice in $C_1(G,\mbbR)=\mcalT(G)$, and $C^1(G,\mbbZ)$ is a full-rank lattice in $C^1(G,\mbbR)=\Omega(G)$. 

There is a canonical isomorphism $i_R^0:C_0(G,R)\xrightarrow{\sim} C^0(G,R)$ sending a $0$-chain $\sum_{v\in V(G)} c_v\cdot v \in C_0(G,R)$ to a $0$-cochain $f\in C^0(G,R)$ given by $f(v)=c_v$, and a canonical isomorphism $i_R^1:C_1(G,R)\xrightarrow{\sim} C^1(G,R)$ sending a $1$-chain $\sum_{\ve\in \vE_{\mfrako}(G)} c_{\ve}\cdot \ve \in C_1(G,R)$ to a $1$-cochain $\omega\in C^1(G,R)$ given by $\omega(\ve)=c_{\ve}$.  Note that $i_{\mbbR}^1(\ve_i)=d\ve_i$ for $i=1,\cdots,m$.

Let $\partial_R: C_1(G,R) \to C_0(G,R)$ be the \emph{boundary map} defined by $$\partial_R\left(\sum_{\ve\in \vE_{\mfrako}(G)}c_\ve \cdot \ve\right)=\sum_{\ve\in \vE_{\mfrako}(G)}\left(c_\ve\cdot \ve(1)-c_\ve\cdot \ve(0)\right),$$ and  $d_R: C^0(G,R) \to C^1(G,R)$ be the \emph{differential map} defined by 
$$d_Rf=\sum_{\ve\in \vE_{\mfrako}(G)} \left(f(\ve(1))-f(\ve(0))\right) \cdot d\ve$$ for each $f\in C^0(G,R)$. Then the adjoint operator $\partial_R^*:C_0(G,R) \to C_1(G,R)$ of $\partial_R$ is defined by $\partial_R^*=(i_R^1)^{-1}\circ d_R\circ i_R^0$, and the adjoint operator $d_R^*:C^1(G,R) \to C^0(G,R)$ is defined by $d_R^*=i_R^0 \circ \partial_R\circ (i_R^1)^{-1}$.  For simplicity of notation, we also write $d_\mbbR$ and $d_\mbbR^*$ as $d$ and $d^*$ respectively. We say $d^*d$ is the \emph{Laplacian on functions} and $d d^*$ is the \emph{Laplacian on $1$-forms}, both denoted  by $\Delta$.  

In the following, we will introduce the notion of harmonic $1$-form and present a proposition of ``Hodge'' orthogonal decomposition which states that any $1$-form can be uniquely decomposed into a harmonic $1$-form and an exact $1$-form. More detailed discussions can be found in \cite{BF2011Metric}. 

\begin{definition} \label{D:harmonic}
We say $\omega\in\Omega(G)$ is a \emph{harmonic $1$-form} if  $\Delta\omega=0$. For each $f\in C^0(G,\mbbR)$, we say $df$ is an \emph{exact $1$-form}. 
\end{definition}

\begin{lemma} \label{L:harmonic}
$\omega=\sum_{\ve\in \vE_{\mfrako}(G)}  \omega_\ve\cdot d\ve$ is a harmonic $1$-form if and only if $d^*\omega = 0$ if and only if $ \sum_{\ve\in \vE_{\mfrako}(G),\ve(1)=v}\omega_\ve = \sum _{\ve\in \vE_{\mfrako}(G),\ve(0)=v}\omega_\ve $ for all $v\in V(G)$. 
\end{lemma}

Denote the space of all harmonic $1$-forms by $\mcalH^1(G)$ and the space of all exact $1$-forms by $\imag(d)$. Clearly $\mcalH^1(G)$ and $\imag(d)$ are both linear subspaces of $\Omega(G)$. The following proposition says that they are orthogonal complements of each other. 

\begin{proposition}[Hodge orthogonal decomposition] \label{P:hodge}
There is a canonical Hodge decomposition $\Omega(G)  = \mcalH^1(G) \oplus \imag(d)$. 
That is, any $1$-form $\omega$ can be written uniquely as $\omega= \phi_1(\omega)+\phi_2(\omega)$ where  $\phi_1(\omega)$ is a harmonic $1$-form and $\phi_2(\omega)$ is an exact $1$-form. In addition, $\langle \phi_1(\omega), \phi_2(\omega) \rangle = 0$. 
\end{proposition}

\begin{remark}
Since the first real cohomology group is computed as $H^1(G,\mbbR)=\Omega(G)/ \imag(d)$, we have $H^1(G,\mbbR)\simeq \mcalH^1(G)$. This means that any cohomology class in $H^1(G,\mbbR)$
can be represented canonically by a unique harmonic $1$-form. 
\end{remark}

The first real and integral homology groups are  $H_1(G,\mbbR)=\ker \partial_{\mbbR}$ and $H_1(G,\mbbZ)=\ker \partial_{\mbbZ}$ respectively. In particular,  $H_1(G,\mbbZ)=H_1(G,\mbbR)\bigcap C_1(G,\mbbZ)$. Note that integral $1$-cycles (elements of $H_1(G,\mbbZ)$) are also called  \emph{circulations} in graph theory. 

For a spanning tree $T$ of $G$ and vertices $x,y\in V(G)$, let $T(x,y)$ denote the unique path in $T$ with initial vertex $x$ and terminal vertex $y$.

\begin{lemma} \label{L:GraphHomology}
We have the following properties of $H_1(G,R)$:
\begin{enumerate}[(a)]
\item $H_1(G,R)$ is isomorphic to  $R^g$;
\item $H_1(G,\mbbZ)$ is a full-rank lattice in $H_1(G,\mbbR)$;
\item $H_1(G,\mbbZ/t\mbbZ)$ is canonically isomorphic to $H_1(G,\mbbZ)/tH_1(G,\mbbZ)$. 
\end{enumerate}
\end{lemma}

\begin{remark} \label{R:lattices}
Since $H_1(G,\mbbZ)$ is a full-rank lattice in $H_1(G,\mbbR)$ (or a discrete cocompact subgroup of $H_1(G,\mbbR)$) and the pairing $\mcalH^1(G)\times H_1(G,\mbbR) \to \mbbR$ is perfect, the dual lattice $H_1(G,\mbbZ)^\vee=\Hom(H_1(G,\mbbZ),\mbbZ)$ of $H_1(G,\mbbZ)$ is a full-rank lattice in $\mcalH^1(G)$. Moreover, $i_\mbbZ^1(H_1(G,\mbbZ))$ is a finite-index subgroup of $H_1(G,\mbbZ)^\vee$.
\end{remark}

The proof of Lemma~\ref{L:GraphHomology} uses the standard spanning-tree basis $u_1,\ldots,u_g$ of $H_1(G,\mbbZ)$. The following proposition gives its dual basis in $H_1(G,\mbbZ)^\vee$.

\begin{proposition} \label{P:basis_gen}
Let $G$ be a graph of genus $g$, let $T$ be a spanning tree, and fix an orientation $\mfrako$. If $\ve_1,\ldots,\ve_g\in\vE_{\mfrako}(G)$ are the oriented cotree edges, then
$$
\{\phi_1(d\ve_1),\ldots,\phi_1(d\ve_g)\}
$$
is a basis of $H_1(G,\mbbZ)^\vee$ in $\mcalH^1(G)$. Moreover,
$H_1(G,\mbbZ)^\vee=\phi_1(C^1(G,\mbbZ))$.
\end{proposition}

\subsection{Unitary and full character groups} \label{SS:character-group}
We use two related character spaces.  The first is the compact character group used for Fourier analysis and spectral-radius extremality; the second is its complexification, where the algebraic spectral antisymmetry also holds.

\begin{definition}[Unitary and full complex characters]\label{D:unit-full-character}
The \emph{unitary character group} of $G$ is
$$\mcalX(G):=\Hom(H_1(G,\mbbZ),S^1).$$
The \emph{full complex character torus} of $G$ is
$$\mcalX_\mbbC(G):=\Hom(H_1(G,\mbbZ),\mbbC^\times).$$
Thus $\mcalX(G)$ is the compact subtorus of $\mcalX_\mbbC(G)$ consisting of characters with image in $S^1$.  We write both character groups additively: for
$\xi,\xi'\in\mcalX_\mbbC(G)$ and $\alpha\in H_1(G,\mbbZ)$, set
$$
\bigl(\xi+\xi'\bigr)(\alpha):=\xi(\alpha)\cdot\xi'(\alpha),
\qquad
\bigl(-\xi\bigr)(\alpha):=\xi(\alpha)^{-1}.
$$
On $\mcalX(G)$ this agrees with the additive harmonic-quotient notation used
below.  Although the character groups are written additively, character values
in $\mbbC^\times$ are multiplied.  For a closed walk $C$, we write
$\xi(C):=\xi(C^{\ab})$.
\end{definition}

A unitary character can be associated to each $1$-form $\omega\in\Omega$ by letting $\chi_\omega(\alpha)=e(\omega(\alpha))$ for all $\alpha\in \mcalT$. Here $e(x):=\exp(2\pi\sqrt{-1}x)$ and we will use this conventional notation throughout the paper. Note that $\chi_{-\omega}=\overline{\chi_\omega}$ where $\overline{\chi_\omega}$ is the conjugate character of $\chi_\omega$.  

For a locally compact group $\Gamma$, denote by $\widehat{\Gamma}$ the \emph{Pontryagin dual} of $\Gamma$, which is the group of unitary characters (continuous group homomorphisms from $\Gamma$ to the unit circle $S^1$) on $\Gamma$. 

\begin{lemma} \label{L:PontrDual}
We have the following commutative diagram where the bottom row is exact and $\mcalX(G)=\mcalH^1(G)/H_1(G,\mbbZ)^\vee$:
\[
\begin{tikzcd}[row sep=normal, column sep = small]
&&&\widehat{\mcalT(G)}\arrow[d,two heads]& \\
&&\Omega(G)\arrow[ur,"\sim"{sloped,anchor=south}, "\omega\mapsto \chi_\omega"{sloped,anchor=south, yshift=6pt}]\arrow[d,two heads,"\phi_1" ', pos = 0.4]& \widehat{H_1(G,\mbbR)}\arrow[r,two heads]&\widehat{H_1(G,\mbbZ)} \\
 0 \arrow[r,] &  H_1(G,\mbbZ)^\vee \arrow[r,hook] &\mcalH^1(G) \arrow[r,two heads]\arrow[ur,"\sim"{sloped,anchor=south}] &\mcalX(G)\arrow[ur,"\sim"{sloped,anchor=south}] \arrow[r] &0.
\end{tikzcd}
\]
\end{lemma}

\begin{remark} \label{R:PontrDual}
For the rest of the paper, we will identify $\widehat{\mcalT(G)}$, $\widehat{H_1(G,\mbbR)}$ and  $\widehat{H_1(G,\mbbZ)}$ with $\Omega(G)$, $\mcalH^1(G)$ and the unitary character group $\mcalX(G)$ respectively. In particular, $\mcalX(G)$ is a $g$-dimensional real torus. Provided that $G$ is known, we also simply write $\mcalX(G)$ as $\mcalX$. For $\omega\in \Omega(G)$, we write $\underline{\omega} :=\phi_1(\omega)+H_1(G,\mbbZ)^\vee\in\mcalX$. By Lemma~\ref{L:PontrDual}, when restricted to $H_1(G,\mbbZ)$, we also write $\chi_\omega$ as $\chi_{\underline{\omega}}$ for all $\omega\in \Omega(G)$. 
In comparison to the unitary character group, it is worth mentioning some other related groups. Note that both being lattices in $H_1(G,\mbbR)$, $H_1(G,\mbbZ)$ is subgroup of $(i_\mbbZ^1)^{-1}(H_1(G,\mbbZ)^\vee)$ of finite index (Remark~\ref{R:lattices}). Then  $(i_\mbbZ^1)^{-1}(H_1(G,\mbbZ)^\vee)/H_1(G,\mbbZ)$ is the Jacobian group $\Jac(G)$ of the graph $G$ \cite{BN2007Riemann}, and the $g$-dimensional real torus $H_1(G,\mbbR)/H_1(G,\mbbZ)$ is the (tropical) Jacobian group $\Jac(\Gamma)$ of the metric graph $\Gamma$ obtained from $G$ by identifying edges with the unit interval \cite{BF2011Metric}. 
\end{remark}

\begin{definition}[Switching representatives of full characters]\label{D:full-character}
A \emph{complex gain} on $G$ is a function
$$\gamma:\vE(G)\to\mbbC^\times,\qquad \gamma(\ve^{-1})=\gamma(\ve)^{-1}.$$
Gains and their values are written multiplicatively.  For a walk
$\Delta=\va_1\cdots\va_l$, put
$$\gamma(\Delta):=\prod_{j=1}^l\gamma(\va_j).$$
A gain $\gamma$ represents a full character $\xi\in\mcalX_\mbbC(G)$ if
$$\gamma(C)=\xi(C^{\ab})$$
for every closed walk $C$.  Every full character has such representatives, and two representatives of the same full character differ by switching: there is a map $s:V(G)\to\mbbC^\times$ such that
$$\gamma'(\ve)=s(\ve(0))^{-1}\cdot\gamma(\ve)\cdot s(\ve(1)).$$
For $\omega\in\Omega(G)$, the unitary gain $\gamma_\omega(\ve):=\chi_\omega(\ve)$ represents the unitary character $\underline{\omega}$.
\end{definition}

For elements $x$ and $y$ in a set $X$, let $\delta_{xy}$ be the delta function, i.e., $\delta_{xy}=1$ if $x=y$, and $\delta_{xy}=0$ otherwise. 

\begin{proposition} \label{P:CharacterOrth}
Fixing $\alpha,\beta\in H_1(G,\mbbZ)$, we have the following orthogonal relation
$$\frac{1}{\vol(\mcalX)}\int_{\mcalX}\chi_\omega(\alpha)\overline{\chi_\omega(\beta)}dV_\omega = \delta_{\alpha\beta}$$
where $\chi_\omega(\alpha)\overline{\chi_\omega(\beta)}$ is considered as a function on $\mcalX$ with variable $\omega$, $dV_\omega$ is the volume $g$-form on $\mcalX$  and  $\vol(\mcalX)$ is the total  volume of $\mcalX$. 
\end{proposition}

Now let  us consider  a subgroup $\Lambda$ of $H_1(G,\mbbZ)$ of finite index. Then $\Lambda$ is also a rank-$g$ lattice in $H_1(G,\mbbR)$, and the corresponding dual lattice $\Lambda^\vee$  in $\mcalH^1(G)$ contains $H_1(G,\mbbZ)^\vee$ as a sub-lattice. Let $Q_\Lambda:=H_1(G,\mbbZ)/\Lambda$ and $\mcalX_\Lambda:=\mcalH^1(G)/\Lambda^\vee$. For $\alpha\in H_1(G,\mbbZ)$ and $\omega\in \Omega(G)$, we write $\underline{\alpha}=\alpha+\Lambda\in Q_\Lambda$ and $\dunderline{\omega} :=\phi_1(\omega)+\Lambda^\vee\in\mcalX_\Lambda$.

\begin{lemma}  \label{L:PontrDualLambda}
There exists a commutative diagram with the rows being exact:
\[
\begin{tikzcd}[row sep=large, column sep = small]
0 \arrow[r] & \widehat{Q_\Lambda} \arrow[r,hook] &\widehat{H_1(G,\mbbZ)} \arrow[r,two heads] &\widehat{\Lambda} \arrow[r] &0 \\
0 \arrow[r] & \Lambda^\vee/H_1(G,\mbbZ)^\vee \arrow[r,hook] \arrow[u,"\sim"{sloped,anchor=south}]&\mcalX \arrow[r, two heads, "\underline{\omega}\mapsto \dunderline{\omega}" '] \arrow[u,"\sim"{sloped,anchor=south}, "\underline{\omega}\mapsto\chi_{\underline{\omega}}" '{sloped, anchor = north}] & \mcalX_\Lambda \arrow[r] \arrow[u,"\sim"{sloped,anchor=south}] & 0
\end{tikzcd}
\]
\end{lemma}

\begin{remark}
As in Remark~\ref{R:PontrDual}, for the rest of the paper, we will identify $\widehat{Q_\Lambda}$ and $\widehat{\Lambda}$ with $\Lambda^\vee/H_1(G,\mbbZ)^\vee$ and $\mcalX_\Lambda$, respectively. Note that $\widehat{Q_\Lambda}$ is a finite abelian group noncanonically isomorphic to $Q_\Lambda$. In addition, by Lemma~\ref{L:PontrDualLambda}, when restricted to $\Lambda$, we also write $\chi_\omega$ as $\chi_{\dunderline{\omega}}$ for all $\omega\in \Omega(G)$. 
\end{remark}

The following proposition is a  generalization of Proposition~\ref{P:CharacterOrth}.

\begin{proposition}\phantomsection\label{P:CharacterOrthLambda}

\begin{enumerate}[(a)]
 \item Fixing $\alpha,\beta\in \Lambda$, we have the following orthogonal relation
$$\frac{1}{\vol(\mcalX_\Lambda)}\int_{\mcalX_\Lambda}\chi_\omega(\alpha)\overline{\chi_\omega(\beta)}dV_\omega = \delta_{\alpha\beta}$$
where $\chi_\omega(\alpha)\overline{\chi_\omega(\beta)}$ is considered as a function on $\mcalX_\Lambda$ with variable $\omega$, $dV_\omega$ is the volume $g$-form on $\mcalX_\Lambda$  and  $\vol(\mcalX_\Lambda)$ is the total  volume of $\mcalX_\Lambda$. 
 
 \item Fixing $\underline{\alpha},\underline{\beta}\in Q_\Lambda$, we have the following orthogonal relation
$$\frac{1}{|Q_\Lambda|}\sum_{\underline{\omega}\in \widehat{Q_\Lambda}}\chi_{\underline{\omega}}(\alpha)\overline{\chi_{\underline{\omega}}(\beta)} = \delta_{\underline{\alpha}\underline{\beta}}.
$$
\end{enumerate}
\end{proposition}

\begin{remark}
Lemma~\ref{L:PontrDualLambda} gives
$|Q_\Lambda|\cdot\vol(\mcalX_\Lambda)=\vol(\mcalX)$.  The volume
$\vol(\mcalX)$ also has a spanning-tree interpretation.  Indeed,
$\mcalX=\mcalH^1(G)/H_1(G,\mbbZ)^\vee$, while the Jacobian of the associated
unit-length metric graph is
$\Jac(\Gamma)=H_1(G,\mbbR)/H_1(G,\mbbZ)$.  Hence
$\vol(\mcalX)\cdot\vol(\Jac(\Gamma))=1$.  Moreover,
$\vol(\Jac(\Gamma))=\vol(\mcalX)\cdot|\Jac(G)|$, where $\Jac(G)$ is the finite
graph Jacobian, or critical group.  By Kirchhoff's matrix-tree theorem,
$|\Jac(G)|$ is the number of spanning trees of $G$; the geometric refinement
of An--Baker--Kuperberg--Shokrieh gives a canonical break-divisor
representative for every class in $\Jac(G)$~\cite[Theorem~1.5 and
Corollary~1.6]{ABKS2014canonical}.  These observations give the following
proposition.
\end{remark}

\begin{proposition}
 Let $w(G)$ be the complexity of $G$, i.e., the number of spanning trees of $G$. Then $\vol(\mcalX) = 1/\sqrt{w(G)}$ and  $\vol(\mcalX_\Lambda) = 1/(\sqrt{w(G)}\cdot|Q_\Lambda|)$. 
\end{proposition}

\subsection{Twisted weighted adjacency matrices}\label{SS:weighted-adjacency}
The vertex and edge adjacency matrices can be twisted by either a unitary character $\chi_\omega$ or, more generally, by a full complex character represented by a complex gain.  We state the weighted version first, since this is the setting of the main theorem.

A \emph{positive directed weight} on $G$ is a function
$$W:\vE(G)\to\mbbR_{>0}.$$
For a walk $\Delta=\va_1\cdots\va_l$, put
$$W(\Delta):=\prod_{j=1}^l W(\va_j).$$
We use the following terminology for special cases:
\begin{enumerate}[(i)]
\item $W$ is the \emph{trivial weight}, written $W\equiv 1$, if $W(\ve)=1$ for every $\ve\in\vE(G)$;
\item $W$ is \emph{symmetric} if $W(\ve)=W(\ve^{-1})$ for every $\ve\in\vE(G)$;
\item $W$ is \emph{reciprocal} if $W(\ve^{-1})=W(\ve)^{-1}$ for every $\ve\in\vE(G)$.
\end{enumerate}
The trivial weight is both symmetric and reciprocal.

\begin{definition}\label{D:weighted-adj}
Let $W$ be a positive directed weight on $G$.
For a complex gain $\gamma:\vE(G)\to\mbbC^\times$, we define the following weighted twisted matrices:
\begin{enumerate}[(i)]
\item The \emph{twisted weighted vertex adjacency matrix} $A_{\gamma,W}$ is the $n\times n$ matrix
$$\left(A_{\gamma,W}\right)_{vw}=\sum_{\substack{\ve\in \vE(G)\\ \ve(0)=v,\ \ve(1)=w}}W(\ve)\gamma(\ve).$$
\item The \emph{twisted weighted edge adjacency matrix} $B_{\gamma,W}$ is the $2m\times 2m$ matrix
$$\left(B_{\gamma,W}\right)_{\va\vb}=
\begin{cases}
W(\vb)\gamma(\vb),&\text{if }\va\shortrightarrow\vb,\\
0,&\text{otherwise.}
\end{cases}$$
\end{enumerate}
If $\xi\in\mcalX_\mbbC(G)$ and $\gamma$ represents $\xi$, define
$$\spec A_{\xi,W}:=\spec A_{\gamma,W},\qquad
\spec B_{\xi,W}:=\spec B_{\gamma,W}.$$
These spectra are independent of the representative by
Lemma~\ref{L:switching-similarity} below.  For a unitary character $\underline{\omega}\in\mcalX(G)$, represented by $\gamma_\omega(\ve)=\chi_\omega(\ve)$, we retain the notation
$$A_{\omega,W}:=A_{\gamma_\omega,W},\qquad
B_{\omega,W}:=B_{\gamma_\omega,W}.$$
For the trivial character, write
$$
A_W:=A_{0,W},\qquad B_W:=B_{0,W}.
$$
At the spectral level we may equivalently write
$$
\spec A_{\underline{\omega},W}:=\spec A_{\omega,W},
\qquad
\spec B_{\underline{\omega},W}:=\spec B_{\omega,W},
$$
since these multisets depend only on the additive class
$\underline{\omega}$.
When $W\equiv 1$, we write $A_\omega:=A_{\omega,1}$ and call it the \emph{twisted vertex adjacency matrix}; similarly, we write $B_\omega:=B_{\omega,1}$ and call it the \emph{twisted edge adjacency matrix}. Explicitly,
$$\left(A_\omega\right)_{vw}=\sum_{\substack{\ve\in \vE(G)\\ \ve(0)=v,\ \ve(1)=w}}\chi_\omega(\ve),$$
and
$$\left(B_\omega\right)_{\va\vb}=
\begin{cases}
\chi_\omega(\vb),&\text{if }\va\shortrightarrow\vb,\\
0,&\text{otherwise.}
\end{cases}$$
\end{definition}

\begin{lemma}\label{L:switching-similarity}
Let $W$ be a positive directed weight.  If two complex gains $\gamma$ and
$\gamma'$ are switching equivalent, then $A_{\gamma,W}$ is similar to
$A_{\gamma',W}$ and $B_{\gamma,W}$ is similar to $B_{\gamma',W}$.  If the
switching function is unitary, these similarities are unitary.  In particular,
unitary gains representing the same unitary character yield unitarily similar
vertex and edge matrices.
\end{lemma}

The powers of the weighted edge adjacency matrices are naturally indexed by
the following \emph{edge-walks}.  An edge-walk $\vDelta$ of length $N$ is a
sequence of oriented edges
$$
\va_0\blacktriangleright\va_1\blacktriangleright\cdots
\blacktriangleright\va_N
$$
such that $\va_{i+1}(0)=\va_i(1)$ for $0\leq i<N$.  Its
\emph{initial oriented edge} is $\vDelta(0):=\va_0$, and its
\emph{terminal oriented edge} is $\vDelta(1):=\va_N$.  The walk
$\va_1\cdots\va_N$ is the \emph{associated walk} of $\vDelta$, while
$\va_1\cdots\va_{N-1}$ is its \emph{interior walk}.  The inverse edge-walk is
$$
\vDelta^{-1}:=\va_N^{-1}\blacktriangleright\cdots
\blacktriangleright\va_0^{-1}.
$$
An edge-walk
$\vDelta'=\vb_0\blacktriangleright\cdots\blacktriangleright\vb_M$ is a
\emph{sub-edge-walk} of $\vDelta$ if
$\vb_j=\va_{i+j}$ for $0\leq j\leq M$ and some
$i\in\{0,\ldots,N-M\}$.  The edge-walk $\vDelta$ has a
\emph{backtrack} if $\va_{i+1}=\va_i^{-1}$ for some $0\leq i<N$, and it is
\emph{closed} at the \emph{base oriented edge} $\ve$ if
$\ve=\va_0=\va_N$.  The associated walk of a closed edge-walk is closed; if
the edge-walk has no backtrack, its associated walk is a circuit.

If $\Delta$ is the associated walk of $\vDelta$, write
$$
\chi_\omega(\vDelta):=\chi_\omega(\Delta)
:=\chi_\omega(\Delta^{\ab}),\qquad
W(\vDelta):=W(\Delta).
$$
Then
$\chi_\omega(\Delta^{-1})=\chi_{-\omega}(\Delta)
=\overline{\chi_\omega(\Delta)}$.

\begin{remark}
When $\omega=0$, $A_\omega$ and $B_{\omega}$ specialize to the conventional (vertex) adjacency matrix $A$ and edge adjacency matrix $B$ respectively.\footnote{The matrix $B$ is more commonly called the nonbacktracking matrix/operator, or the Hashimoto matrix; see Hashimoto~\cite{H1989Zeta} and Terras~\cite{T2011Zeta}.} In this case, if $G$ has multiple edges or loops, then for distinct vertices $v$ and $w$, $A_{vw}$ is exactly the number of edges between $v$ and $w$, and $A_{vv}$ is twice the number of loops at $v$ (corresponding to the two orientations for each loop). 
\end{remark}

\begin{remark}\label{R:weighted-nonhermitian}
A positive directed weight is not required to be symmetric.  The trivial specialization $W\equiv 1$ recovers the unweighted twisted matrices $A_\omega$ and $B_\omega$.  If $W$ is symmetric, then $A_{\omega,W}$ is Hermitian for every $\omega$.  For a general directed weight, $A_{\omega,W}$ is typically non-Hermitian, while the nonbacktracking edge matrix $B_{\omega,W}$ is generally non-Hermitian even in the trivial or symmetric cases. Since $W$ is real, however,
$$A_{-\omega,W}=\overline{A_{\omega,W}},\qquad B_{-\omega,W}=\overline{B_{\omega,W}}.$$
Thus the spectra of the $-\omega$ matrices are the complex conjugates of the spectra of the $\omega$ matrices.
\end{remark}

\begin{remark}
If $W$ is reciprocal, then
$$\phi_{\omega,W}(\ve):=W(\ve)\chi_\omega(\ve)$$
defines a complex gain on $G$, in the sense that $\phi_{\omega,W}(\ve^{-1})=\phi_{\omega,W}(\ve)^{-1}$.  Thus $A_{\omega,W}$ and $B_{\omega,W}$ can be viewed as the vertex and nonbacktracking edge adjacency matrices attached to this complex gain graph.  The unweighted case $W\equiv1$ is the special case of a complex unit gain graph, with gain $\chi_\omega(\ve)$ on each oriented edge $\ve$~\cite{Z1989Biased,R2012Spectral}.  Equivalently, a complex gain graph may be interpreted as a discrete complex line bundle with connection over $G$: after choosing a one-dimensional fiber at each vertex, the gain gives the parallel transport along an oriented edge, and switching gains corresponds to changing local trivializations; see Kenyon~\cite{K2011Spanning}, and also Section~3.1 and Definition~3.1 of~\cite{BPT2015Matrix}. 
\end{remark}

By convention, the \emph{spectrum} of a square matrix $M$, denoted by $\spec M$, is the multiset of eigenvalues of $M$ counted with multiplicity. For two square matrices $M$ and $N$, we say that $M$ is \emph{isospectral} to $N$ if $\spec M=\spec N$. We say that \emph{the spectrum of $M$ is the complex conjugate (resp.\ negation) of the spectrum of $N$}, written as $\spec M=\overline{\spec N}$ (resp.\ $\spec M=-\spec N$), if each $\lambda\in\spec M$ of multiplicity $m$ corresponds to $\overline{\lambda}\in\spec N$ of multiplicity $m$ (resp.\ to $-\lambda\in\spec N$ of multiplicity $m$). We say that $\spec M$ is \emph{invariant under complex conjugation (resp.\ negation)} if $\spec M=\overline{\spec M}$ (resp.\ $\spec M=-\spec M$).

Some basic conjugation properties of the weighted twisted matrices are stated in the following lemma.

\begin{lemma}\phantomsection\label{L:PropAdj}
Let $W:\vE(G)\to\mbbR_{>0}$ be a positive directed weight.  For every $\omega\in\Omega(G)$:
\begin{enumerate}[(a)]
\item $A_{-\omega,W}=\overline{A_{\omega,W}}$ and $B_{-\omega,W}=\overline{B_{\omega,W}}$.  Consequently
$$\spec A_{-\omega,W}=\overline{\spec A_{\omega,W}},
\qquad
\spec B_{-\omega,W}=\overline{\spec B_{\omega,W}}.$$
\item If $W$ is symmetric, then $A_{\omega,W}$ is Hermitian and the eigenvalues of $A_{\omega,W}$ are all real.  In particular,
$$\overline{A_{\omega,W}}=A_{-\omega,W}=A_{\omega,W}^T,$$
and $A_{\omega,W}$ and $A_{-\omega,W}$ are similar matrices.
\item If $W$ is symmetric, then $B_{\omega,W}$ and $B_{-\omega,W}$ are similar matrices.  Hence $\spec B_{\omega,W}$ is invariant under complex conjugation.
\end{enumerate}
In particular, taking $W\equiv1$ recovers the corresponding unweighted statements for $A_\omega$ and $B_\omega$.
\end{lemma}

\subsection{Periodic Floquet spectra and spectral-edge terminology}\label{SS:floquet-edges}
We record the periodic-graph language needed in Subsection~\ref{SS:outer-spectral-edges}.  This is deliberately stated for a general abelian periodic graph, since the maximal abelian cover of a finite graph is only one special case.

Let $\Gamma$ be a connected graph with a free co-compact action of a lattice $\Lambda\simeq\mbbZ^d$, and let $G=\Gamma/\Lambda$ be the finite quotient graph.  After choosing lifts of the vertices of $G$, each oriented edge $\ve\in\vE(G)$ has a displacement $\nu(\ve)\in\Lambda$.  This gives a homomorphism, still denoted by
$$\nu:H_1(G,\mbbZ)\longrightarrow \Lambda,$$
by summing displacements along closed walks.  If $\widetilde W:\vE(G)\to\mbbR$ is a real symmetric periodic edge weight, meaning $\widetilde W(\ve)=\widetilde W(\ve^{-1})$ and no positivity is assumed, the Bloch fiber over a unitary character $\zeta\in\widehat{\Lambda}:=\Hom(\Lambda,U(1))$ is the Hermitian matrix
$$\left(H_\zeta\right)_{vw}
=\sum_{\substack{\ve\in\vE(G)\\ \ve(0)=v,\ \ve(1)=w}}
\widetilde W(\ve)\zeta(\nu(\ve)).
$$
Since loops are allowed in our graph convention, a real periodic vertex potential can be encoded by zero-displacement loops and need not be treated as a separate term here.

\begin{definition}[Bands and spectral edges]\label{D:floquet-edges}
Let
$$\lambda_1(\zeta)\leq\cdots\leq\lambda_n(\zeta),\qquad \zeta\in\widehat{\Lambda},$$
be the ordered eigenvalues of the Hermitian fibers $H_\zeta$.  These functions are called the \emph{band functions}.  The \emph{Floquet spectrum} is
$$\Sigma(H):=\bigcup_{\zeta\in\widehat{\Lambda}}\spec H_\zeta
=\bigcup_{j=1}^n \lambda_j(\widehat{\Lambda}).$$
The endpoints of the intervals $\lambda_j(\widehat{\Lambda})$ are called \emph{band edges}.  A \emph{spectral edge} is an endpoint of the spectrum $\Sigma(H)$, equivalently an endpoint of a spectral gap or one of the two outer endpoints
$$\sup\Sigma(H),\qquad \inf\Sigma(H).$$
\end{definition}

\begin{conjecture}[Spectral edges conjecture, periodic graph form]\label{Conj:spectral-edges}
For a fixed abelian periodic graph, a generic self-adjoint periodic graph operator has the following property: every spectral edge is attained by a single band, the set of Bloch characters at which it is attained is finite, and at each such character the corresponding band is locally simple and has a non-degenerate extremum.
\end{conjecture}

\begin{remark}
The word ``generic'' depends on the chosen parameter space, for instance real symmetric periodic edge weights, with vertex potentials included through loops when desired.  The conjecture is usually understood as a local statement about the dispersion relation near spectral edges: isolation, non-degeneracy of the Hessian, and single-band attainment.  This formulation is in the spirit of Kuchment's periodic-operator perspective~\cite{K2016Overview}; for graph-theoretic developments, see for example Berkolaiko--Canzani--Cox--Marzuola~\cite{BCCM2022Local} and Faust--Sottile~\cite{FS2025Corners}.
\end{remark}

\subsection{Weighted L-functions and determinant formulas}\label{SS:CompLfunctions}
For a walk or cycle $P$, write $\chi_\omega(P):=\chi_\omega(P^{\ab})$.  Let $W$ be a positive directed weight.

\begin{definition}[Weighted L-function and weighted Ihara zeta function]
The \emph{weighted L-function} associated with $W$ and $\chi_\omega$ is
$$L_W(u,\chi_\omega):=\prod_{[P]\in\bmcalP}\left(1-W(P)\chi_\omega(P)u^{\tau(P)}\right)^{-1}.$$
Its untwisted specialization
$$z_W(u):=L_W(u,\chi_0)=\prod_{[P]\in\bmcalP}\left(1-W(P)u^{\tau(P)}\right)^{-1}$$
is the \emph{weighted Ihara zeta function}.  When $W\equiv1$, we retain the notation
$$L(u,\chi_\omega):=L_1(u,\chi_\omega),\qquad z(u):=z_1(u).$$
\end{definition}

\begin{lemma} \label{L:LFuncExp}
As formal power series,
$$L_W(u,\chi_\omega)=\exp\left(\sum_{[C]\in\bmcalC}\frac{W(C)\chi_\omega(C)}{r(C)}u^{\tau(C)}\right),$$
and in particular
$$z_W(u)=\exp\left(\sum_{[C]\in\bmcalC}\frac{W(C)}{r(C)}u^{\tau(C)}\right).$$
\end{lemma}

\begin{lemma} \label{L:L-symmetry}
If $W$ is symmetric, then $L_W(u,\chi_\omega)=L_W(u,\chi_{-\omega})$.
\end{lemma}

\begin{proposition}[Weighted character detection] \label{P:LFuncIden}
Let $W$ be a positive directed weight and let $\omega,\omega'\in\Omega(G)$.
\begin{enumerate}[(a)]
\item $L_W(u,\chi_\omega)=z_W(u)$ if and only if $\underline{\omega}=0$, equivalently,
$$\omega\in\ker(\Omega(G)\twoheadrightarrow\mcalX(G))
=\imag(d_\mbbR)\oplus H_1(G,\mbbZ)^\vee,$$
or equivalently $\phi_1(\omega)\in H_1(G,\mbbZ)^\vee$.
\item If $\underline{\omega}=\underline{\omega'}$, then
$$L_W(u,\chi_\omega)=L_W(u,\chi_{\omega'}).$$
\end{enumerate}
\end{proposition}

\begin{remark}
The positivity of $W$ is essential in Proposition~\ref{P:LFuncIden}(a): it prevents cancellation when the coefficients of $\log L_W$ and $\log z_W$ are compared.  This character-detection property will be used in Lemma~\ref{L:weighted-spectral-rigidity}.
\end{remark}

The determinant formulas have a long history.  Ihara introduced the zeta function in the setting of discrete subgroups of $p$-adic groups~\cite{I1966Padic}.  Hashimoto obtained its nonbacktracking edge-matrix formulation~\cite{H1989Zeta}, and Bass established the vertex-matrix determinant formula for arbitrary finite graphs~\cite{B1992Ihara}.  Stark and Terras subsequently introduced the multivariate edge zeta function, assigning independent variables to oriented edges~\cite{ST1996Zeta}; its edge determinant specializes to the weighted twisted identity below by assigning the variable $uW(\ve)\chi_\omega(\ve)$ to $\ve$.  Weighted zeta and $L$-functions, together with determinant expressions, were further developed by Mizuno and Sato~\cite{MS2004Weighted}.  Since conventions for placing an oriented-edge weight in the transition matrix vary, we record the precise vertex reduction compatible with Definition~\ref{D:weighted-adj}.

For an oriented edge $\ve$, put
$$r_W(\ve):=W(\ve)\cdot W(\ve^{-1}).$$
Define the following $n\times n$ matrices of rational functions:
\begin{align*}
\bigl(\mathsf A_{\omega,W}(u)\bigr)_{vw}
&:=\sum_{\substack{\ve\in\vE(G)\\ \ve(0)=v,\ \ve(1)=w}}
\frac{uW(\ve)\chi_\omega(\ve)}{1-u^2r_W(\ve)},\\
\bigl(\mathsf D_W(u)\bigr)_{vw}
&:=\begin{cases}
\displaystyle\sum_{\substack{\ve\in\vE(G)\\ \ve(0)=v}}
\frac{u^2r_W(\ve)}{1-u^2r_W(\ve)},&v=w,\\
0,&v\neq w.
\end{cases}
\end{align*}
Also, let $Q$ be the diagonal matrix with entries $Q_{vv}=\deg(v)-1$.

\begin{theorem}[Weighted Hashimoto--Bass determinant formulas]\phantomsection\label{T:DertminantL}
\begin{enumerate}[(a)]
\item Let $p_{\omega,W}$ be the characteristic polynomial of $B_{\omega,W}$. Then
$$L_W(u,\chi_\omega)^{-1}=\det(I-uB_{\omega,W})=u^{2m}p_{\omega,W}(u^{-1}).$$
\item Fix an orientation $\mfrako$ of $G$.  For every positive directed weight $W$,
\begin{align*}
L_W(u,\chi_\omega)^{-1}
&=\prod_{\ve\in\vE_{\mfrako}(G)}\bigl(1-u^2r_W(\ve)\bigr)
\det\bigl(I-\mathsf A_{\omega,W}(u)+\mathsf D_W(u)\bigr).
\end{align*}
The right-hand side is independent of $\mfrako$, and its apparent denominators cancel.  In particular, if $W$ is reciprocal, then
$$L_W(u,\chi_\omega)^{-1}
=(1-u^2)^{g-1}\det\bigl(I-uA_{\omega,W}+u^2Q\bigr).$$
Thus the classical twisted Ihara--Bass formula is recovered by taking $W\equiv1$.
\end{enumerate}
\end{theorem}

\begin{remark}
Symmetry and reciprocity play different roles here.  If $W$ is symmetric, then $A_{\omega,W}$ is Hermitian, but $r_W(\ve)=W(\ve)^2$ generally depends on the edge.  Consequently the factors $1-u^2W(\ve)^2$ and the corresponding denominators in $\mathsf A_{\omega,W}(u)$ and $\mathsf D_W(u)$ cannot in general be replaced by the single classical factor $(1-u^2)^{g-1}$ and the matrix $Q$.  By contrast, reciprocity gives $r_W(\ve)=1$ on every edge and yields the clean three-term determinant above.  A positive weight that is both symmetric and reciprocal is necessarily trivial.
\end{remark}

\section{Main results} \label{S:main}
Throughout this section, $G$ is a connected finite graph with $n$ vertices,
$m$ unoriented edges, and genus $g=m-n+1$; multiple edges and loops are
allowed.
With the preliminary notation in place, we now introduce the canonical character and state the main results (Theorems~\ref{T:SpecDuality},~\ref{T:outer-spectral-edges},~\ref{T:trace-formula}, and~\ref{T:counting}). Recall that $\Omega(G)$ is the space of $1$-forms, $\mcalX(G)$ is the unitary character group, $\mcalX_\mbbC(G)$ is the full complex character torus, and $\underline{\omega}\in\mcalX(G)$ denotes the character class of a $1$-form $\omega$. The twisted weighted matrices are defined in Definition~\ref{D:weighted-adj}; for a full character $\xi\in\mcalX_\mbbC(G)$, the notation $\spec A_{\xi,W}$ and $\spec B_{\xi,W}$ means the spectrum of any switching representative.

\subsection{Canonical and period characters: spectral antisymmetry and extremality}
For a square matrix $M$, its \emph{spectral radius} is
$\rho(M):=\max\{|\lambda|:\lambda\in\spec M\}$.

\begin{lemma}\label{L:unique}
For an orientation $\mfrako$ of $G$, define
$$\omega_{\mfrako}:=\frac{1}{2}\sum_{\ve\in\vE_{\mfrako}(G)}d\ve.$$
Then the character
$$\theta:=\underline{\omega_{\mfrako}}\in\mcalX(G)$$
is independent of $\mfrako$.  It is the unique full complex character satisfying
$$\theta(C)=(-1)^{\tau(C)}$$
for every closed walk $C$.
\end{lemma}

\begin{proof}
By definition, $\chi_{\omega_{\mfrako}}(\ve)=-1$ for every oriented edge $\ve\in\vE(G)$, so $\theta(C)=(-1)^{\tau(C)}$ for every closed walk $C$.  Conversely, fix a spanning tree $T$ and let $\ve_1,\ldots,\ve_g$ be the positively oriented cotree edges.  The spanning-tree cycles
$$u_i=\ve_i+T(\ve_i(1),\ve_i(0))^{\ab},\qquad i=1,\ldots,g,$$
form a basis of $H_1(G,\mbbZ)$, as in the proof of Proposition~\ref{P:basis_gen}.  If a full complex character $\xi\in\mcalX_\mbbC(G)$ has the stated property, then its values on these basis cycles agree with the values of $\theta$.  Hence $\xi=\theta$ in $\mcalX_\mbbC(G)$.  This proves uniqueness, and therefore the class $\underline{\omega_{\mfrako}}$ is independent of the orientation $\mfrako$.
\end{proof}

\begin{definition}[Canonical character and canonical $1$-forms]\label{D:canonical-character}
The character $\theta$ in Lemma~\ref{L:unique} is called the \emph{canonical character} of $G$.  A $1$-form $\omega\in\Omega(G)$ is called a \emph{canonical $1$-form} of $G$ if
$$\underline{\omega}=\theta.$$
In particular, $\omega_{\mfrako}$ is a canonical $1$-form for every orientation $\mfrako$, but canonical $1$-forms are not unique.
\end{definition}

\begin{construction}[An alternative cotree construction of canonical $1$-forms]\label{Cs:CanoForm2}
Let $T$ be a spanning tree of $G$, and fix an orientation $\mfrako$ of $G$.  Since $T$ is a tree, it is bipartite: choose the bipartition
$$V(G)=V_1\sqcup V_2$$
such that every edge of $T$ has one endpoint in $V_1$ and the other in $V_2$.  Because $T$ is connected, this bipartition is unique up to interchanging $V_1$ and $V_2$, which does not affect the construction below.  For each positively oriented cotree edge $\ve\in\vE_{\mfrako}(G)\setminus T$, set
$$
\kappa(\ve)=
\begin{cases}
0,&\text{if $\ve$ joins $V_1$ and $V_2$,}\\
1/2,&\text{otherwise.}
\end{cases}
$$
Equivalently, $\kappa(\ve)=0$ or $1/2$ according as the unique tree path
$T(\ve(1),\ve(0))$ has odd or even length.
Define
$$\omega_{(T,\mfrako)}:=\sum_{\ve\in\vE_{\mfrako}(G)\setminus T}\kappa(\ve)d\ve.$$
This construction is topological; it contains no reference to the weight $W$.
Theorem~\ref{T:SpecDuality}(SA2) shows that
$\underline{\omega_{(T,\mfrako)}}=\theta$, so $\omega_{(T,\mfrako)}$ is a
canonical $1$-form.
\end{construction}

\begin{remark}\label{R:X_TO}
Let $\mcalX_{(T,\mfrako)}:=\{\sum_{i=1}^g\omega_i d\ve_i:0\leq\omega_i<1\}$, where $\{\ve_1,\cdots,\ve_g\}=\vE_{\mfrako}(G)\setminus T$.  By Proposition~\ref{P:basis_gen}, the quotient map $\omega\mapsto\underline{\omega}$ restricts to a bijection of sets
$$\mcalX_{(T,\mfrako)}\longrightarrow \mcalX.$$
Thus $\mcalX_{(T,\mfrako)}$ is a set-theoretic fundamental domain, not a canonical group-theoretic identification.  In particular, every character has a unique representative in $\mcalX_{(T,\mfrako)}$.
In these coordinates, the $2$-torsion characters are precisely the $2^g$
points whose cotree coordinates all belong to $\{0,1/2\}$.
\end{remark}

For a graph of genus $g\geq2$, put
$$
\nu_G:=\gcd\{\tau(P):[P]\in\bmcalP\}.
$$
We call $\nu_G$ the \emph{nonbacktracking period}, or simply the \emph{period},
of $G$.
Equivalently, by the classical period theorem for irreducible nonnegative
matrices, $\nu_G$ is the Perron--Frobenius period of the irreducible edge
adjacency matrix on the nonbacktracking core; see
\cite[Theorem~1.3]{S1981Nonnegative}.
Thus $\nu_G$ depends only on the graph and not on the positive directed
weight.

The spectral-radius extremality problem for twisted edge adjacency on the
unitary character torus exhibits a subtle bipartite alternative; see
Theorem~\ref{T:SpecDuality}(SR4)--(SR5).  In the non-bipartite case, the
maximizers are the trivial and canonical characters.  In the bipartite case,
however, the canonical character is trivial, and a further subtlety arises:
the edge dynamics may admit either no additional maximizer or exactly one,
governed by a different parity obstruction.  The following definition
isolates the character associated with that obstruction, which
Theorem~\ref{T:SpecDuality}(SR5) identifies, when it exists, as the unique
possible additional maximizer of the weighted edge spectral radius in the
bipartite case.

\begin{definition}[Period character]\label{D:period-character}
Let $G$ have genus $g\geq2$, and let $\nu_G$ denote its period.  The assignment
$$
P^{\ab}\longmapsto \frac{\tau(P)}{\nu_G}\pmod 2,
\qquad [P]\in\bmcalP,
$$
need not be compatible with the integral homology relations among prime
cycles.  If it descends to a homomorphism
$$
\varepsilon_G:H_1(G,\mbbZ)\longrightarrow\mbbZ/2\mbbZ,
$$
we define
$$
\eta(\alpha):=(-1)^{\varepsilon_G(\alpha)},
\qquad \alpha\in H_1(G,\mbbZ),
$$
and call $\eta\in\mcalX(G)$ the \emph{period character} of $G$.
Whenever $\eta$ exists, a $1$-form $\omega\in\Omega(G)$ satisfying
$\underline{\omega}=\eta$ is called a \emph{period $1$-form}.
\end{definition}

With the canonical character and the possible period character now defined,
we can state the main theorem.

\begin{theorem}[\textbf{Weighted spectral antisymmetry and unitary spectral-radius extremality}]\label{T:SpecDuality}
Let $G$ be a connected graph of genus $g$, let $W:\vE(G)\to\mbbR_{>0}$ be a positive directed weight, let $\mcalX=\mcalX(G)$ be the unitary character group of $G$, and let $\mcalX_\mbbC=\mcalX_\mbbC(G)$ be the full complex character torus. Let $\theta\in\mcalX\subset\mcalX_\mbbC$ be the canonical character of Definition~\ref{D:canonical-character}. Then $\theta$ depends only on $G$, and the following statements hold.

\medskip
\noindent\textbf{Canonical character and spectral antisymmetry.}
\begin{enumerate}[\textup{(SA}1\textup{)}]
\item For every $\xi\in\mcalX_\mbbC$,
$$\spec A_{\theta+\xi,W}=-\spec A_{\xi,W},\qquad
\spec B_{\theta+\xi,W}=-\spec B_{\xi,W}.$$
If $W$ is symmetric, then in addition, for every $\omega\in\Omega(G)$,
$$
\spec A_{\theta-\underline{\omega},W}=-\spec A_{\underline{\omega},W},
\qquad
\spec B_{\theta-\underline{\omega},W}=-\spec B_{\underline{\omega},W}.
$$
\item The character $\theta$ is the unique full complex character satisfying
$$\theta(C)=(-1)^{\tau(C)}$$
for every closed walk $C$.  For every orientation $\mfrako$ and every spanning tree $T$,
$$\theta=\underline{\omega_{\mfrako}}=\underline{\omega_{(T,\mfrako)}},$$
so the spanning-tree form $\omega_{(T,\mfrako)}$ of Construction~\ref{Cs:CanoForm2} is also canonical.  In particular, $\theta$ is a $2$-torsion character of $\mcalX$.
\item $\theta=0$ if and only if $G$ is bipartite. If $G$ is bipartite, then
for every full complex character $\xi\in\mcalX_\mbbC$, the spectra of
$A_{\xi,W}$ and $B_{\xi,W}$ are invariant under negation.
\item For every positive directed weight $W$ and every $\omega\in\Omega(G)$, the trivial and canonical characters are spectrally rigid on the unitary torus:
\begin{align*}
\underline{\omega}=0
&\Longleftrightarrow \spec A_{\omega,W}=\spec A_{0,W}
\Longleftrightarrow \spec B_{\omega,W}=\spec B_{0,W},\\
\underline{\omega}=\theta
&\Longleftrightarrow \spec A_{\omega,W}=-\spec A_{0,W}
\Longleftrightarrow \spec B_{\omega,W}=-\spec B_{0,W}.
\end{align*}
\end{enumerate}

\medskip
\noindent\textbf{Period character and spectral rotation.}
\begin{enumerate}[\textup{(PC}1\textup{)}]
\item Suppose $g\geq2$.  Whenever the period character $\eta$ exists, it is
unique and is a nontrivial $2$-torsion character of $\mcalX$.  If
$G$ is non-bipartite, then $\eta$ always exists and $\eta=\theta$; if $G$ is
bipartite, it may or may not exist.
\item If the period character $\eta$ exists, then for every positive directed
weight $W$ and every full complex character $\xi\in\mcalX_\mbbC$,
$$
\spec B_{\eta+\xi,W}
=e\left(\frac{1}{2\nu_G}\right)\spec B_{\xi,W}.
$$
\end{enumerate}

\medskip
\noindent\textbf{Spectral-radius extremality on the unitary character torus.}
\begin{enumerate}[\textup{(SR}1\textup{)}]
\item For all $\omega\in\Omega(G)$,
$$\rho(A_{\omega,W})\leq \rho(A_{0,W}),\qquad
\rho(B_{\omega,W})\leq \rho(B_{0,W}).$$
\item $\rho(A_{\omega,W})<\rho(A_{0,W})$ for all $\omega\in\Omega(G)$ such that $\underline{\omega}\notin\{0,\theta\}$.
\item If $g=1$, then $\rho(B_{\omega,W})=\rho(B_{0,W})$ for all $\omega\in\Omega(G)$.
\item If $G$ is non-bipartite and has genus $g\geq 2$, then
$$
\rho(B_{\omega,W})=\rho(B_{0,W})
\quad\Longleftrightarrow\quad
\underline{\omega}\in\{0,\theta\}.
$$
\item Suppose that $G$ is bipartite and has genus $g\geq2$.  If the period
character does not exist, then
$$
\rho(B_{\omega,W})=\rho(B_{0,W})
\quad\Longleftrightarrow\quad
\underline{\omega}=0.
$$
If the period character $\eta$ exists, then
$$
\rho(B_{\omega,W})=\rho(B_{0,W})
\quad\Longleftrightarrow\quad
\underline{\omega}\in\{0,\eta\}.
$$
\end{enumerate}
\end{theorem}

\begin{remark}
By (SA2), the canonical translation in (SA1) is an involution of
$\mcalX_\mbbC$.  At the level of gain representatives, if $\gamma$ represents
$\xi$ and $\gamma_\theta$ represents $\theta$, then
$\gamma_\theta\cdot\gamma$ represents $\theta+\xi$ and
$$
A_{\gamma_\theta\cdot\gamma,W}\sim-A_{\gamma,W},\qquad
B_{\gamma_\theta\cdot\gamma,W}\sim-B_{\gamma,W}.
$$
This canonical translation is available for every positive directed weight.
If $W$ is symmetric, (SA1) also gives spectral antisymmetry at the
complementary character $\theta-\underline{\omega}$, which is characterized by
$$\underline{\omega}+(\theta-\underline{\omega})=\theta.$$
\end{remark}

\begin{remark}
In the language of complex unit gain graphs, fixing a unitary character $\underline{\omega}$ amounts to fixing a switching class of gains on the underlying graph $G$.  The usual gain-graph notions of balancedness, sign-symmetry, and spectral symmetry are properties of an individual switching class or compare that class with its negative~\cite{Z1989Biased,R2012Spectral,WvD2023Symmetry}.  Statement (SA1) gives a torus-level analogue: canonical translation pairs every full character $\xi$ with $\theta+\xi$, and these two switching classes have spectra related by sign-negation.  Thus the canonical character plays the role of a global antibalancing twist for the whole family of character gains.
\end{remark}

\begin{remark}
The theorem is stated in the positive directed-weight setting.  The spectral antisymmetry in (SA1) is purely algebraic and does not require symmetric weights, Hermitian matrices, or even unitary gains.  The symmetric-weight refinement follows by combining the representative-level similarity above with Lemma~\ref{L:PropAdj}, which gives $\spec A_{-\omega,W}=\spec A_{\omega,W}$ and $\spec B_{-\omega,W}=\spec B_{\omega,W}$ in the symmetric case.
If $W$ is reciprocal, then the product $W(\ve)\gamma(\ve)$ is itself a complex gain for every complex gain $\gamma$ representing a full character.  In this gain-graph reading, (SA1) says that multiplying the gain by the canonical sign gain, which is switching equivalent to the all-edge sign $-1$, sends both the vertex and nonbacktracking edge spectra to their negatives.  In the unweighted specialization $W\equiv1$, similarity holds for the edge matrices as well as for the vertex matrices.  In particular, for $1$-forms $\omega,\omega'$ satisfying $\underline{\omega'}=\theta+\underline{\omega}$,
$$\spec A_{\omega'}=-\spec A_\omega,\qquad \spec B_{\omega'}=-\spec B_\omega.$$
The unitary rigidity statement in (SA4) specializes to the usual unweighted uniqueness clause, now for both vertex and edge matrices: $\underline{\omega}=\theta$ if and only if $\spec A_\omega=-\spec A$ and if and only if $\spec B_\omega=-\spec B$.
\end{remark}

\begin{remark}\label{R:period-character}
For non-bipartite $G$, the period $\nu_G$ is odd and $\eta=\theta$; for
bipartite $G$, $\nu_G$ is even, $\theta=0$, and $\eta$ may not exist.  Thus
(SR5) shows that among the $2^g$ $2$-torsion characters, at most one
nontrivial character can be an additional spectral-radius maximizer.  Its
existence and the equality locus are independent of the positive directed
weight.  The bipartite theta graphs $G_2$ and $G_3$ in
Subsection~\ref{SS:example_spec} realize the two alternatives: the period
character does not exist for $G_2$, whereas it exists for $G_3$.  The detailed
homological verification is given there.  Whenever $\eta$ exists, the
rotation in (PC2) is consistent with its being $2$-torsion: applying (PC2)
twice shows that every twisted edge spectrum $\spec B_{\xi,W}$ is invariant
under multiplication by $e(1/\nu_G)$.  The rotation persists even when $W$ is
symmetric.  Indeed, because $\eta=-\eta$, the symmetry
$\spec B_{-\omega,W}=\spec B_{\omega,W}$ becomes tautological at $\eta$ and
does not identify its spectrum with the untwisted one; (SA4) instead gives
$\spec B_{\eta,W}\neq\spec B_{0,W}$.  For $G_3$, more precisely,
$\spec B_{\eta_{G_3},W}=\sqrt{-1}\,\spec B_{0,W}$.
Unlike (SA1), the period-translation law concerns only edge adjacency: there
is no corresponding scalar-rotation formula for the vertex matrices in
general.  When $G$ is non-bipartite, $\eta=\theta$, and the vertex statement
is the sign-negation already supplied by (SA1).
\end{remark}

The conceptual definition of the period character leads to the following
finite decision algorithm.  When the character exists, the same algorithm
also constructs a cotree-supported period $1$-form, parallel to
Construction~\ref{Cs:CanoForm2} for the canonical character.

\begin{algorithm}[Period-character decision and cotree construction]
\label{A:period-character-algorithm}
\textbf{Input.}\quad A connected graph $G$ of genus $g\geq2$, an orientation
$\mfrako$, and a spanning tree $T$ of $G$.

\smallskip
\noindent\textbf{Step 1.}\quad Form the cyclic core $G^\circ$ of
Lemma~\ref{L:genus-2-edge-walk} by repeatedly deleting vertices of degree one
and their incident edges.  On the oriented
edges of $G^\circ$, use the nonbacktracking successor relation
$$
\va\shortrightarrow\vb
\quad\Longleftrightarrow\quad
\va(1)=\vb(0),\qquad \vb\neq\va^{-1}.
$$
Thus a sequence of successive oriented edges is precisely a nonbacktracking
edge-walk.  Compute the period $\nu_G$ of $G$.

\smallskip
\noindent\textbf{Step 2.}\quad Write
$$
\{\ve_1,\ldots,\ve_g\}=\vE_{\mfrako}(G)\setminus T,
\qquad
C_i:=\ve_i\cdot T(\ve_i(1),\ve_i(0)).
$$
Since each $C_i$ is a prime cycle, $\nu_G\mid\tau(C_i)$.  Let
$q_i\in\{0,1\}$ be determined by
$$
q_i\equiv\frac{\tau(C_i)}{\nu_G}\pmod2.
$$
Form the candidate cotree-supported $1$-form
\begin{equation}\label{E:period-cotree-form}
\omega^{\mathrm{per}}_{(T,\mfrako)}
:=\frac12\sum_{i=1}^g q_i\,d\ve_i.
\end{equation}
Equivalently, define $c_T:\vE(G^\circ)\to\mbbZ/2\mbbZ$ by
$$
c_T(\ve_i)=c_T(\ve_i^{-1})=q_i,
\qquad
c_T(\ve)=0
$$
on the oriented tree edges.

\smallskip
\noindent\textbf{Step 3.}\quad Choose a base oriented edge $\va_\ast$.
For each $\vb\neq\va_\ast$, choose a predecessor
$\operatorname{pred}(\vb)$ such that
$\operatorname{pred}(\vb)\shortrightarrow\vb$ and repeated application of
$\operatorname{pred}$ eventually reaches $\va_\ast$.  Such a choice exists
because $G^\circ$ is edge-walk-connected; for example, choose the
predecessors along a family of shortest edge-walks from $\va_\ast$.

Set $p(\va_\ast)=0$.  Beginning with the oriented edges nearest to
$\va_\ast$, define recursively
$$
p(\vb)\equiv p\bigl(\operatorname{pred}(\vb)\bigr)
+\nu_Gc_T(\vb)-1\pmod{2\nu_G}.
$$
This assigns $p(\vb)$ to every oriented edge.  Finally, test the congruence
\begin{equation}\label{E:period-character-congruence}
p(\vb)-p(\va)\equiv \nu_Gc_T(\vb)-1\pmod{2\nu_G}
\end{equation}
for every allowed transition $\va\shortrightarrow\vb$ not already used as a
chosen predecessor transition.

\smallskip
\noindent\textbf{Output.}\quad
If any allowed transition fails the congruence, then the period character
does not exist.  If every transition passes, then the period character exists
and
\begin{equation}\label{E:period-form-output}
\eta=\underline{\omega^{\mathrm{per}}_{(T,\mfrako)}},
\qquad
\eta(C_i^{\ab})=(-1)^{q_i}\qquad(i=1,\ldots,g).
\end{equation}
Thus $\omega^{\mathrm{per}}_{(T,\mfrako)}$ is a period $1$-form.  The output
character and the success or failure of the test are independent of the
chosen orientation and spanning tree.
\end{algorithm}

\begin{proposition}\label{P:period-character-algorithm}
Algorithm~\ref{A:period-character-algorithm} terminates and reports success
if and only if the period character exists.  When it succeeds, its output is
the period character and
$\omega^{\mathrm{per}}_{(T,\mfrako)}$ is a cotree-supported period
$1$-form.  The output character and the success or failure of the test are
independent of the chosen orientation and spanning tree.
\end{proposition}

The algorithm requires neither cycle enumeration nor spectral calculation.
The proposition is proved at the end of
Section~\ref{S:proof-antisymmetry}; a faster local obstruction is recorded in
Subsection~\ref{SS:2adic-refinement}.

We defer the technical and lengthy proof of Theorem~\ref{T:SpecDuality} to
Section~\ref{S:proof-antisymmetry}.  Its unweighted regular-graph
specialization is recorded in Appendix~\ref{C:regular-specializations}.

\subsection{Outer spectral edges and periodic graph interpretation}\label{SS:outer-spectral-edges}
We next spell out what the strict spectral-radius part of Theorem~\ref{T:SpecDuality} says in the language of periodic graph spectra.  The terminology used here is collected in Subsection~\ref{SS:floquet-edges}; the point of the present subsection is to specialize that general periodic language to the maximal abelian cover of $G$.

Let $W$ be a positive symmetric weight, so every $A_{\omega,W}$ is
Hermitian.  The deck group of the maximal abelian cover is canonically
identified with $H_1(G,\mbbZ)$.  Under this identification, the displacement
homomorphism of Subsection~\ref{SS:floquet-edges} sends the homology class of
a closed walk to itself.  Its dual therefore identifies the Bloch torus of
the cover with
$$
\widehat{H_1(G,\mbbZ)}=\mcalX(G).
$$
Thus, unlike a general abelian periodic cover, the maximal abelian cover
samples every unitary character of $G$.

For $\xi\in\mcalX(G)$, choose a $1$-form $\omega$ with
$\underline{\omega}=\xi$.  The Bloch fiber at $\xi$ is represented by
$A_{\omega,W}$.  Different choices of $\omega$ give unitarily similar
matrices by Lemma~\ref{L:switching-similarity}, so write
$$\lambda_1(\xi)\leq\cdots\leq\lambda_n(\xi)$$
for its ordered eigenvalues and $\spec A_{\xi,W}$ for its spectrum.  Set
$$\Sigma_{\operatorname{Floq}}(W):=\bigcup_{\xi\in\mcalX(G)}\spec A_{\xi,W}.$$
The endpoints of $\Sigma_{\operatorname{Floq}}(W)$ are the two \emph{outer spectral edges} of this finite-dimensional Floquet family.
For a general periodic cover, the Bloch characters themselves form a
subtorus.  More precisely, let $\Gamma$ be an abelian periodic cover with
finite quotient $G$ and lattice deck group $\Lambda$, and let
$$\nu:H_1(G,\mbbZ)\longrightarrow\Lambda$$
be its displacement homomorphism.  The pullback
$$\nu^*:\widehat{\Lambda}\longrightarrow\mcalX(G),\qquad
\nu^*\zeta:=\zeta\circ\nu,$$
has image the closed Bloch subtorus
$$\mathcal S_\Gamma:=\nu^*\widehat{\Lambda}\subset\mcalX(G)$$
of $\Gamma$.  For the maximal abelian cover,
$\mathcal S_\Gamma=\mcalX(G)$.

Now let $\widetilde W$ be a nowhere-zero real symmetric periodic weight,
regarded also as a weight on $G$, and put
$$
W:=|\widetilde W|,
\qquad
\sigma:=\operatorname{sgn}\widetilde W.
$$
Let $\zeta_\sigma\in\mcalX(G)$ be the $2$-torsion switching class of the sign
gain $\sigma$.  We use $A_{\xi,\widetilde W}$ for the matrix defined by the
same entrywise formula as in Definition~\ref{D:weighted-adj}, with
$\widetilde W$ in place of a positive weight.  For every
$\zeta\in\widehat\Lambda$, absorbing the sign into the Bloch gain gives
$$
\spec A_{\nu^*\zeta,\widetilde W}
=\spec A_{\zeta_\sigma+\nu^*\zeta,W}.
$$
Consequently, rewriting the signed Bloch family using the positive weight
$W$ replaces $\mathcal S_\Gamma$ by the affine subtorus
$$
\mcalT_{\Gamma,\widetilde W}
:=\zeta_\sigma+\mathcal S_\Gamma.
$$
This affine subtorus is a subtorus precisely when
$\zeta_\sigma\in\mathcal S_\Gamma$.  In particular, it is a subtorus when
$\widetilde W$ is positive.  Thus the affine parameter set below is not an
extra Floquet parameter; it is produced by the sign of $\widetilde W$.

For the statement, let $\mathcal S\subset\mcalX(G)$ be a closed subtorus and
put
$$
\mcalT:=\zeta_\sigma+\mathcal S.
$$
For $\xi\in\mcalX(G)$, write
$$\lambda_1^W(\xi)\leq\cdots\leq\lambda_n^W(\xi)$$
for the ordered eigenvalues of $A_{\xi,W}$, and set
$$
\Sigma_{\mathcal S}(\widetilde W):=
\bigcup_{\xi\in\mathcal S}\spec A_{\xi,\widetilde W}.
$$
Notice that the lower universal value need not be the smallest eigenvalue of
the untwisted matrix $A_{0,W}$.  Rather, spectral antisymmetry gives
$$
\lambda_1^W(\theta)=-\lambda_n^W(0)=-\rho(A_{0,W}),
$$
so this value is realized by the canonical twist $A_{\theta,W}$.  It is also
realized by $A_{0,W}$ precisely when $G$ is bipartite, in which case
$\theta=0$.

\begin{theorem}[\textbf{Outer spectral edges on affine Bloch tori}]\label{T:outer-spectral-edges}
Under the preceding notation,
\begin{enumerate}[(i)]
\item The outer spectral edges are
$$
\sup\Sigma_{\mathcal S}(\widetilde W)
=\max_{\xi\in\mcalT}\lambda_n^W(\xi),
\qquad
\inf\Sigma_{\mathcal S}(\widetilde W)
=-\max_{\xi\in\theta+\mcalT}\lambda_n^W(\xi).
$$

\item The upper endpoint equals $\rho(A_{0,W})$ precisely when
$$
\sup\Sigma_{\mathcal S}(\widetilde W)=\rho(A_{0,W})
\quad\Longleftrightarrow\quad
\zeta_\sigma\in\mathcal S.
$$
When these equivalent conditions hold, the supremum is attained among
$\xi\in\mathcal S$ only at $\xi=\zeta_\sigma$, and the eigenvalue
$\rho(A_{0,W})$ of $A_{\zeta_\sigma,\widetilde W}$ is simple.

\item The lower endpoint equals $-\rho(A_{0,W})$ precisely when
$$
\inf\Sigma_{\mathcal S}(\widetilde W)=-\rho(A_{0,W})
\quad\Longleftrightarrow\quad
\theta+\zeta_\sigma\in\mathcal S.
$$
When these equivalent conditions hold, the infimum is attained among
$\xi\in\mathcal S$ only at $\xi=\theta+\zeta_\sigma$, and the eigenvalue
$-\rho(A_{0,W})$ of $A_{\theta+\zeta_\sigma,\widetilde W}$ is simple.

\item If neither membership condition holds, then
$$
-\rho(A_{0,W})
<\inf\Sigma_{\mathcal S}(\widetilde W)
\leq\sup\Sigma_{\mathcal S}(\widetilde W)
<\rho(A_{0,W}).
$$
\end{enumerate}

In particular, for a positive symmetric weight $W$ on the maximal abelian cover,
$$
\sup\Sigma_{\operatorname{Floq}}(W)
=\lambda_n^W(0)=\rho(A_{0,W}),
\qquad
\inf\Sigma_{\operatorname{Floq}}(W)
=\lambda_1^W(\theta)=-\rho(A_{0,W}).
$$
\end{theorem}

\begin{proof}
The sign gain represents the $2$-torsion character $\zeta_\sigma$.  If a gain $\gamma$ represents $\xi$, then
$$
A_{\gamma,\widetilde W}=A_{\sigma\cdot\gamma,W},
$$
and $\sigma\cdot\gamma$ represents $\zeta_\sigma+\xi$.  Hence
$$
\Sigma_{\mathcal S}(\widetilde W)
=\bigcup_{\xi\in\mcalT}\spec A_{\xi,W}.
$$

Since $W$ is symmetric, every $A_{\xi,W}$ is Hermitian.  Therefore
$$
\sup\Sigma_{\mathcal S}(\widetilde W)
=\max_{\xi\in\mcalT}\lambda_n^W(\xi).
$$
Theorem~\ref{T:SpecDuality}(SA1) gives
$$
\lambda_1^W(\xi)=-\lambda_n^W(\theta+\xi),
$$
and consequently
$$
\inf\Sigma_{\mathcal S}(\widetilde W)
=-\max_{\xi\in\theta+\mcalT}\lambda_n^W(\xi).
$$

The matrix $A_{0,W}$ is irreducible and nonnegative, so
$\lambda_n^W(0)=\rho(A_{0,W})$ is simple.  By
Theorem~\ref{T:SpecDuality}(SR2), a positive-weight character carries the
top eigenvalue $\rho(A_{0,W})$ only if it is trivial; indeed, the other
possible spectral-radius extremum $\theta$ can carry the top eigenvalue only
in the bipartite case, when $\theta=0$.  A signed parameter
$\xi\in\mathcal S$ corresponds to the positive-weight parameter
$\zeta_\sigma+\xi$.  Hence the upper universal value is attained if and
only if $\zeta_\sigma+\xi=0$, or equivalently
$\xi=\zeta_\sigma$, since $\zeta_\sigma$ is $2$-torsion.  This proves the
upper-edge criterion and the uniqueness assertion in (ii).

For the lower edge, spectral antisymmetry gives
$$
\lambda_1^W(\zeta_\sigma+\xi)
=-\lambda_n^W(\theta+\zeta_\sigma+\xi).
$$
Thus the lower universal value is attained if and only if
$\theta+\zeta_\sigma+\xi=0$, or equivalently
$\xi=\theta+\zeta_\sigma$, since both $\theta$ and $\zeta_\sigma$ are
$2$-torsion.  This proves the criterion and uniqueness in
(iii), while spectral antisymmetry transfers the Perron--Frobenius
simplicity from $\lambda_n^W(0)$ to $\lambda_1^W(\theta)$.

If neither distinguished character belongs to the relevant translated
torus, compactness and continuity of the band functions give the strict
inequalities.  Finally, taking $\widetilde W=W$, $\zeta_\sigma=0$, and
$\mathcal S=\mcalX(G)$ gives the maximal-abelian-cover statement.
\end{proof}

\begin{remark}
Theorem~\ref{T:outer-spectral-edges} identifies the two endpoints of the union
of spectra over the relevant character torus or affine Bloch subtorus.  When
one of the universal values $\pm\rho(A_{0,W})$ is attained, the simplicity in
parts (ii)--(iii) is the Perron--Frobenius simplicity of that extremal fiber
eigenvalue.  Its relation to the spectral edges conjecture is through these
two global outer endpoints: the conjecture concerns the finer generic
structure of all band extrema, including internal band edges, such as their
isolation, non-degeneracy, and single-band attainment.
\end{remark}

\begin{remark}[Spanning-tree test]
The two membership conditions in the theorem are easy to test in spanning-tree coordinates.  Choose an orientation $\mfrako$ and a spanning tree $T$, let $\ve_1,\ldots,\ve_g$ be the positively oriented cotree edges, and use Remark~\ref{R:X_TO} to identify $\mcalX(G)$ set-theoretically with $[0,1)^g$.  Construction~\ref{Cs:CanoForm2} gives the canonical coordinate vector
$$\kappa_i=
\begin{cases}
0,&\text{if }\ve_i\text{ joins the two parts of the bipartition of }T,\\
1/2,&\text{otherwise.}
\end{cases}$$
If the displacement map $\nu$ has matrix $N$ in the spanning-tree basis,
then $\mathcal S_\Gamma=\nu^*\widehat{\Lambda}$ is the image of the induced
torus homomorphism
$N^T:(\mbbR/\mbbZ)^d\to(\mbbR/\mbbZ)^g$.  For a signed weight
$\widetilde W=\sigma\cdot W$, the upper and lower outer edges are inherited
exactly when the coordinate vectors of $\zeta_\sigma$ and
$\theta+\zeta_\sigma$, respectively, lie in this image.  This gives a
concrete way to build examples in which the periodic family sees either,
both, or neither distinguished outer edge.  Subsection~\ref{SS:example-periodic-edges}
applies this test to a $\mbbZ^2$-periodic graph and exhibits the cases in
which only the upper edge, only the lower edge, or neither edge is inherited.
\end{remark}

\subsection{Weighted trace and Fourier formulas}\label{SS:trace}

\emph{Circuit classes.}
For $l\geq1$, let $\mcalC(l)$ be the set of all circuits of length $l$.  For
$\alpha\in H_1(G,\mbbZ)$ and $\underline{\alpha}\in Q_\Lambda$, using the
quotient notation of Subsection~\ref{SS:character-group}, set
$$
\mcalC(\alpha,l):=\{C\in\mcalC(l):C^{\ab}=\alpha\},\qquad
\mcalC(\underline{\alpha},l):=\{C\in\mcalC(l):\underline{C^{\ab}}=\underline{\alpha}\}.
$$
Since $G$ is finite, $\mcalC(l)$ is finite.  Thus every sum over homology
classes below has finite support when the length $l$ is fixed.

\begin{definition}[Weighted and unweighted circuit sums]\label{D:weighted-circuit-sums}\label{D:count-function-N}
Let $W:\vE(G)\to\mbbR_{>0}$ be a positive directed weight.  For
$\alpha\in H_1(G,\mbbZ)$ and $\underline{\alpha}\in Q_\Lambda$, define
$$
N_W(l):=\sum_{C\in\mcalC(l)}W(C),\qquad
N_W(\alpha,l):=\sum_{C\in\mcalC(\alpha,l)}W(C),\qquad
N_W(\underline{\alpha},l):=\sum_{C\in\mcalC(\underline{\alpha},l)}W(C).
$$

For the trivial weight, write
$$
N(l):=N_1(l)=|\mcalC(l)|,
$$
and
$$
N(\alpha,l):=N_1(\alpha,l)=|\mcalC(\alpha,l)|,\qquad
N(\underline{\alpha},l):=N_1(\underline{\alpha},l)
=|\mcalC(\underline{\alpha},l)|.
$$
\end{definition}

\begin{definition}[Weighted trace distribution]\label{D:weighted-trace}\label{D:trace}
Let $W:\vE(G)\to\mbbR_{>0}$ be a positive directed weight. For each $\omega\in\Omega(G)$ and $l\geq1$, define
$$\mcalK_W(\omega,l):=\operatorname{tr}\left(B_{\omega,W}^l\right).$$
We call $\mcalK_W(\omega,l)$ the \emph{weighted trace distribution} of order $l$.
For the trivial weight $W\equiv1$, we write
$$\mcalK(\omega,l):=\mcalK_1(\omega,l).$$
\end{definition}

\begin{lemma} \label{L:trace-def}
Let $W$ be a positive directed weight. The following are equivalent definitions of $\mcalK_W(\omega,l)$:
\begin{enumerate}[(a)]
\item $\mcalK_W(\omega,l)=\sum_{C\in \mcalC(l)}W(C)\chi_\omega(C)$;
\item $\mcalK_W(\omega,l)=\sum_{\lambda\in \spec B_{\omega,W}}\lambda^l$.
\end{enumerate}
\end{lemma}

\begin{proof}[Proof of Lemma~\ref{L:trace-def}]
For (a), first note that by definition of the twisted weighted edge adjacency matrix $B_{\omega,W}$, the $(\ve,\ve)$-entry of $B_{\omega,W}^l$ is the sum over the circuits at the base oriented edge $\ve$, each weighted by $W$ and twisted by $\chi_\omega$. Therefore, being the trace of $B_{\omega,W}^l$, $\mcalK_W(\omega,l)$ is exactly $\sum_{C\in \mcalC(l)}W(C)\chi_\omega(C)$. 

(b) follows from the facts that  $\mcalK_W(\omega,l)=\sum_{\lambda'\in \spec B_{\omega,W}^l}\lambda'$ and $\spec B_{\omega,W}^l=\{\lambda^l\mid \lambda\in \spec B_{\omega,W}\}$ (as a multiset). 
\end{proof}

By Lemma~\ref{L:trace-def}(a), for fixed $W$ the function
$\omega\mapsto\mcalK_W(\omega,l)$ is smooth.

\begin{lemma} \label{L:trace-alt}
Let $W$ be a positive directed weight. For $\omega,\omega'\in\Omega$, if $\underline{\omega}=\underline{\omega'}$, then $\mcalK_W(\omega,l)=\mcalK_W(\omega',l)$ for all $l$. 
\end{lemma}

\begin{proof}
We see from Lemma~\ref{L:switching-similarity} that $B_{\omega,W}$ and $B_{\omega',W}$ are isospectral whenever  $\underline{\omega}=\underline{\omega'}$. Therefore, $\mcalK_W(\omega,l)=\mcalK_W(\omega',l)$ by Lemma~\ref{L:trace-def}(b). 
\end{proof}

By this lemma, for each $\underline{\omega} \in \mcalX$, we may also write $\mcalK_W(\underline{\omega},l)=\mcalK_W(\omega,l)$, considered as a function over the character group $\mcalX$.

\begin{lemma} \label{L:Kappa}
Let $W$ be a positive directed weight, and let $\theta$ be the canonical character of $G$.  The weighted trace distribution has the following properties:
\begin{enumerate}[(a)]
\item $\max_{\omega\in\Omega}|\mcalK_W(\omega,l)|=\mcalK_W(0,l)=N_W(l)$. 
\item For every $\omega\in\Omega$, $\mcalK_W(-\omega,l)=\overline{\mcalK_W(\omega,l)}$.  If $W$ is symmetric, then $\mcalK_W(\omega,l)$ is real and $\mcalK_W(\omega,l)=\mcalK_W(-\omega,l)$.
\item For every $\omega\in\Omega$,
$$\mcalK_W(\theta+\underline{\omega},l)=(-1)^l\mcalK_W(\underline{\omega},l).$$
In particular, $\mcalK_W(\theta,l)=\mcalK_W(0,l)$ when $l$ is even, and $\mcalK_W(\theta,l)=-\mcalK_W(0,l)$ when $l$ is odd.  If $W$ is symmetric, equivalently
$$\mcalK_W(\theta-\underline{\omega},l)=(-1)^l\mcalK_W(\underline{\omega},l).$$
\item If $G$ is bipartite, then $\mcalK_W(\omega,l)=0$ for all $\omega\in\Omega$ when $l$ is odd. 
\end{enumerate}
\end{lemma}
\begin{proof}
For (a), note that $\mcalK_W(\omega,l)=\sum_{C\in \mcalC(l)}W(C)\chi_\omega(C)$ (Lemma~\ref{L:trace-def}(a)). Therefore, $\mcalK_W(0,l)=\sum_{C\in\mcalC(l)}W(C)=N_W(l)$ and $|\mcalK_W(\omega,l)|\leq \mcalK_W(0,l)$. 

For (b), the first assertion follows from Lemma~\ref{L:trace-def}(a), since $W$ is real and $\chi_{-\omega}(C)=\overline{\chi_\omega(C)}$.  If $W$ is symmetric, then Lemma~\ref{L:PropAdj}(c) gives that $B_{\omega,W}$ and $B_{-\omega,W}$ are similar, so Lemma~\ref{L:trace-def}(b) gives $\mcalK_W(\omega,l)=\mcalK_W(-\omega,l)$.  Together with the first assertion, this also implies that $\mcalK_W(\omega,l)$ is real.

For (c), by Lemma~\ref{L:unique}, $\theta(C)=(-1)^l$ for every $C\in\mcalC(l)$. Hence Lemma~\ref{L:trace-def}(a) gives
$$\mcalK_W(\theta+\underline{\omega},l)
=\sum_{C\in\mcalC(l)}W(C)\theta(C)\chi_\omega(C)
=(-1)^l\mcalK_W(\underline{\omega},l).$$
The final equivalent form in the symmetric case follows from (b), because $\theta-\underline{\omega}=\theta+(-\underline{\omega})$.

For (d), if $G$ is bipartite, then there are no odd-length circuits, so $\mcalC(l)=\emptyset$ when $l$ is odd.  The assertion follows from Lemma~\ref{L:trace-def}(a).

\end{proof}

\begin{theorem}[\textbf{Weighted Fourier formulas}]\label{T:Fourier}\label{T:weighted-trace-formula}
Let $W:\vE(G)\to\mbbR_{>0}$ be a positive directed weight and let $l\geq1$.
\begin{enumerate}[(a)]
\item For every $\omega\in\Omega(G)$,
$$
\mcalK_W(\omega,l)
=\sum_{\lambda\in\spec B_{\omega,W}}\lambda^l
=\sum_{C\in\mcalC(l)}W(C)\chi_\omega(C)
=\sum_{\alpha\in H_1(G,\mbbZ)}\chi_\omega(\alpha)N_W(\alpha,l),
$$
and, for every $\alpha\in H_1(G,\mbbZ)$,
$$
N_W(\alpha,l)=\frac{1}{\vol(\mcalX)}
\int_{\mcalX}\chi_{-\omega}(\alpha)\mcalK_W(\underline{\omega},l)dV_\omega.
$$
\item Let $\Lambda\subseteq H_1(G,\mbbZ)$ be a full-rank sublattice.  For
$\underline{\omega}\in\widehat{Q_\Lambda}$ and
$\underline{\alpha}\in Q_\Lambda$,
$$
\mcalK_W(\underline{\omega},l)
=\sum_{\underline{\alpha}\in Q_\Lambda}
\chi_{\underline{\omega}}(\underline{\alpha})N_W(\underline{\alpha},l),
$$
and
$$
N_W(\underline{\alpha},l)=\frac{1}{|Q_\Lambda|}
\sum_{\underline{\omega}\in\widehat{Q_\Lambda}}
\chi_{-\underline{\omega}}(\underline{\alpha})
\mcalK_W(\underline{\omega},l).
$$
\end{enumerate}
\end{theorem}

\begin{proof}
By Lemma~\ref{L:trace-def},
$$
\mcalK_W(\omega,l)
=\sum_{\lambda\in\spec B_{\omega,W}}\lambda^l
=\sum_{C\in\mcalC(l)}W(C)\chi_\omega(C).
$$
Since $\chi_\omega(C)=\chi_\omega(C^{\ab})$, grouping the circuits according
to their homology classes gives
\begin{align*}
\mcalK_W(\omega,l)
&=\sum_{\alpha\in H_1(G,\mbbZ)}
\sum_{C\in\mcalC(\alpha,l)}W(C)\chi_\omega(C)\\
&=\sum_{\alpha\in H_1(G,\mbbZ)}
\chi_\omega(\alpha)\sum_{C\in\mcalC(\alpha,l)}W(C)\\
&=\sum_{\alpha\in H_1(G,\mbbZ)}
\chi_\omega(\alpha)N_W(\alpha,l).
\end{align*}
Only finitely many homology classes contribute, because $\mcalC(l)$ is
finite.  Thus both the grouping above and the passage of the homology sum
through the integral below are finite rearrangements.

Now multiply this identity by $\chi_{-\omega}(\alpha_0)$ and integrate over
the character torus.  Proposition~\ref{P:CharacterOrth} gives
\begin{align*}
&\frac1{\vol(\mcalX)}
\int_{\mcalX}\chi_{-\omega}(\alpha_0)
\mcalK_W(\underline{\omega},l)\,dV_\omega\\
&\qquad =
\sum_{\alpha\in H_1(G,\mbbZ)}N_W(\alpha,l)
\frac1{\vol(\mcalX)}
\int_{\mcalX}\chi_\omega(\alpha-\alpha_0)\,dV_\omega\\
&\qquad =
\sum_{\alpha\in H_1(G,\mbbZ)}
N_W(\alpha,l)\delta_{\alpha,\alpha_0}
=N_W(\alpha_0,l).
\end{align*}
This proves the continuous Fourier pair.

Suppose next that $\underline{\omega}\in\widehat{Q_\Lambda}$.  If
$\alpha-\alpha'\in\Lambda$, then
$\chi_{\underline{\omega}}(\alpha)=\chi_{\underline{\omega}}(\alpha')$; hence
the character descends to a character of $Q_\Lambda$.  Therefore
\begin{align*}
\mcalK_W(\underline{\omega},l)
&=\sum_{\alpha\in H_1(G,\mbbZ)}
\chi_{\underline{\omega}}(\alpha)N_W(\alpha,l)\\
&=\sum_{\underline{\alpha}\in Q_\Lambda}
\chi_{\underline{\omega}}(\underline{\alpha})
\sum_{C\in\mcalC(\underline{\alpha},l)}W(C)\\
&=\sum_{\underline{\alpha}\in Q_\Lambda}
\chi_{\underline{\omega}}(\underline{\alpha})
N_W(\underline{\alpha},l).
\end{align*}
Since $Q_\Lambda$ and $\widehat{Q_\Lambda}$ are finite, all rearrangements
in the finite Fourier inversion are finite.
Multiplying by $\chi_{-\underline{\omega}}(\underline{\alpha}_0)$ and summing
over the finite dual group gives, by
Proposition~\ref{P:CharacterOrthLambda}(b),
\begin{align*}
&\frac1{|Q_\Lambda|}
\sum_{\underline{\omega}\in\widehat{Q_\Lambda}}
\chi_{-\underline{\omega}}(\underline{\alpha}_0)
\mcalK_W(\underline{\omega},l)\\
&\qquad =
\sum_{\underline{\alpha}\in Q_\Lambda}
N_W(\underline{\alpha},l)
\frac1{|Q_\Lambda|}
\sum_{\underline{\omega}\in\widehat{Q_\Lambda}}
\chi_{\underline{\omega}}(\underline{\alpha}-\underline{\alpha}_0)\\
&\qquad =
N_W(\underline{\alpha}_0,l).
\end{align*}
This proves the finite Fourier pair.
\end{proof}

\begin{remark}
For the trivial weight $W\equiv1$, Theorem~\ref{T:Fourier} gives the Fourier
and inverse Fourier transforms between $\mcalK$ and the circuit-counting
functions $N$ of Definition~\ref{D:count-function-N}.
\end{remark}

The canonical character imposes an immediate support restriction on these
weighted circuit sums.  By Lemma~\ref{L:unique} and
Definition~\ref{D:canonical-character}, it is characterized by
$\theta(C^{\ab})=(-1)^{\tau(C)}$ for every closed walk $C$; in particular,
$2\theta=0$.
\begin{corollary}[\textbf{Weighted parity vanishing}]\label{C:vanishing}
Let $W$ be a positive directed weight.
\begin{enumerate}[(a)]
\item If $\alpha\in H_1(G,\mbbZ)$, then
$$N_W(\alpha,l)=0\qquad\text{unless}\qquad
\theta(\alpha)=(-1)^l.$$
\item If $\theta\in\widehat{Q_\Lambda}$ and
$\underline{\alpha}\in Q_\Lambda$, then
$$N_W(\underline{\alpha},l)=0\qquad\text{unless}\qquad
\theta(\underline{\alpha})=(-1)^l.$$
\item If $G$ is bipartite, then $N_W(\alpha,l)=0$ and
$N_W(\underline{\alpha},l)=0$ for every odd $l$.
\end{enumerate}
In each assertion, the same vanishing holds for one positive directed weight
if and only if it holds for every such weight, equivalently for $W\equiv1$.
\end{corollary}

\begin{proof}
If $C\in\mcalC(l)$, then Lemma~\ref{L:unique} gives
$\theta(C^{\ab})=(-1)^l$.  Thus the indicated circuit class is empty when
the displayed parity condition fails.  Conversely, every term in a weighted
circuit sum is positive, so $N_W(\alpha,l)$ or
$N_W(\underline{\alpha},l)$ vanishes exactly when its indexing circuit class
is empty.  The bipartite assertion follows because there are no odd-length
circuits.
\end{proof}

The preceding formulas use the monomial test functions $z\mapsto z^l$.
Their analytic extension below is closer in form to a Selberg trace formula:
it equates a spectral sum against an analytic test function with a
homology-resolved weighted sum over closed circuits.

\begin{theorem}[\textbf{Weighted analytic trace formulas}]\label{T:trace-formula}
Let $W$ be a positive directed weight, let $R>\rho(B_{0,W})$, and let $h$ be
analytic on the disk $D_R=\{z\in\mbbC:|z|<R\}$.  Write
$$
h(z)=\sum_{l\geq0}\widehat h(l)z^l,
\qquad
\widehat h(l)=\frac{1}{2\pi\sqrt{-1}}
\oint_\Gamma h(z)z^{-l-1}\,dz,
$$
where $\Gamma$ is any positively oriented simple closed contour about the
origin contained in $D_R$.  Then:
\begin{enumerate}[(a)]
\item For every $\omega\in\Omega(G)$,
$$
\sum_{\lambda\in\spec B_{\omega,W}}\bigl(h(\lambda)-h(0)\bigr)
=\sum_{l\geq1}\sum_{\alpha\in H_1(G,\mbbZ)}
\chi_\omega(\alpha)N_W(\alpha,l)\widehat h(l).
$$
\item For every $\alpha\in H_1(G,\mbbZ)$,
\begin{align*}
&\int_{\mcalX}\chi_{-\omega}(\alpha)
\sum_{\lambda\in\spec B_{\omega,W}}\bigl(h(\lambda)-h(0)\bigr)dV_\omega\\
&\hspace{35mm}=\vol(\mcalX)\sum_{l\geq1}N_W(\alpha,l)\widehat h(l).
\end{align*}
\item For $\underline{\omega}\in\widehat{Q_\Lambda}$,
$$
\sum_{\lambda\in\spec B_{\omega,W}}\bigl(h(\lambda)-h(0)\bigr)
=\sum_{l\geq1}\sum_{\underline{\alpha}\in Q_\Lambda}
\chi_{\underline{\omega}}(\underline{\alpha})
N_W(\underline{\alpha},l)\widehat h(l).
$$
\item For $\underline{\alpha}\in Q_\Lambda$,
\begin{align*}
&\sum_{\underline{\omega}\in\widehat{Q_\Lambda}}
\chi_{-\underline{\omega}}(\underline{\alpha})
\sum_{\lambda\in\spec B_{\omega,W}}\bigl(h(\lambda)-h(0)\bigr)\\
&\hspace{35mm}=|Q_\Lambda|\sum_{l\geq1}
N_W(\underline{\alpha},l)\widehat h(l).
\end{align*}
\end{enumerate}
\end{theorem}

\begin{proof}
Theorem~\ref{T:SpecDuality}(SR1) gives
$\rho(B_{\omega,W})\leq\rho(B_{0,W})<R$, so the Taylor series for $h$ converges
at every eigenvalue of $B_{\omega,W}$.  For fixed $\omega$, the spectrum is a
finite multiset, so the eigenvalue sum and the Taylor series may be
interchanged:
\begin{align*}
\sum_{\lambda\in\spec B_{\omega,W}}\bigl(h(\lambda)-h(0)\bigr)
&=\sum_{\lambda\in\spec B_{\omega,W}}\sum_{l\geq1}
\widehat h(l)\lambda^l\\
&=\sum_{l\geq1}\widehat h(l)
\sum_{\lambda\in\spec B_{\omega,W}}\lambda^l\\
&=\sum_{l\geq1}\mcalK_W(\omega,l)\widehat h(l).
\end{align*}
Substituting the Fourier expansion in Theorem~\ref{T:Fourier}(a) gives
\begin{align*}
\sum_{\lambda\in\spec B_{\omega,W}}\bigl(h(\lambda)-h(0)\bigr)
&=\sum_{l\geq1}\widehat h(l)
\sum_{\alpha\in H_1(G,\mbbZ)}
\chi_\omega(\alpha)N_W(\alpha,l),
\end{align*}
which proves (a).

For (b), multiply the displayed identity by $\chi_{-\omega}(\alpha_0)$ and
integrate over $\mcalX$.  To justify term-by-term integration, choose $r$ with
$\rho(B_{0,W})<r<R$.  If $q$ is the common size of the matrices, then,
uniformly in $\omega$,
$$
\sum_{\lambda\in\spec B_{\omega,W}}\sum_{l\geq1}
\left|\widehat h(l)\lambda^l\right|
\leq q\sum_{l\geq1}|\widehat h(l)|r^l<\infty.
$$
The inverse transform in Theorem~\ref{T:Fourier}(a) therefore gives
\begin{align*}
&\int_{\mcalX}\chi_{-\omega}(\alpha_0)
\sum_{\lambda\in\spec B_{\omega,W}}\bigl(h(\lambda)-h(0)\bigr)dV_\omega\\
&\qquad =
\sum_{l\geq1}\widehat h(l)
\int_{\mcalX}\chi_{-\omega}(\alpha_0)
\mcalK_W(\omega,l)dV_\omega\\
&\qquad =
\vol(\mcalX)\sum_{l\geq1}N_W(\alpha_0,l)\widehat h(l).
\end{align*}

If $\underline{\omega}\in\widehat{Q_\Lambda}$, the same initial identity and
Theorem~\ref{T:Fourier}(b) give
\begin{align*}
\sum_{\lambda\in\spec B_{\omega,W}}\bigl(h(\lambda)-h(0)\bigr)
&=\sum_{l\geq1}\widehat h(l)\mcalK_W(\underline{\omega},l)\\
&=\sum_{l\geq1}\widehat h(l)
\sum_{\underline{\alpha}\in Q_\Lambda}
\chi_{\underline{\omega}}(\underline{\alpha})
N_W(\underline{\alpha},l),
\end{align*}
which proves (c).  Finally, multiply this identity by
$\chi_{-\underline{\omega}}(\underline{\alpha}_0)$ and sum over
$\widehat{Q_\Lambda}$.  Using the inverse transform in
Theorem~\ref{T:Fourier}(b), we obtain
\begin{align*}
&\sum_{\underline{\omega}\in\widehat{Q_\Lambda}}
\chi_{-\underline{\omega}}(\underline{\alpha}_0)
\sum_{\lambda\in\spec B_{\omega,W}}\bigl(h(\lambda)-h(0)\bigr)\\
&\qquad =
\sum_{l\geq1}\widehat h(l)
\sum_{\underline{\omega}\in\widehat{Q_\Lambda}}
\chi_{-\underline{\omega}}(\underline{\alpha}_0)
\mcalK_W(\underline{\omega},l)\\
&\qquad =
|Q_\Lambda|\sum_{l\geq1}
N_W(\underline{\alpha}_0,l)\widehat h(l).
\end{align*}
This is (d).
\end{proof}

\subsection{Weighted cycle counting in homology classes}

\emph{Prime-cycle and cycle classes.}
For $l\geq1$, $\alpha\in H_1(G,\mbbZ)$, and
$\underline{\alpha}\in Q_\Lambda$, set
\begin{align*}
\bmcalP(l)&:=\{[P]\in\bmcalP:\tau(P)=l\},
&\bmcalC(l)&:=\{[C]\in\bmcalC:\tau(C)=l\},\\
\bmcalP(\alpha,l)&:=\{[P]\in\bmcalP(l):P^{\ab}=\alpha\},
&\bmcalC(\alpha,l)&:=\{[C]\in\bmcalC(l):C^{\ab}=\alpha\},\\
\bmcalP(\underline{\alpha},l)&:=
\{[P]\in\bmcalP(l):\underline{P^{\ab}}=\underline{\alpha}\},
&\bmcalC(\underline{\alpha},l)&:=
\{[C]\in\bmcalC(l):\underline{C^{\ab}}=\underline{\alpha}\}.
\end{align*}
All these sets are finite for fixed $l$.

\begin{definition}[Weighted and unweighted cycle sums]\label{D:weighted-pi}
Let $W$ be a positive directed weight.  Define
\begin{align*}
\pi_W(l)&:=\sum_{[P]\in\bmcalP(l)}W(P),
&\Pi_W(l)&:=\sum_{[C]\in\bmcalC(l)}W(C),\\
\pi_W(\alpha,l)&:=\sum_{[P]\in\bmcalP(\alpha,l)}W(P),
&\Pi_W(\alpha,l)&:=\sum_{[C]\in\bmcalC(\alpha,l)}W(C),\\
\pi_W(\underline{\alpha},l)&:=
\sum_{[P]\in\bmcalP(\underline{\alpha},l)}W(P),
&\Pi_W(\underline{\alpha},l)&:=
\sum_{[C]\in\bmcalC(\underline{\alpha},l)}W(C).
\end{align*}
For the trivial weight, write
$$
\pi:=\pi_1,\qquad \Pi:=\Pi_1.
$$
Equivalently, these unweighted functions are the cardinalities of the
corresponding sets above.

For $r\geq1$, the \emph{pointwise $r$-th power} of $W$ is the positive
directed weight
$$W^r(\ve):=W(\ve)^r.$$
Thus $W(P^r)=W(P)^r=W^r(P)$.
\end{definition}

For $r\geq1$, put
$$
[\alpha\mid r]:=\{\beta\in H_1(G,\mbbZ):r\beta=\alpha\},\qquad
[\underline{\alpha}\mid r]:=
\{\underline{\beta}\in Q_\Lambda:r\underline{\beta}=\underline{\alpha}\},
$$
and let $\mu$ denote the M\"obius function.

\begin{proposition}[\textbf{Weighted prime-power identities}]\label{P:IdenCountFunc}
Let $W$ be a positive directed weight and $l\geq1$.
\begin{enumerate}[(a)]
\item For $\alpha\in H_1(G,\mbbZ)$,
\begin{align*}
\Pi_W(\alpha,l)
&=\sum_{r\mid l}\sum_{\beta\in[\alpha\mid r]}
\pi_{W^r}\left(\beta,\frac lr\right),\\
N_W(\alpha,l)
&=\sum_{r\mid l}\frac lr\sum_{\beta\in[\alpha\mid r]}
\pi_{W^r}\left(\beta,\frac lr\right),\\
\pi_W(\alpha,l)
&=\frac1l\sum_{r\mid l}\mu(r)
\sum_{\beta\in[\alpha\mid r]}N_{W^r}\left(\beta,\frac lr\right).
\end{align*}
\item For $\underline{\alpha}\in Q_\Lambda$,
\begin{align*}
\Pi_W(\underline{\alpha},l)
&=\sum_{r\mid l}\sum_{\underline{\beta}\in
[\underline{\alpha}\mid r]}
\pi_{W^r}\left(\underline{\beta},\frac lr\right),\\
N_W(\underline{\alpha},l)
&=\sum_{r\mid l}\frac lr\sum_{\underline{\beta}\in
[\underline{\alpha}\mid r]}
\pi_{W^r}\left(\underline{\beta},\frac lr\right),\\
\pi_W(\underline{\alpha},l)
&=\frac1l\sum_{r\mid l}\mu(r)
\sum_{\underline{\beta}\in[\underline{\alpha}\mid r]}
N_{W^r}\left(\underline{\beta},\frac lr\right).
\end{align*}
\end{enumerate}
\end{proposition}

\begin{proof}
Every cycle has a unique expression $[C]=[P^r]$, where $[P]$ is a prime
cycle and $r\geq1$.  If $\tau(C)=l$, then $r\mid l$ and
$\tau(P)=l/r$.  Moreover
$$
C^{\ab}=rP^{\ab},\qquad W(C)=W(P^r)=W(P)^r=W^r(P).
$$
Thus a cycle $[C]\in\bmcalC(\alpha,l)$ with prime root $[P]$ and exponent
$r$ is the same thing as a prime cycle
$[P]\in\bmcalP(\beta,l/r)$ with $\beta\in[\alpha\mid r]$, counted with
the powered weight $W^r(P)$.  Summing over the possible exponents $r\mid l$
proves
$$
\Pi_W(\alpha,l)
=\sum_{r\mid l}\sum_{\beta\in[\alpha\mid r]}
\pi_{W^r}\left(\beta,\frac lr\right).
$$
The quotient-homology formula for $\Pi_W(\underline{\alpha},l)$ is the same,
with $Q_\Lambda$ in place of $H_1(G,\mbbZ)$.

For circuits, fix a prime cycle $[P]$ of length $d=l/r$.  The closed
edge-walk $P^r$ has length $l$, and its cyclic orbit contains exactly $d$
based circuits, one for each choice of a starting oriented edge on the prime
cycle $P$.  Hence the same decomposition of cycles into prime powers gives
the additional factor $d=l/r$ and proves
$$
N_W(\alpha,l)
=\sum_{r\mid l}\frac lr
\sum_{\beta\in[\alpha\mid r]}
\pi_{W^r}\left(\beta,\frac lr\right),
$$
and similarly in $Q_\Lambda$.

It remains to invert the second formula.  Let
$$
F_W(\alpha,l):=\frac{N_W(\alpha,l)}{l}.
$$
The identity just proved can be rewritten as
$$
F_W(\alpha,l)
=\sum_{r\mid l}\frac1r
\sum_{\beta\in[\alpha\mid r]}
\pi_{W^r}\left(\beta,\frac lr\right).
$$
Now compute
\begin{align*}
&\frac1l\sum_{r\mid l}\mu(r)
\sum_{\beta\in[\alpha\mid r]}
N_{W^r}\left(\beta,\frac lr\right)\\
&\quad =
\sum_{r\mid l}\frac{\mu(r)}{r}
\sum_{\beta\in[\alpha\mid r]}
F_{W^r}\left(\beta,\frac lr\right)\\
&\quad =
\sum_{r\mid l}\frac{\mu(r)}{r}
\sum_{\beta\in[\alpha\mid r]}
\sum_{s\mid l/r}\frac1s
\sum_{\gamma\in[\beta\mid s]}
\pi_{W^{rs}}\left(\gamma,\frac{l}{rs}\right)\\
&\quad =
\sum_{t\mid l}\frac1t
\left(\sum_{r\mid t}\mu(r)\right)
\sum_{\gamma\in[\alpha\mid t]}
\pi_{W^t}\left(\gamma,\frac lt\right).
\end{align*}
In the last line we put $t=rs$ and used the equivalence
$r\beta=\alpha,\ s\gamma=\beta \Longleftrightarrow t\gamma=\alpha$.
Since $\sum_{r\mid t}\mu(r)$ is $1$ for $t=1$ and $0$ otherwise, only the
$t=1$ term remains, namely $\pi_W(\alpha,l)$.  This proves the inversion
formula in integral homology.  The quotient-homology inversion is identical:
one replaces $\alpha,\beta,\gamma$ by their classes in $Q_\Lambda$.
\end{proof}

\begin{remark}[Weighted and unweighted vanishing]\label{R:vanishing}
Because every weight is positive, a weighted sum in
Definitions~\ref{D:weighted-circuit-sums} and~\ref{D:weighted-pi} vanishes if
and only if its indexing set is empty.  In particular,
\begin{align*}
N_W(\alpha,l)=0
&\Longleftrightarrow \Pi_W(\alpha,l)=0\\
&\Longleftrightarrow
\pi_{W^r}(\beta,l/r)=0
\quad\text{for all }r\mid l\text{ and }\beta\in[\alpha\mid r],
\end{align*}
and the same equivalence holds in $Q_\Lambda$.  These conditions are also
equivalent to the corresponding unweighted vanishing statements.
\end{remark}

We next control the contribution of proper powers.  Let $G^\circ$ be the cyclic
core of Lemma~\ref{L:genus-2-edge-walk}, and write $W^\circ$ for the restriction
of $W$ to $G^\circ$.  The weighted edge matrices of $G$ and $G^\circ$ have the
same nonzero spectrum by Theorem~\ref{T:DertminantL}(a).

\begin{lemma}[\textbf{Powered-weight spectral gap}]\label{L:powered-weight-gap}
Assume that $g\geq2$ and let $W$ be a positive directed weight.  For every
integer $r\geq2$,
$$
\rho(B_{W^r})^{1/r}
\leq \rho(B_{W^2})^{1/2}
<\rho(B_W).
$$
\end{lemma}

\begin{proof}
Put
$$
M:=B^\circ_{W^\circ}=(m_{\ve\ve'})_{\ve,\ve'\in\vE(G^\circ)}.
$$
By Lemma~\ref{L:genus-2-edge-walk}, every pair of oriented edges of
$G^\circ$ is joined by a nonbacktracking edge-walk.  Since $W^\circ$ is
positive, the positive entries of powers of $M$ record exactly the possible
nonbacktracking edge-walks.  Hence $M$ is irreducible.

For $s>0$, write
$$
M^{\circ s}:=(m_{\ve\ve'}^s)_{\ve,\ve'\in\vE(G^\circ)}
$$
for the entrywise $s$-th power.  It has the same zero pattern as $M$ and is
therefore irreducible.  Moreover, for every integer $r\geq1$,
$$
M^{\circ r}=B^\circ_{(W^\circ)^r}.
$$
For a positive vector $v=(v_i)$, use the corresponding notation
$v^{\circ s}:=(v_i^s)$.

At least one row of $M$ has at least two positive entries.  Indeed, if every
oriented edge had only one possible nonbacktracking successor, then every
vertex of $G^\circ$ would have degree two.  Since $G^\circ$ is connected, it
would be a cycle and would have genus one, contrary to $g\geq2$.  Fix such a
row $i_0$.

Let $x>0$ be a right Perron eigenvector of $M$, so
$Mx=\rho(M)x$.  For every $s>1$, the elementary inequality
$\sum_j a_j^s\leq(\sum_j a_j)^s$ for nonnegative numbers $a_j$ gives, for
every row $i$,
$$
\bigl(M^{\circ s}x^{\circ s}\bigr)_i
=\sum_j(m_{ij}x_j)^s
\leq\left(\sum_j m_{ij}x_j\right)^s
=\rho(M)^s x_i^s.
$$
In the row $i_0$, at least two of the numbers $m_{i_0j}x_j$ are positive.
For positive $a,b$ and $s>1$, one has $(a+b)^s>a^s+b^s$; hence the preceding
inequality is strict when $i=i_0$.

Let $y>0$ be a left Perron eigenvector of $M^{\circ s}$.  Multiplication by
$y^T$ gives a strict inequality because $y_{i_0}>0$:
$$
\rho(M^{\circ s})\,y^Tx^{\circ s}
=y^TM^{\circ s}x^{\circ s}
<\rho(M)^s y^Tx^{\circ s}.
$$
Consequently $\rho(M^{\circ s})<\rho(M)^s$ for every $s>1$.  Taking
$s=2$ yields
$$
\rho(B^\circ_{(W^\circ)^2})^{1/2}<\rho(B^\circ_{W^\circ}).
$$

It remains to compare the different powered weights.  Fix $r\geq2$, put
$t=r/2\geq1$, and let $z>0$ be a right Perron eigenvector of
$M^{\circ2}$.  Applying the same elementary inequality with exponent $t$
gives
\begin{align*}
\bigl(M^{\circ r}z^{\circ t}\bigr)_i
&=\sum_j(m_{ij}^2z_j)^t\\
&\leq\left(\sum_jm_{ij}^2z_j\right)^t
=\rho(M^{\circ2})^t z_i^t.
\end{align*}
Let $q>0$ be a left Perron eigenvector of $M^{\circ r}$.  Multiplying the
coordinatewise inequality by $q^T$ gives
$$
\rho(M^{\circ r})\,q^Tz^{\circ t}
=q^TM^{\circ r}z^{\circ t}
\leq\rho(M^{\circ2})^t q^Tz^{\circ t}.
$$
Since $q^Tz^{\circ t}>0$, division by this scalar gives
$$
\rho(M^{\circ r})\leq\rho(M^{\circ2})^t.
$$
The choice $t=r/2$ has two roles: the assumption $r\geq2$ gives $t\geq1$,
which is needed for the power-sum inequality above, while $t/r=1/2$.
Therefore, taking $r$-th roots proves
$$
\rho(M^{\circ r})^{1/r}\leq\rho(M^{\circ2})^{1/2}.
$$

Finally, Theorem~\ref{T:DertminantL}(a), applied to $W$, $W^2$, and $W^r$,
shows that the corresponding weighted edge matrices of $G$ and $G^\circ$
have the same nonzero spectra and hence the same spectral radii.  The two
inequalities on the cyclic core therefore give the asserted inequalities on
$G$.
\end{proof}

\begin{lemma}[\textbf{Negligibility of proper powers}]\label{L:proper-powers}
Assume $g\geq2$, let $W$ be a positive directed weight, and let $\Lambda$ be a
full-rank sublattice.  There exists
$0<\delta_{G,W}<\rho(B_W)$ such that, uniformly for
$\alpha\in H_1(G,\mbbZ)$,
\begin{align*}
\pi_W(\alpha,l)
&=\frac{N_W(\alpha,l)}{l}+O(\delta_{G,W}^l),\\
\Pi_W(\alpha,l)
&=\pi_W(\alpha,l)+O(\delta_{G,W}^l).
\end{align*}
Moreover, uniformly for $\underline{\alpha}\in Q_\Lambda$,
\begin{align*}
\pi_W(\underline{\alpha},l)
&=\frac{N_W(\underline{\alpha},l)}{l}+O(\delta_{G,W}^l),\\
\Pi_W(\underline{\alpha},l)
&=\pi_W(\underline{\alpha},l)+O(\delta_{G,W}^l).
\end{align*}
\end{lemma}

\begin{proof}
Put
$$
\delta_0:=\rho(B_{W^2})^{1/2}.
$$
Lemma~\ref{L:powered-weight-gap} gives
$$
\rho(B_{W^r})^{1/r}\leq\delta_0<\rho(B_W)
\qquad (r\geq2).
$$
Choose $\delta_{G,W}$ with
$\delta_0<\delta_{G,W}<\rho(B_W)$.

We first bound the total weight of the circuits arising from an $r$-th
power.  If $r\geq2$ divides $l$, Lemma~\ref{L:trace-def} gives
$$
N_{W^r}(l/r)=\operatorname{tr}(B_{W^r}^{l/r}).
$$
The matrix $B_{W^r}$ has size $2m$, and every eigenvalue has modulus at most
$\rho(B_{W^r})$.  Since the trace above is nonnegative, it follows that
\begin{align*}
N_{W^r}(l/r)
&\leq \sum_{\lambda\in\spec B_{W^r}}|\lambda|^{l/r}\\
&\leq 2m\,\rho(B_{W^r})^{l/r}
\leq 2m\,\delta_0^l.
\end{align*}

Fix $\alpha\in H_1(G,\mbbZ)$.  Separating the term $r=1$ in the Möbius
formula of Proposition~\ref{P:IdenCountFunc} gives
\begin{align*}
&\pi_W(\alpha,l)-\frac{N_W(\alpha,l)}l\\
&\qquad=\frac1l
\sum_{\substack{r\mid l\\r\geq2}}\mu(r)
\sum_{\beta\in[\alpha\mid r]}
N_{W^r}\left(\beta,\frac lr\right).
\end{align*}
For fixed $r$, the sets of circuits indexed by the classes $\beta$ are
disjoint.  Positivity therefore gives
$$
\sum_{\beta\in[\alpha\mid r]}
N_{W^r}\left(\beta,\frac lr\right)
\leq N_{W^r}(l/r).
$$
Using $|\mu(r)|\leq1$, the preceding total trace bound, and the fact that
$l$ has at most $l$ positive divisors, we obtain
$$
\left|\pi_W(\alpha,l)-\frac{N_W(\alpha,l)}l\right|
\leq\frac{2m}{l}\sum_{r\mid l}\delta_0^l
\leq2m\,\delta_0^l
=O(\delta_{G,W}^l),
$$
uniformly in $\alpha$.

For the cycle sum, the first identity in
Proposition~\ref{P:IdenCountFunc} gives
$$
\Pi_W(\alpha,l)-\pi_W(\alpha,l)
=\sum_{\substack{r\mid l\\r\geq2}}
\sum_{\beta\in[\alpha\mid r]}
\pi_{W^r}\left(\beta,\frac lr\right).
$$
A prime cycle of length $n$ gives $n$ based circuits of length $n$, all with
the same weight.  Hence
$$
n\,\pi_{W^r}(\beta,n)\leq N_{W^r}(\beta,n).
$$
With $n=l/r$, summing over the relevant $\beta$ and then over $r$ yields
\begin{align*}
0\leq\Pi_W(\alpha,l)-\pi_W(\alpha,l)
&\leq\sum_{\substack{r\mid l\\r\geq2}}
\frac r l\,N_{W^r}(l/r)\\
&\leq2m l\,\delta_0^l
=O(\delta_{G,W}^l).
\end{align*}
The last estimate follows because
$l(\delta_0/\delta_{G,W})^l$ is bounded.  This proves the two
integral-homology estimates uniformly in $\alpha$.

The quotient classes in $Q_\Lambda$ also partition the circuits.  Replacing
$\alpha$ and $\beta$ above by $\underline{\alpha}$ and
$\underline{\beta}$ therefore gives the same two bounds, independent of
$\underline{\alpha}$.  This proves the quotient-homology estimates.
\end{proof}

Recall that $\nu_G$ is the period of $G$.  For a full-rank sublattice
$\Lambda$, Theorem~\ref{T:SpecDuality}(SR4)--(SR5) shows that, for every
positive directed weight $W$, the following set is independent of $W$:
$$
\mathcal E_G(\Lambda):=
\{\underline{\omega}\in\widehat{Q_\Lambda}:
\rho(B_{\omega,W})=\rho(B_W)\}
=
\begin{cases}
\{0,\eta\},&\text{if $\eta$ exists and belongs to $\widehat{Q_\Lambda}$,}\\
\{0\},&\text{otherwise.}
\end{cases}
$$

\begin{theorem}[\textbf{Weighted prime-cycle distribution}]\label{T:counting}
Assume $g\geq2$, let $W$ be a positive directed weight, and let
$\underline{\alpha}\in Q_\Lambda$.
\begin{enumerate}[(a)]
\item If $\nu_G\nmid l$, then
$$
N_W(\underline{\alpha},l)=\pi_W(\underline{\alpha},l)
=\Pi_W(\underline{\alpha},l)=0.
$$
If $\eta$ exists and belongs to $\widehat{Q_\Lambda}$, the same vanishing
holds whenever
$$
\eta(\underline{\alpha})\neq(-1)^{l/\nu_G}.
$$
\item Suppose $\mathcal E_G(\Lambda)=\{0\}$.  Along integers satisfying
$\nu_G\mid l$,
$$
\pi_W(\underline{\alpha},l)
\sim\Pi_W(\underline{\alpha},l)
\sim\frac{N_W(\underline{\alpha},l)}{l}
\sim\frac{\nu_G\rho(B_W)^l}{l|Q_\Lambda|}.
$$
\item Suppose that $\eta$ exists and belongs to $\widehat{Q_\Lambda}$.
Along integers satisfying
$\nu_G\mid l$ and
$\eta(\underline{\alpha})=(-1)^{l/\nu_G}$,
$$
\pi_W(\underline{\alpha},l)
\sim\Pi_W(\underline{\alpha},l)
\sim\frac{N_W(\underline{\alpha},l)}{l}
\sim\frac{2\nu_G\rho(B_W)^l}{l|Q_\Lambda|}.
$$
\end{enumerate}
\end{theorem}

\begin{proof}
The edge matrix of the cyclic core is a nonnegative irreducible matrix with
period $\nu_G$.  Therefore every closed nonbacktracking edge-walk has length
divisible by $\nu_G$, and there are no circuits of length $l$ when
$\nu_G\nmid l$.  Hence $N_W(l)=0$ for such $l$; positivity and
Remark~\ref{R:vanishing} then give the same vanishing for
$N_W(\underline{\alpha},l)$, $\pi_W(\underline{\alpha},l)$, and
$\Pi_W(\underline{\alpha},l)$.  If $\eta$ exists and belongs to
$\widehat{Q_\Lambda}$, Definition~\ref{D:period-character} gives
$$
\eta(C^{\ab})=(-1)^{l/\nu_G}
$$
for every length-$l$ circuit.  Thus no such circuit can represent
$\underline{\alpha}$ when
$\eta(\underline{\alpha})\neq(-1)^{l/\nu_G}$.  Positivity again gives
the three asserted vanishings.  This proves (a).

For the asymptotics, Perron--Frobenius theory for an irreducible matrix of
period $\nu_G$ gives, along multiples of $\nu_G$,
$$
\mcalK_W(0,l)=N_W(l)
=\nu_G\rho(B_W)^l+O(\sigma^l)
$$
for some $\sigma<\rho(B_W)$.  If
$\underline{\omega}\in\widehat{Q_\Lambda}$ is not extremal, then
$\rho(B_{\omega,W})<\rho(B_W)$; by Lemma~\ref{L:trace-def}(b),
$$
\mcalK_W(\underline{\omega},l)
=\sum_{\lambda\in\spec B_{\omega,W}}\lambda^l
=O(\sigma_\omega^l)
$$
for some $\sigma_\omega<\rho(B_W)$.  Since $\widehat{Q_\Lambda}$ is finite,
we may choose one $\sigma'<\rho(B_W)$ which works for all non-extremal
characters in this finite group.

If $\mathcal E_G(\Lambda)=\{0\}$, finiteness of $\widehat{Q_\Lambda}$ and
the finite Fourier inversion in Theorem~\ref{T:Fourier}(b) give
\begin{align*}
N_W(\underline{\alpha},l)
&=\frac1{|Q_\Lambda|}
\sum_{\underline{\omega}\in\widehat{Q_\Lambda}}
\chi_{-\underline{\omega}}(\underline{\alpha})
\mcalK_W(\underline{\omega},l)\\
&=\frac{\mcalK_W(0,l)}{|Q_\Lambda|}+O((\sigma')^l)\\
&=\frac{\nu_G\rho(B_W)^l}{|Q_\Lambda|}+O(\sigma_1^l)
\end{align*}
for some $\sigma_1<\rho(B_W)$.  Lemma~\ref{L:proper-powers} converts this
estimate into the asserted asymptotics for $\pi_W$ and $\Pi_W$.

In (c), Theorem~\ref{T:SpecDuality}(SR4)--(SR5) gives
$\mathcal E_G(\Lambda)=\{0,\eta\}$, while (PC2) gives
$$
\mcalK_W(\eta,l)
=e\left(\frac{l}{2\nu_G}\right)\mcalK_W(0,l)
=(-1)^{l/\nu_G}\mcalK_W(0,l)
$$
when $\nu_G\mid l$.  Fourier inversion gives
\begin{align*}
N_W(\underline{\alpha},l)
&=\frac1{|Q_\Lambda|}
\left(\mcalK_W(0,l)
 +\eta(\underline{\alpha})\mcalK_W(\eta,l)\right)
+O(\sigma_2^l)\\
&=\frac{1+\eta(\underline{\alpha})(-1)^{l/\nu_G}}{|Q_\Lambda|}
\nu_G\rho(B_W)^l+O(\sigma_2^l),
\end{align*}
where $\sigma_2<\rho(B_W)$ and we used the fact that $\eta$ is
$2$-torsion.  Under the stated parity condition the leading factor is $2$.
Lemma~\ref{L:proper-powers}
again gives the asymptotics for $\pi_W$ and $\Pi_W$.
\end{proof}

We specialize to $\Lambda=tH_1(G,\mbbZ)$.  Then
$Q_\Lambda=H_1(G,\mbbZ/t\mbbZ)$ has order $t^g$, and its elements are the
$t$-circulations on $G$.

\begin{corollary}[\textbf{Weighted $t$-circulations}]\label{C:weighted-t-circulations}
Assume $g\geq2$, let $W$ be a positive directed weight, and let
$\underline{\alpha}\in H_1(G,\mbbZ/t\mbbZ)$.
\begin{enumerate}[(a)]
\item If $t$ is odd, then, along $\nu_G\mid l$,
$$
\pi_W(\underline{\alpha},l)
\sim\Pi_W(\underline{\alpha},l)
\sim\frac{N_W(\underline{\alpha},l)}l
\sim\frac{\nu_G\rho(B_W)^l}{lt^g}.
$$
\item Suppose $t$ is even and the period character does not exist.  Then the
same uniform asymptotic holds along $\nu_G\mid l$.
\item Suppose $t$ is even and the period character $\eta$ exists.  All
three sums vanish when
$$
\eta(\underline{\alpha})\neq(-1)^{l/\nu_G};
$$
along the remaining integers with $\nu_G\mid l$,
$$
\pi_W(\underline{\alpha},l)
\sim\Pi_W(\underline{\alpha},l)
\sim\frac{N_W(\underline{\alpha},l)}l
\sim\frac{2\nu_G\rho(B_W)^l}{lt^g}.
$$
\end{enumerate}
\end{corollary}

\begin{proof}
For $\Lambda=tH_1(G,\mbbZ)$, we have
$Q_\Lambda=H_1(G,\mbbZ/t\mbbZ)$ and $|Q_\Lambda|=t^g$.  Its dual is the
subgroup of unitary characters whose order divides $t$.  Hence, when $t$ is
odd, this dual group has no nonzero $2$-torsion; when $t$ is even, it contains
all $2$-torsion characters.

Since $\eta$ is nontrivial and $2$-torsion whenever it exists, it does not
belong to the finite dual group when $t$ is odd.  Thus
$\mathcal E_G(tH_1(G,\mbbZ))=\{0\}$ and
Theorem~\ref{T:counting}(b) proves (a).  When $t$ is even, the period
character belongs to the finite dual exactly when it exists.  Theorem~
\ref{T:counting}(b) then gives (b), while parts (a) and (c) of that theorem
give the vanishing and doubled asymptotic in (c).
\end{proof}

The unweighted regular-bipartite specialization is stated separately in
Appendix~\ref{C:regular-specializations}.

\section{Examples}\label{S:examples}

The examples are organized around three questions suggested by the main
results.  First, how do the canonical and period characters appear in an
actual spectral-radius landscape?  Three genus-$2$ theta graphs exhibit the
different extremal-character loci allowed by Theorem~\ref{T:SpecDuality}; a
comparison of the trivial and nontrivial weights shows that the surfaces and
their maximal values change, while these distinguished loci do not.  Second,
what remains of the outer-edge theorem when a periodic graph samples only a
proper part of the character torus?  A $\mbbZ^2$-periodic graph with a
genus-$3$ quotient realizes the affine-subtorus membership conditions in
Theorem~\ref{T:outer-spectral-edges}.  Third, how do the spectral identities
pass into weighted cycle counting?  On $K_4$, we follow the trace formulas
through Fourier inversion to the resulting cycle sums and their exponential
growth rates.

\subsection{Weighted spectral radii and trace distributions on
\texorpdfstring{genus-$2$ theta graphs}{genus-2 theta graphs}} \label{SS:example_spec}
We consider a particular family of genus-$2$ graphs, rather than all such
graphs.  Let $G$ be a theta graph formed by three internally vertex-disjoint
paths $\Delta_0$, $\Delta_1$ and $\Delta_2$ joining two distinct vertices
$v$ and $w$ (Figure~\ref{F:genus2graph}).  Thus $G$ has neither loops nor
vertices of degree one.  Let $l_j\geq1$ be the length of $\Delta_j$ for
$j=0,1,2$. Let $v_1$ and $v_2$ be the vertices adjacent to $v$ on $\Delta_1$ and $\Delta_2$, respectively, and let $\ve_1$ and $\ve_2$ be the oriented edges from $v$ to $v_1$ and $v_2$. Then $\{\phi_1(d\ve_1), \phi_1(d\ve_2)\}$ is a basis of $H_1(G,\mbbZ)^\vee$ in $\mcalH^1(G)$ (Proposition~\ref{P:basis_gen}), and the character group $\mcalX=\mcalH^1(G)/H_1(G,\mbbZ)^\vee$ is a $2$-dimensional real torus. Each $\underline{\omega}\in \mcalX$ has a unique representative
$$
\omega=\omega_1\cdot\phi_1(d\ve_1)+\omega_2\cdot\phi_1(d\ve_2),
\qquad 0\leq\omega_1,\omega_2<1.
$$

We consider three cases of this type of graphs: a non-bipartite graph $G_1$ with $l_0=1$, $l_1=2$ and $l_2=3$; a bipartite graph $G_2$ with $l_0=1$, $l_1=3$ and $l_2=5$; and a bipartite graph $G_3$ with $l_0=2$, $l_1=2$ and $l_2=4$. 

Besides the trivial weight, we use the same positive symmetric weight $W^\star$ on all three graphs.  Put
$$
(a_0,a_1,a_2)=\left(1,2,\frac12\right)
$$
and, for either orientation $\ve$ of an edge belonging to $\Delta_j$, define
$$
W^\star(\ve):=a_j^{1/l_j}.
$$
Thus the total weight in either direction along $\Delta_j$ is $W^\star(\Delta_j)=a_j$.  Normalizing by the path length makes the three displayed weighted examples directly comparable even though their path lengths differ.

\begin{figure}[H]
\centering
\begin{tikzpicture}[x=1cm,y=1cm]
\coordinate (v) at (0,4.3);
\coordinate (w) at (0,-0.3);

\coordinate (v0) at (0,3);
\coordinate (v01) at (0,2.5);
\coordinate (w0) at (0,1);
\coordinate (w01) at (0,1.5);
\coordinate (vw0) at (0,2.1);

\coordinate (v1) at (1,3);
\coordinate (v11) at (1.05,2.5);
\coordinate (w1) at (1,1);
\coordinate (w11) at (1.05,1.5);
\coordinate (vw1) at (1.05,2.1);

\coordinate (v2) at (2,3);
\coordinate (v21) at (2.1,2.5);
\coordinate (w2) at (2,1);
\coordinate (w21) at (2.1,1.5);
\coordinate (vw2) at (2.1,2.1);

\path[-,font=\scriptsize,  line width=1.5pt, black]
(v) edge[out=-90,in=90]   (v0)
(w) edge[out=90,in=-90]  (w0)
(v0) edge[out=-90,in=90]   (v01)
(w0) edge[out=90,in=-90] (w01);

\path[-,font=\scriptsize,  line width=1.5pt, black]
(v) edge[out=-15,in=100, gray] node[pos=0.7,left,black]{\large $\ve_1$} node[pos=0.65,sloped,allow upside down]{\midarrow}  (v1)
(w) edge[out=15,in=-100]    (w1)
(v1) edge[out=-80,in=90]   (v11)
(w1) edge[out=80,in=-90] (w11);

\path[-,font=\scriptsize,  line width=1.5pt, black]
(v) edge[out=-5,in=110, gray] node[pos=0.55, below,black]{\large $\ve_2$}  node[pos=0.6,sloped,allow upside down]{\midarrow}  (v2)
(w) edge[out=5,in=-110]  (w2)
(v2) edge[out=-70,in=90]   (v21)
(w2) edge[out=70,in=-90] (w21);

\fill [black] (v) circle (2.5pt);
    \draw (v) node[anchor=east] {\large $v$};
\fill [black] (w) circle (2.5pt);
    \draw (w) node[anchor=east] {\large $w$};
\fill [black] (v0) circle (2.5pt);
\fill [black] (w0) circle (2.5pt);
\draw (vw0) node {\LARGE $\vdots$};
\draw (vw0)+(-0.1,-0.1)  node[anchor=east] {\large $\Delta_0$};

\fill [black] (v1) circle (2.5pt);
	\draw (v1) node[anchor=east] {\large $v_1$};
\fill [black] (w1) circle (2.5pt);
\draw (vw1) node {\LARGE $\vdots$};
\draw (vw1)+(-0.1,-0.1) node[anchor=east] {\large $\Delta_1$};

\fill [black] (v2) circle (2.5pt);
	\draw (v2) node[anchor=east] {\large $v_2$};
\fill [black] (w2) circle (2.5pt);
\draw (vw2) node {\LARGE $\vdots$};
\draw (vw2)+(-0.1,-0.1)  node[anchor=east] {\large $\Delta_2$};
\end{tikzpicture}
\caption{A genus-$2$ graph: two vertices $v$ and $w$ are connected by three paths $\Delta_0$, $\Delta_1$ and $\Delta_2$ of length $l_0$, $l_1$, and $l_2$ respectively. } \label{F:genus2graph}
\end{figure}

\subsubsection*{Distribution of $\rho(B_{\omega})$}

The upper panels of Figure~\ref{F:SpecRad}(a)--(c) show the distributions of the unweighted spectral radii $\rho(B_{\omega})$ over $\mcalX(G_1)$, $\mcalX(G_2)$ and $\mcalX(G_3)$, respectively.  In every panel the horizontal and vertical axes are the coordinates $\omega_1$ and $\omega_2$ with respect to $\phi_1(d\ve_1)$ and $\phi_1(d\ve_2)$.

\begin{figure}[H]
\centering
\includegraphics[width=.84\textwidth]{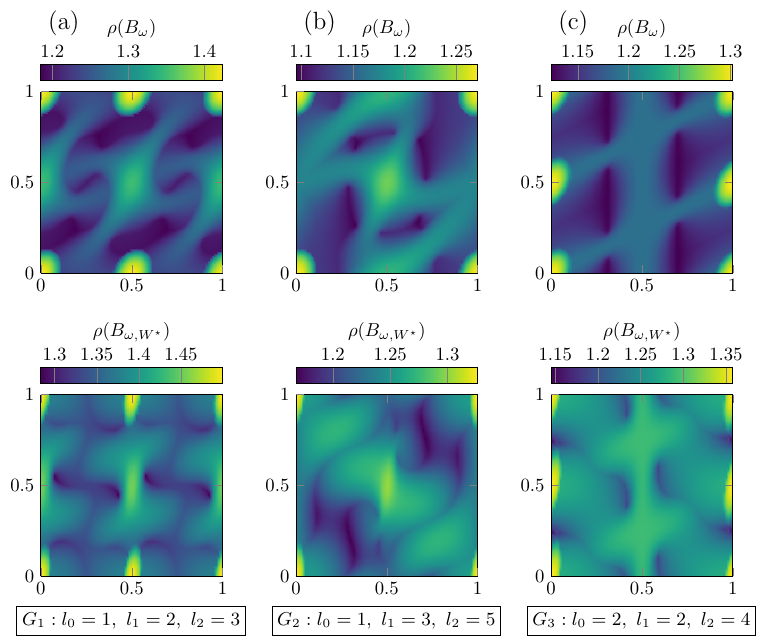}
\caption{Distributions of the spectral radii for $G_1$, $G_2$ and $G_3$.  Parts (a)--(c) correspond to these three graphs, respectively.  In each part, the upper panel is the unweighted distribution $\rho(B_\omega)$ and the lower panel is the corresponding weighted distribution $\rho(B_{\omega,W^\star})$, where the path products of $W^\star$ are $(1,2,1/2)$.  The horizontal and vertical axes are the coordinates $\omega_1$ and $\omega_2$, respectively.}\label{F:SpecRad}
\end{figure}

For the graph $G_1$, spectral antisymmetry (Theorem~\ref{T:SpecDuality}) gives a unique canonical character $\theta_{G_1}\in\mcalX(G_1)$, which is a nontrivial $2$-torsion character because $G_1$ is non-bipartite.  Construction~\ref{Cs:CanoForm2} gives the coordinates $(0.5,0)$ for $\theta_{G_1}$, as can also be observed in Figure~\ref{F:SpecRad}(a). In particular, $\rho(B_{\omega})=\rho(B)$ (of value approximately $1.42405$) if $\underline{\omega}\in\{0,\theta_{G_1}\}$, and $\rho(B_{\omega})<\rho(B)$ otherwise.  The distribution of $\rho(B_{\omega})$ is symmetric with respect to both the trivial and canonical characters, as expected from Lemma~\ref{L:PropAdj}(c) and Theorem~\ref{T:SpecDuality}(SA1).

Graphs $G_2$ and $G_3$ are both bipartite, so their canonical characters are
trivial.  Their edge periods are nevertheless the same:
$$
\nu_{G_2}=\gcd(4,6,8)=2,
\qquad
\nu_{G_3}=\gcd(4,6,6)=2.
$$
For $G_2$, normalized length parity on two basis cycles of lengths $4$ and
$6$ has values $+1$ and $-1$, while the homological difference of those
cycles has length $8$ and would have to receive both $-1$ and $+1$.
Consequently the period character does not exist, and (SR5) gives
$\rho(B_\omega)<\rho(B)\approx1.27065$ for every nontrivial character.
For $G_3$, the corresponding cycle lengths are $4,6,6$, so normalized length
parity is compatible with the homology relation.  The period character exists
and has coordinates $(0,0.5)$.  Hence
$\rho(B_{\eta_{G_3}})=\rho(B)\approx1.30216$, while the inequality is strict
at every other nontrivial character.  These two alternatives are visible in
the upper panels of Figure~\ref{F:SpecRad}(b) and (c).

The lower panels of Figure~\ref{F:SpecRad}(a)--(c) repeat the calculation with $W^\star$.  The weights change both the surfaces and their maximal values:
$$
\rho(B_{0,W^\star})\approx1.49947,\quad1.32633,\quad1.35804
$$
for $G_1,G_2,G_3$, respectively.  They do not change the loci of the maxima.  These loci are $\{0,\theta_{G_1}\}$ for $G_1$, $\{0\}$ for $G_2$, and $\{0,\eta_{G_3}\}$ for $G_3$, exactly as in the unweighted panels.  This agreement is the weight-independence asserted in Theorem~\ref{T:SpecDuality}(SR4)--(SR5), rather than a numerical accident.  We use a symmetric weight here only so that the weighted trace distributions below are real-valued.

The complete edge spectra for $G_3$ make the rotation in (PC2) explicit.
For a polynomial $p$, write $\mathsf Z(p)$ for its multiset of complex roots,
counting multiplicity.  In the trivial-weight case,
\begin{align*}
\spec B
={}&\{1,1,-1,-1,\sqrt{-1},-\sqrt{-1}\}
\sqcup\mathsf Z(z^4+2z^2+2)
\sqcup\mathsf Z(z^6-z^4-2),\\
\spec B_{\eta_{G_3}}
={}&\{1,-1,\sqrt{-1},\sqrt{-1},-\sqrt{-1},-\sqrt{-1}\}
\sqcup\mathsf Z(z^4-2z^2+2)
\sqcup\mathsf Z(z^6+z^4+2).
\end{align*}
Thus all sixteen eigenvalues, including multiplicities, satisfy
$$
\spec B_{\eta_{G_3}}=\sqrt{-1}\,\spec B.
$$
For the weighted example, put
$$
q_\pm(z):=2z^8-4z^4-3z^2\pm4.
$$
Direct factorization of the two characteristic polynomials gives
\begin{align*}
\spec B_{0,W^\star}
&=\mathsf Z(q_-)\sqcup\mathsf Z(q_+),\\
\spec B_{\eta_{G_3},W^\star}
&=\mathsf Z(q_-(-\sqrt{-1}z))
  \sqcup\mathsf Z(q_+(-\sqrt{-1}z))
 =\sqrt{-1}\,\spec B_{0,W^\star}.
\end{align*}
Equivalently, the two weighted characteristic polynomials are
\begin{align*}
\det(zI-B_{0,W^\star})
&=\frac14(2z^8-4z^4-3z^2-4)(2z^8-4z^4-3z^2+4),\\
\det(zI-B_{\eta_{G_3},W^\star})
&=\frac14(2z^8-4z^4+3z^2-4)(2z^8-4z^4+3z^2+4).
\end{align*}

\raggedbottom
\subsubsection*{Distribution of $\mcalK(\omega,l)$}

Figure~\ref{F:Kappa} compares the distributions of the unweighted traces $\mcalK(\omega,l)$ with the corresponding weighted traces $\mcalK_{W^\star}(\omega,l)$ over the character groups of graphs $G_1$, $G_2$ and $G_3$ (see Subsection~\ref{SS:trace}).  For each graph, the unweighted and weighted cases are placed in consecutive rows so that panels in the same column have the same length.

For $G_1$ (the unweighted row of Figure~\ref{F:Kappa}(a)), we show the cases of $l=3,4,20,21$. In particular, $\mcalK(0,3)=6$, $\mcalK(0,4)=8$, $\mcalK(0,20)=1278$, and $\mcalK(0,21)=1574$, which correspond to the number of circuits of length $3$, $4$, $20$ and $21$ respectively (Lemma~\ref{L:Kappa}(a)). One may observe that the distribution of $\mcalK(\omega,l)$ is symmetric with respect to the trivial character for all cases (Lemma~\ref{L:Kappa}(b)). 
As shown previously, the canonical character $\theta_{G_1}$ has coordinates $(0.5,0)$. The distribution of $\mcalK(\omega,l)$ is symmetric with respect to $\theta_{G_1}/2$ at $(0.25,0)$ when $l=4,20$, and antisymmetric with respect to $\theta_{G_1}/2$ when $l=3,21$ (Lemma~\ref{L:Kappa}(c)).  For small $l$ ($3$ and $4$ here), it is possible that $|\mcalK(\omega,l)|=\mcalK(0,l)$ for $\underline{\omega}\notin \{0,\theta_{G_1}\}$: in the case $l=3$, $\mcalK(\omega,3)=6$ for all characters $\underline{\omega}$ with the first coordinate $0$, and $\mcalK(\omega,3)=-6$ for all characters $\underline{\omega}$ with the first coordinate $0.5$; in the case $l=4$, $\mcalK(\omega,4)=8$ for all characters $\underline{\omega}$ with the second coordinate $0$, and $\mcalK(\omega,4)=-8$ for all characters $\underline{\omega}$ with the second coordinate $0.5$. For $G_1$, the matrix $B$ has the unique dominant eigenvalue $\rho(B)$.  Theorem~\ref{T:SpecDuality} therefore gives the unique dominant eigenvalue $-\rho(B)$ of $B_{\omega}$ when $\underline{\omega}=\theta_{G_1}$, while $\rho(B_{\omega})<\rho(B)$ for every $\underline{\omega}\notin\{0,\theta_{G_1}\}$.  Hence, as $l\to\infty$, $\mcalK(0,l)\sim \rho(B)^l$ and $\mcalK(\theta_{G_1},l)=(-1)^l\mcalK(0,l)\sim (-1)^l\rho(B)^l$.  Consequently, for sufficiently large $l$ ($20$ and $21$ here), the equality $|\mcalK(\omega,l)|=\mcalK(0,l)$ can hold only when $\underline{\omega}\in\{0,\theta_{G_1}\}$.

The factor $(-1)^l$ shows that opposite colorbar endpoints are forced only
for odd $l$; for even $l$, spectral antisymmetry gives a symmetry of the trace
function on the character torus, not of its range about zero.

\begin{figure}[H]
\centering
\includegraphics[width=.82\textwidth]{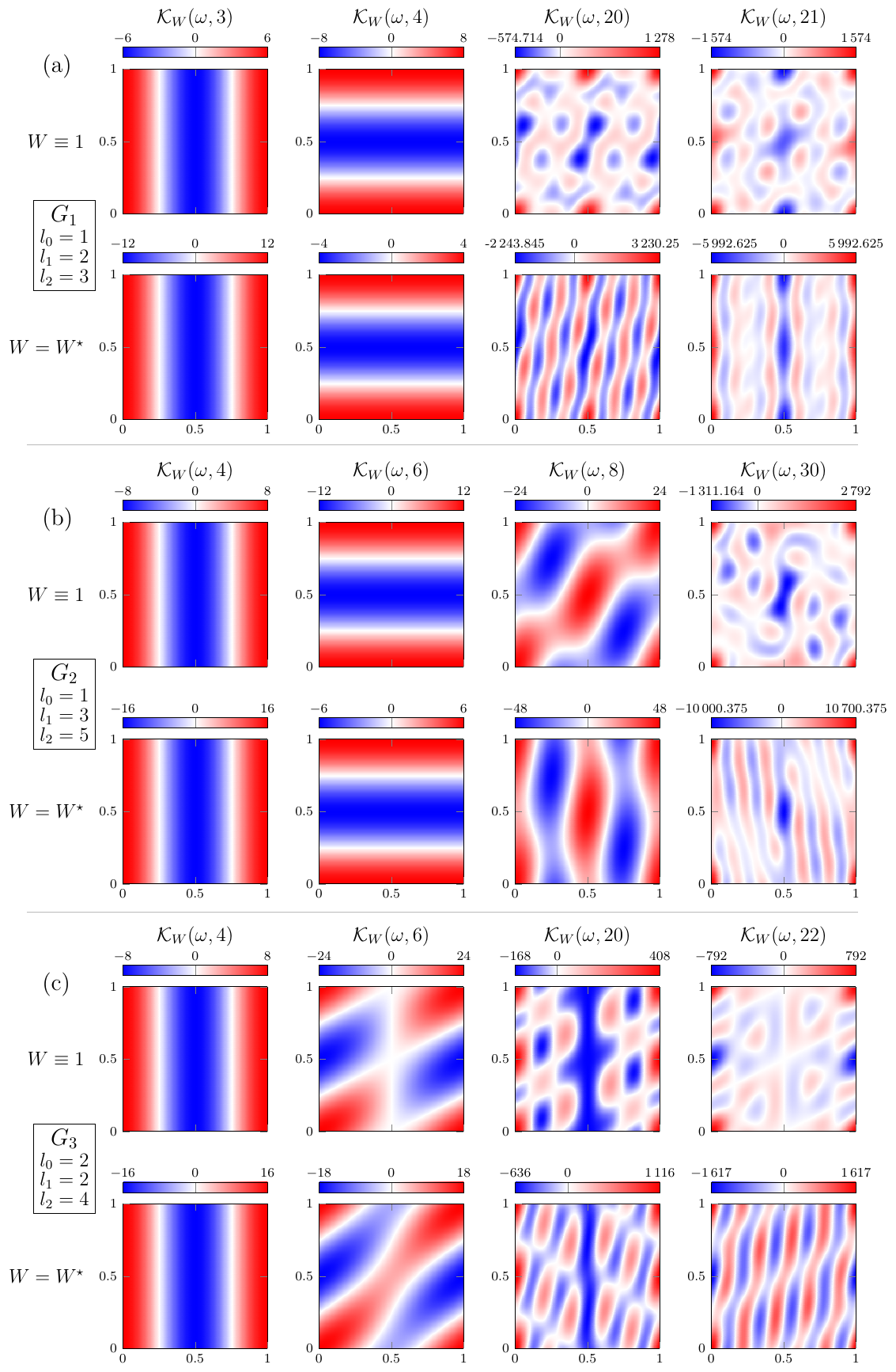}
\caption{Distributions of the unweighted and weighted traces over the character groups of $G_1$, $G_2$ and $G_3$.  Parts (a)--(c) correspond to $G_1,G_2,G_3$, respectively.  In each part, the first row has $W\equiv1$ and the second has $W=W^\star$, where the path products of $W^\star$ are $(1,2,1/2)$; a shared heading $\mcalK_W(\omega,l)$ applies to each vertical pair.  Every colorbar marks its actual minimum, zero, and its actual maximum.  The horizontal and vertical axes are the coordinates $\omega_1$ and $\omega_2$.}\label{F:Kappa}
\end{figure}

Graphs $G_2$ and $G_3$ are bipartite. Then the corresponding canonical characters are both trivial, and by Lemma~\ref{L:Kappa}(d), $\mcalK(\omega,l)$ vanishes for all characters $\underline{\omega}$ whenever $l$ is odd. Therefore, here we only show cases for even $l$.  For $G_2$ and $G_3$ (the unweighted rows of Figure~\ref{F:Kappa}(b) and (c)), we show the cases $l=4,6,8,30$ and $l=4,6,20,22$, respectively. Recall from Figure~\ref{F:SpecRad} that for $G_2$, $\rho(B_{\omega})<\rho(B)$ for every nontrivial character $\underline{\omega}$.  For $G_3$, however, the period character $\eta_{G_3}$ has coordinates $(0,0.5)$ and satisfies
$$
\spec B_{\eta_{G_3}}=e(1/4)\spec B=\sqrt{-1}\,\spec B.
$$
Consequently,
$$
\mcalK(\eta_{G_3},l)=(\sqrt{-1})^l\mcalK(0,l)
=(-1)^{l/2}\mcalK(0,l)
$$
for even $l$.  Thus the exceptional trace agrees with the untwisted trace at
$l=20$ and is its negative at $l=22$.  For $G_2$, by contrast, when $l$ is a
sufficiently large even number ($30$ in this example),
$|\mcalK(\omega,l)|<\mcalK(0,l)$ for every nontrivial character.

The weighted row in each part of Figure~\ref{F:Kappa} gives the corresponding traces $\mcalK_{W^\star}(\omega,l)=\operatorname{tr}(B_{\omega,W^\star}^l)$. Because $W^\star$ is symmetric, these functions are real-valued.  Their values at the trivial character are weighted circuit sums rather than cardinalities.  For $G_1$, for example, $\mcalK_{W^\star}(0,3)=12$, $\mcalK_{W^\star}(0,4)=4$, $\mcalK_{W^\star}(0,20)=3230.25$, and $\mcalK_{W^\star}(0,21)=5992.625$.  Although the path weights substantially deform the trace surfaces, the canonical relation $\mcalK_{W^\star}(\theta_{G_1}+\underline{\omega},l)=(-1)^l\mcalK_{W^\star}(\underline{\omega},l)$ remains unchanged.  On $G_3$, the equally weight-independent period relation is
$$
\mcalK_{W^\star}(\eta_{G_3},l)
=(\sqrt{-1})^l\mcalK_{W^\star}(0,l).
$$
The long-length panels are again governed by the extremal-character loci identified in Figure~\ref{F:SpecRad}.

\subsection{Outer spectral edges for a
\texorpdfstring{$\mbbZ^2$}{Z2}-periodic graph}
\label{SS:example-periodic-edges}

We now give an example in which the Bloch torus is a proper subtorus of the
character torus of the quotient graph.  Let
$$
\mathbf a_1=(1,0),\qquad \mathbf a_2=(0,1)
$$
be the standard basis of $\mbbZ^2$.  The periodic graph $\Gamma$ has vertices
$u_{\mathbf n},v_{\mathbf n}$, $\mathbf n\in\mbbZ^2$, and the following four
edge orbits:
\begin{align*}
\ve_0&:u_{\mathbf n}\longrightarrow
u_{\mathbf n+\mathbf a_1+\mathbf a_2},&
\ve_1&:u_{\mathbf n}\longrightarrow v_{\mathbf n},\\
\ve_2&:u_{\mathbf n}\longrightarrow v_{\mathbf n+\mathbf a_1},&
\ve_3&:u_{\mathbf n}\longrightarrow v_{\mathbf n+\mathbf a_2}.
\end{align*}
Thus $\ve_1,\ve_2,\ve_3$ form a honeycomb-type bipartite graph,
whereas $\ve_0$ joins two vertices in the same vertex orbit and makes
$\Gamma$ non-bipartite.

\begin{figure}[t]
\centering
\makebox[\textwidth][c]{%
\begin{minipage}[t]{7.30cm}
\centering
\begin{tikzpicture}[x=.80cm,y=.80cm]
\path[use as bounding box] (-4.55,-3.20) rectangle (4.55,3.05);

\begin{scope}[xshift=-.55cm]
\clip (-2.30,-2.55) rectangle (4.40,2.55);
\foreach \i in {-3,-2,...,3}{
  \foreach \j in {-3,-2,...,3}{
    \pgfmathsetmacro{\ux}{1.65*\i-.24}
    \pgfmathsetmacro{\uy}{1.65*\j-.24}
    \pgfmathsetmacro{\vx}{\ux-.585}
    \pgfmathsetmacro{\vy}{\uy-.585}
    \pgfmathsetmacro{\vxone}{\ux+1.065}
    \pgfmathsetmacro{\vxtwo}{\ux-.585}
    \pgfmathsetmacro{\vytwo}{\uy+1.065}
    \pgfmathsetmacro{\utopx}{\ux+1.65}
    \pgfmathsetmacro{\utopy}{\uy+1.65}
    \draw[line width=1.05pt] (\ux,\uy) -- (\vx,\vy);
    \draw[line width=1.05pt] (\ux,\uy) -- (\vxone,\vy);
    \draw[line width=1.05pt] (\ux,\uy) -- (\vxtwo,\vytwo);
    \draw[line width=1.05pt,black!45]
      (\ux,\uy) to[bend right=19] (\utopx,\utopy);
  }
}
\foreach \i in {-3,-2,...,3}{
  \foreach \j in {-3,-2,...,3}{
    \pgfmathsetmacro{\ux}{1.65*\i-.24}
    \pgfmathsetmacro{\uy}{1.65*\j-.24}
    \pgfmathsetmacro{\vx}{\ux-.585}
    \pgfmathsetmacro{\vy}{\uy-.585}
    \fill[black] (\ux,\uy) circle (2.15pt);
    \filldraw[fill=white,draw=black,line width=.8pt]
      (\vx,\vy) circle (2.15pt);
  }
}

\coordinate (u) at (-.24,-.24);
\coordinate (v0) at (-.825,-.825);
\coordinate (v1) at (.825,-.825);
\coordinate (v2) at (-.825,.825);
\coordinate (up) at (1.41,1.41);
\path[-,font=\normalsize,line width=1.5pt,black]
(u) edge
  node[pos=.55,above left=1pt,black,fill=white,inner sep=1pt]{$\ve_1$}
  node[pos=.63,sloped,allow upside down]{\midarrow}
  (v0)
(u) edge
  node[pos=.55,below=3pt,black,fill=white,inner sep=1pt]{$\ve_2$}
  node[pos=.63,sloped,allow upside down]{\midarrow}
  (v1)
(u) edge
  node[pos=.52,right=4pt,black,fill=white,inner sep=1pt]{$\ve_3$}
  node[pos=.63,sloped,allow upside down]{\midarrow}
  (v2);
\draw[font=\normalsize,line width=1.5pt,gray]
  (u) to[bend right=19]
  node[pos=.68,below=14pt,xshift=-4pt,black,fill=white,inner sep=1pt]{$\ve_0$}
  node[pos=.82,sloped,allow upside down]{\midarrow}
  (up);
\fill[black] (u) circle (2.5pt);
\filldraw[fill=white,draw=black,line width=.9pt]
  (v0) circle (2.5pt);
\filldraw[fill=white,draw=black,line width=.9pt]
  (v1) circle (2.5pt);
\filldraw[fill=white,draw=black,line width=.9pt]
  (v2) circle (2.5pt);
\fill[black] (up) circle (2.5pt);
\end{scope}

\coordinate (period-origin) at (-4.00,-2.85);
\draw[->,densely dashed,line width=.9pt,shorten >=2.5pt]
  (period-origin) -- ++(1.65,0)
  node[midway,below=2pt,fill=white,inner sep=1pt]{$\mathbf a_1$};
\draw[->,densely dashed,line width=.9pt,shorten >=2.5pt]
  (period-origin) -- ++(0,1.65)
  node[midway,left=2pt,fill=white,inner sep=1pt]{$\mathbf a_2$};

\fill[black] (-1.55,-2.88) circle (2.5pt);
\node[anchor=west] at (-1.42,-2.88) {$u_{\mathbf n}$};
\filldraw[fill=white,draw=black,line width=.9pt]
  (.55,-2.88) circle (2.5pt);
\node[anchor=west] at (.68,-2.88) {$v_{\mathbf n}$};
\end{tikzpicture}

\makebox[\linewidth][c]{\textup{(a) A periodic window of $\Gamma$}}
\end{minipage}\hspace{.45cm}%
\begin{minipage}[t]{6.12cm}
\centering
\begin{tikzpicture}[x=1.30cm,y=1.27cm]
\path[use as bounding box] (-.85,-1.70) rectangle (3.85,2.23);
\coordinate (u) at (-.28,0);
\coordinate (v) at (3.28,0);

\path[-,font=\normalsize,line width=1.5pt,black]
(u) edge
  node[pos=.50,above=3pt,fill=white,inner sep=1pt]{$\ve_2$}
  node[pos=.62,sloped,allow upside down]{\midarrow}
  (v)
(u) edge[out=35,in=145]
  node[pos=.50,above=3pt,fill=white,inner sep=1pt]{$\ve_1$}
  node[pos=.63,sloped,allow upside down]{\midarrow}
  (v)
(u) edge[out=-35,in=-145]
  node[pos=.50,below=3pt,fill=white,inner sep=1pt]{$\ve_3$}
  node[pos=.64,sloped,allow upside down]{\midarrow}
  (v);
\draw[font=\normalsize,line width=1.5pt,gray,
  decoration={markings,
    mark=at position .20 with {\arrow{latex}}},
  postaction={decorate}]
(u) to[loop above,out=60,in=120,min distance=2.35cm] (u);
\node[font=\normalsize,black,fill=white,inner sep=1pt]
  at (-.28,1.48) {$\ve_0$};

\fill[black] (u) circle (2.5pt);
\filldraw[fill=white,draw=black,line width=.9pt]
  (v) circle (2.5pt);
\draw (u) node[anchor=east] {$u$};
\draw (v) node[anchor=west] {$v$};
\end{tikzpicture}

\makebox[\linewidth][c]{\textup{(b) The quotient graph $G$}}
\end{minipage}
}
\caption{The $\mbbZ^2$-periodic graph and its voltage quotient.  Filled
vertices represent the orbit $u_{\mathbf n}$ and open vertices the orbit
$v_{\mathbf n}$.  The gray arcs in (a) are the translates of $\ve_0$.
The dashed arrows show the equal horizontal and vertical periods
$\mathbf a_1$ and $\mathbf a_2$.
The four thicker arrows mark one lifted set of the edge-orbit
representatives.  The voltages in (b) are
$\mathbf a_1+\mathbf a_2,0,\mathbf a_1,\mathbf a_2$ on
$\ve_0,\ve_1,\ve_2,\ve_3$, respectively, and
$T=\{\ve_1\}$ is the chosen spanning tree.}
\label{F:periodic-quotient}
\end{figure}

The quotient graph $G=\Gamma/\mbbZ^2$, shown in
Figure~\ref{F:periodic-quotient}(b), has two vertices, four edges, and genus
$g(G)=3$.
After deleting the loop $\ve_0$, the remaining graph is bipartite.  The
quotient $G$ itself is non-bipartite because it contains the loop $\ve_0$.
Thus this example illustrates the non-bipartite case in
Theorem~\ref{T:SpecDuality}(SA3) and (SR4).
With the spanning tree $T=\{\ve_1\}$, the fundamental cycles
$$
C_1=\ve_2\cdot\ve_1^{-1},\qquad
C_2=\ve_3\cdot\ve_1^{-1},\qquad
C_3=\ve_0
$$
form a basis of $H_1(G,\mbbZ)$.  Their displacements are
$$
\nu(C_1)=\mathbf a_1,\qquad
\nu(C_2)=\mathbf a_2,\qquad
\nu(C_3)=\mathbf a_1+\mathbf a_2.
$$
Thus, in the bases just chosen, the displacement homomorphism has matrix
$$
N=
\begin{pmatrix}
1&0&1\\
0&1&1
\end{pmatrix}.
$$

Use additive coordinates $(x_1,x_2,x_3)\in(\mbbR/\mbbZ)^3$ on
$\mcalX(G)$ dual to $C_1,C_2,C_3$.  The Bloch-subtorus description preceding
Theorem~\ref{T:outer-spectral-edges} identifies the Bloch torus with
$$
\mcalT_0=\nu^*\widehat{\mbbZ^2}
=\{(k_1,k_2,k_1+k_2):k_1,k_2\in\mbbR/\mbbZ\}.
$$
The first two fundamental cycles have length $2$, while $C_3$ has length
$1$.  Therefore the canonical character is
$$
\theta=(0,0,1/2).
$$
In particular,
$$
0\in\mcalT_0,\qquad \theta\notin\mcalT_0.
$$
This is the simplest codimension-one configuration in which a periodic
family sees one of the two distinguished characters but not the other.

Assign positive symmetric weights $c,b_1,b_2,b_3$ to the edge orbits
$\ve_0,\ve_1,\ve_2,\ve_3$, respectively.  For
$\alpha\in\mbbR/\mbbZ$, also allow a fixed unitary background having gain
$e^{2\pi\sqrt{-1}\alpha}$ on $\ve_0$ and gain $1$ on the other three
positively oriented edges.  Its affine Bloch torus is
$$
\mcalT_\alpha
=\{(k_1,k_2,k_1+k_2+\alpha):
  k_1,k_2\in\mbbR/\mbbZ\}.
$$
Notice that
$$
0\in\mcalT_\alpha\ \Longleftrightarrow\ \alpha=0,
\qquad
\theta\in\mcalT_\alpha\ \Longleftrightarrow\ \alpha=1/2.
$$

Put
$$
p(k_1,k_2):=
b_1+b_2e^{2\pi\sqrt{-1}k_1}
   +b_3e^{2\pi\sqrt{-1}k_2}.
$$
The Bloch fiber over $(k_1,k_2)\in(\mbbR/\mbbZ)^2$ is
$$
H_\alpha(k_1,k_2)=
\begin{pmatrix}
2c\cos\bigl(2\pi(k_1+k_2+\alpha)\bigr)&p(k_1,k_2)\\
\overline{p(k_1,k_2)}&0
\end{pmatrix}.
$$
Consequently, its two band functions are
\begin{align*}
\lambda_\pm^{(\alpha)}(k_1,k_2)
={}&c\cos\bigl(2\pi(k_1+k_2+\alpha)\bigr)\\
&{}\pm
\sqrt{c^2\cos^2\bigl(2\pi(k_1+k_2+\alpha)\bigr)
      +|p(k_1,k_2)|^2}.
\end{align*}
If $S:=b_1+b_2+b_3$, then
$$
A_{0,W}=
\begin{pmatrix}2c&S\\S&0\end{pmatrix},
\qquad
\rho(A_{0,W})=c+\sqrt{c^2+S^2}.
$$
Let
$$
\Sigma_\alpha:=
\bigcup_{(k_1,k_2)\in(\mbbR/\mbbZ)^2}
\spec H_\alpha(k_1,k_2).
$$
Theorem~\ref{T:outer-spectral-edges} now gives three different outer-edge
patterns on the same periodic graph.

\smallskip
\noindent\emph{Positive weights.}
For $\alpha=0$, the Bloch torus contains $0$ but not $\theta$.  Hence
$$
\sup\Sigma_0=\rho(A_{0,W}),\qquad
\inf\Sigma_0>-\rho(A_{0,W}).
$$
The upper edge is attained only at $(k_1,k_2)=(0,0)$, and the corresponding
fiber eigenvalue is simple.

\smallskip
\noindent\emph{A signed edge orbit.}
Replace the weight $c$ on $\ve_0$ by $-c$, leaving the other three weights
positive.  The sign character is
$$
\zeta_\sigma=(0,0,1/2)=\theta.
$$
Equivalently, this is the positive absolute weight with $\alpha=1/2$.
Since $\zeta_\sigma\notin\mcalT_0$ but
$\theta+\zeta_\sigma=0\in\mcalT_0$, Theorem~\ref{T:outer-spectral-edges}
gives
$$
\sup\Sigma_{1/2}<\rho(A_{0,W}),\qquad
\inf\Sigma_{1/2}=-\rho(A_{0,W}).
$$
The lower edge is attained only at $(0,0)$ and its fiber eigenvalue is
simple.  Thus changing the sign of one edge orbit interchanges which of the
two global outer edges is inherited by the periodic family.

\smallskip
\noindent\emph{A genuine affine background.}
For $\alpha=1/4$, the affine Bloch torus contains neither $0$ nor $\theta$.
Compactness and Theorem~\ref{T:outer-spectral-edges} therefore give
$$
-\rho(A_{0,W})
<\inf\Sigma_{1/4}
\leq\sup\Sigma_{1/4}
<\rho(A_{0,W}).
$$
Both outer edges move strictly into the spectral interval obtained by
varying over all unitary characters.

For example, if $b_1=b_2=b_3=c=1$, then
$$
\rho(A_{0,W})=1+\sqrt{10}.
$$
The three choices $\alpha=0,1/2,1/4$ therefore realize, respectively,
inheritance of only the upper edge, inheritance of only the lower edge,
and inheritance of neither edge.  This example concerns the two outermost
edges of the union of the two Bloch bands; it makes no claim that every
internal band edge satisfies the stronger conclusions predicted by the
spectral edges conjecture.

\subsection{Weighted trace formulas and cycle counting on
\texorpdfstring{$K_4$}{K4}}
\label{SS:example_trace_cycle}

We conclude the section by illustrating the weighted trace and cycle-counting
formulas on the complete graph $G=K_4$.  Label the vertices
$v_1,v_2,v_3,v_4$ as in Figure~\ref{F:K4}, and denote by $\ve_{ij}$ the
oriented edge from $v_i$ to $v_j$.  For $\varphi\in\mbbR$, put
$$
\omega_\varphi:=\frac{\varphi}{2\pi}
(d\ve_{12}+d\ve_{23}+d\ve_{34}+d\ve_{41}).
$$
Thus the four positively oriented square edges have gain
$e^{\sqrt{-1}\varphi}$, their inverses have gain
$e^{-\sqrt{-1}\varphi}$, and the two diagonals have gain $1$.

To compare the unweighted and weighted formulas without changing the
underlying character family, let $a,b>0$ and define the positive symmetric
weight $W_{a,b}$ by
$$
W_{a,b}(\ve)=
\begin{cases}
a,&\text{if the underlying edge of $\ve$ belongs to the square},\\
b,&\text{if the underlying edge of $\ve$ is a diagonal}.
\end{cases}
$$
The trivial weight is $W_{1,1}\equiv1$.  Our nontrivial example will be
$$
W^\star:=W_{2,1}.
$$

\begin{figure}[ht]
\centering
\begin{tikzpicture}[>=to,x=2cm,y=2cm]
\coordinate (v1) at (0,0);
\coordinate (v2) at (2.4,0);
\coordinate (v3) at (2.4,1.7);
\coordinate (v4) at (0,1.7);

\path[-,font=\large,line width=1.5pt,black]
(v1) edge[gray] node[pos=.28,anchor=south east,black]{$b$} (v3)
(v2) edge[gray] node[pos=.28,anchor=south west,black]{$b$} (v4)
(v1) edge
 node[pos=.5,anchor=north,black]{$ae^{\sqrt{-1}\varphi}$}
 node[pos=.5,sloped,allow upside down]{\midarrow} (v2)
(v2) edge
 node[pos=.5,anchor=west,black]{$ae^{\sqrt{-1}\varphi}$}
 node[pos=.55,sloped,allow upside down]{\midarrow} (v3)
(v3) edge
 node[pos=.5,anchor=south,black]{$ae^{\sqrt{-1}\varphi}$}
 node[pos=.6,sloped,allow upside down]{\midarrow} (v4)
(v4) edge
 node[pos=.5,anchor=east,black]{$ae^{\sqrt{-1}\varphi}$}
 node[pos=.5,sloped,allow upside down]{\midarrow} (v1);

\fill[black] (v1) circle (2.5pt);
\draw (v1) node[anchor=north east] {\large $v_1$};
\fill[black] (v2) circle (2.5pt);
\draw (v2) node[anchor=north west] {\large $v_2$};
\fill[black] (v3) circle (2.5pt);
\draw (v3) node[anchor=south west] {\large $v_3$};
\fill[black] (v4) circle (2.5pt);
\draw (v4) node[anchor=south east] {\large $v_4$};
\end{tikzpicture}
\caption{The weighted $K_4$ example.  The unweighted specialization is
$(a,b)=(1,1)$, and the displayed weighted specialization is
$W^\star=W_{2,1}$.}
\label{F:K4}
\end{figure}

\subsubsection*{The weighted twisted matrices}

With vertices ordered as $(v_1,v_2,v_3,v_4)$, the twisted weighted vertex
adjacency matrix is
$$
A_{\omega_\varphi,W_{a,b}}=
\begin{pmatrix}
0&ae^{\sqrt{-1}\varphi}&b&ae^{-\sqrt{-1}\varphi}\\
ae^{-\sqrt{-1}\varphi}&0&ae^{\sqrt{-1}\varphi}&b\\
b&ae^{-\sqrt{-1}\varphi}&0&ae^{\sqrt{-1}\varphi}\\
ae^{\sqrt{-1}\varphi}&b&ae^{-\sqrt{-1}\varphi}&0
\end{pmatrix}.
$$
This is a circulant matrix with eigenvalues
\begin{align*}
\mu_1&=b+2a\cos\varphi,&
\mu_2&=-b-2a\sin\varphi,\\
\mu_3&=b-2a\cos\varphi,&
\mu_4&=-b+2a\sin\varphi.
\end{align*}

The edge spectrum requires a little more care.  The weight $W_{a,b}$ is
symmetric, but it is reciprocal only when $a=b=1$.  Thus the classical
three-term Ihara--Bass formula cannot be applied to a nontrivial
$W_{a,b}$.  The general weighted formula in
Theorem~\ref{T:DertminantL}(b), however, remains explicit in this example.
Put
$$
x_a(u):=\frac{ua}{1-u^2a^2},\qquad
x_b(u):=\frac{ub}{1-u^2b^2},
$$
and
$$
d_{a,b}(u):=
\frac{2u^2a^2}{1-u^2a^2}
+\frac{u^2b^2}{1-u^2b^2}.
$$
Let $\nu_1,\ldots,\nu_4$ be the four expressions above with $a$ and $b$
replaced by $x_a(u)$ and $x_b(u)$, respectively.  Then
\begin{equation}\label{E:K4-weighted-determinant}
\det(I-uB_{\omega_\varphi,W_{a,b}})
=(1-u^2a^2)^4(1-u^2b^2)^2
\prod_{j=1}^4\bigl(1+d_{a,b}(u)-\nu_j\bigr).
\end{equation}
The apparent denominators cancel.  At $a=b=1$, this reduces to
$$
\det(I-uB_{\omega_\varphi})
=(1-u^2)^2\prod_{j=1}^4(1-\mu_ju+2u^2),
$$
and hence
$$
\spec B_{\omega_\varphi}
=
\left\{\frac{\mu_j\pm\sqrt{\mu_j^2-8}}2:1\leq j\leq4\right\}
\cup\{1,1,-1,-1\}
$$
as multisets.  This recovers the unweighted computation.

For later comparison, the partition of the oriented edges into square
edges and diagonals is equitable for $B_{0,W^\star}$.  Its quotient matrix
is
$$
\begin{pmatrix}2&1\\4&0\end{pmatrix},
$$
so Perron--Frobenius theory gives the exact weighted growth rate
\begin{equation}\label{E:K4-weighted-rho}
\rho_{G,W^\star}=\rho(B_{0,W^\star})=1+\sqrt5.
\end{equation}
The corresponding unweighted spectral radius is $\rho_{G,1}=2$.

\subsubsection*{Weighted trace identities and finite Fourier inversion}

We first retain the unweighted computation and place its weighted analogue
beside it.  When $\varphi=0$, the unweighted edge matrix has eigenvalues
$2$ with multiplicity $1$,
$(-1+\sqrt{-7})/2$ and $(-1-\sqrt{-7})/2$ each with multiplicity $3$,
$1$ with multiplicity $3$, and $-1$ with multiplicity $2$.  Therefore
$$
\mcalK(0,l)
=2^l+3\left(\frac{-1+\sqrt{-7}}2\right)^l
+3\left(\frac{-1-\sqrt{-7}}2\right)^l+3+2(-1)^l.
$$

Now set
$$
\omega_1:=\frac16
(d\ve_{12}+d\ve_{23}+d\ve_{34}+d\ve_{41}),\qquad
\omega_2:=\frac13
(d\ve_{12}+d\ve_{23}+d\ve_{34}+d\ve_{41}).
$$
Thus $\underline{\omega_2}=-\underline{\omega_1}$.  Since both
$W\equiv1$ and $W^\star$ are symmetric,
Lemma~\ref{L:Kappa}(b) gives
$$
\mcalK_W(\omega_1,l)=\mcalK_W(\omega_2,l).
$$
Some representative trace values are as follows.
\begin{equation}\label{E:K4-trace-table}
\begin{array}{c|r@{\quad}r|r@{\quad}r}
l&
\mcalK(0,l)&\mcalK_{W^\star}(0,l)&
\mcalK(\omega_1,l)&\mcalK_{W^\star}(\omega_1,l)\\ \hline
3&24&96&-12&-48\\
4&24&192&12&0\\
6&96&1536&-12&-192\\
7&168&5376&0&2688\\
8&168&9216&36&-384\\
9&528&33792&96&6144\\
10&1200&153600&-60&-30720\\
12&3960&1179648&-252&-92160\\
15&31944&42663936&768&509952
\end{array}
\end{equation}

Let
$$
\Lambda:=\{\alpha\in H_1(G,\mbbZ):
\chi_{\omega_1}(\alpha)=1\}.
$$
As in the unweighted computation, put
$$
\alpha_1=\ve_{12}+\ve_{23}+\ve_{31},\qquad
\alpha_2=\ve_{14}+\ve_{43}+\ve_{31}.
$$
Then $\Lambda$ has index $3$,
$$
Q_\Lambda=\{\underline{\alpha_0},\underline{\alpha_1},
\underline{\alpha_2}\}\simeq\mbbZ/3\mbbZ,
\qquad
\widehat{Q_\Lambda}
=\{0,\underline{\omega_1},\underline{\omega_2}\},
$$
where $\underline{\alpha_0}=\Lambda$ and
$\underline{\alpha_i}=\alpha_i+\Lambda$ for $i=1,2$.  The weighted finite
Fourier formula gives, simultaneously for $W\equiv1$ and $W=W^\star$,
\begin{align}
N_W(\underline{\alpha_0},l)
&=\frac{\mcalK_W(0,l)+2\mcalK_W(\omega_1,l)}3,\label{E:K4-N0}\\
N_W(\underline{\alpha_1},l)
=N_W(\underline{\alpha_2},l)
&=\frac{\mcalK_W(0,l)-\mcalK_W(\omega_1,l)}3.\label{E:K4-N1}
\end{align}
For example, at $l=6$ the weighted circuit sums are
$$
\bigl(
N_{W^\star}(\underline{\alpha_0},6),
N_{W^\star}(\underline{\alpha_1},6),
N_{W^\star}(\underline{\alpha_2},6)
\bigr)=(384,576,576),
$$
whereas the unweighted values are $(24,36,36)$.

The analytic trace formula is equally concrete.  For $h(z)=e^z$,
Theorem~\ref{T:trace-formula}(a) gives
$$
\sum_{\lambda\in\spec B_{0,W}}(e^\lambda-1)
=\sum_{l\geq1}\frac{N_W(l)}{l!}.
$$
For the two weights under consideration, the values of the two sides are
$$
\begin{cases}
e^2+3e+2e^{-1}
+6e^{-1/2}\cos(\sqrt7/2)-12
=5.172675227\cdots,&W\equiv1,\\
27.577567810\cdots,&W=W^\star.
\end{cases}
$$
The finite-quotient version for $\underline{\alpha_0}$ reads
$$
\sum_{\omega\in\{0,\omega_1,\omega_2\}}
\sum_{\lambda\in\spec B_{\omega,W}}(e^\lambda-1)
=3\sum_{l\geq1}\frac{N_W(\underline{\alpha_0},l)}{l!},
$$
whose common value is $2.141622583\cdots$ for $W\equiv1$ and
$12.108795534\cdots$ for $W=W^\star$.

\subsubsection*{Weighted prime cycles and cycles}

We next apply Proposition~\ref{P:IdenCountFunc}.  The powered weights in
that proposition are visible even in the first nontrivial repeated-cycle
calculation.  In the class $\underline{\alpha_1}$ at length $6$,
\begin{align*}
N_{W^\star}(\underline{\alpha_1},6)
&=6\pi_{W^\star}(\underline{\alpha_1},6)
+3\pi_{(W^\star)^2}(\underline{\alpha_2},3)\\
&=6\cdot64+3\cdot64=576,
\end{align*}
and
$$
\Pi_{W^\star}(\underline{\alpha_1},6)
=\pi_{W^\star}(\underline{\alpha_1},6)
+\pi_{(W^\star)^2}(\underline{\alpha_2},3)
=128.
$$
The second term records squares of prime triangles.  It involves
$(W^\star)^2$, not $W^\star$; this distinction disappears in the
unweighted specialization.

The following tables compare representative unweighted and weighted
values.  The two nonzero classes have equal sums, so it suffices to display
$\underline{\alpha_1}$.
\begin{equation}\label{E:K4-cycle-table}
\begin{array}{c|rr|rr}
&\multicolumn{2}{c|}{\pi_W(\underline{\alpha_0},l)}
&\multicolumn{2}{c}{\Pi_W(\underline{\alpha_0},l)}\\
l&W\equiv1&W=W^\star&W\equiv1&W=W^\star\\ \hline
3&0&0&0&0\\
4&4&16&4&16\\
6&4&64&4&64\\
8&8&320&12&384\\
12&92&24320&102&33792\\
15&744&970752&744&970752
\end{array}
\qquad
\begin{array}{c|rr|rr}
&\multicolumn{2}{c|}{\pi_W(\underline{\alpha_1},l)}
&\multicolumn{2}{c}{\Pi_W(\underline{\alpha_1},l)}\\
l&W\equiv1&W=W^\star&W\equiv1&W=W^\star\\ \hline
3&4&16&4&16\\
4&1&16&1&16\\
6&4&64&8&128\\
8&5&272&6&528\\
12&114&34560&122&36608\\
15&692&935936&696&940032
\end{array}
\end{equation}

The graph $K_4$ is non-bipartite and has $\nu_G=1$, so its period character
is the canonical character $\eta_G=\theta$.  The order-$3$ group
$\widehat{Q_\Lambda}$ does not contain this character.  Hence
Theorem~\ref{T:counting}(b), together with
\eqref{E:K4-weighted-rho}, gives, for every
$\underline{\alpha}\in Q_\Lambda$,
$$
\pi_W(\underline{\alpha},l)
\sim\Pi_W(\underline{\alpha},l)
\sim
\begin{cases}
\displaystyle\frac{2^l}{3l},&W\equiv1,\\[6pt]
\displaystyle\frac{(1+\sqrt5)^l}{3l},&W=W^\star.
\end{cases}
$$
Thus the weight changes the exponential growth rate while preserving the
uniform distribution among the three quotient-homology classes.

\subsubsection*{The canonical-character refinement}

Finally, we retain the order-$6$ refinement from the unweighted example and
observe that it works without change for $W^\star$.  Put
\begin{align*}
\omega_3&:=\frac14
(d\ve_{12}+d\ve_{23}+d\ve_{34}+d\ve_{41}),\\
\omega_4&:=\frac1{12}
(d\ve_{12}+d\ve_{23}+d\ve_{34}+d\ve_{41}),\\
\omega_5&:=\frac5{12}
(d\ve_{12}+d\ve_{23}+d\ve_{34}+d\ve_{41}).
\end{align*}
Then $\underline{\omega_3}=\theta$,
$2\underline{\omega_4}=\underline{\omega_1}$,
$3\underline{\omega_4}=\theta$,
$4\underline{\omega_4}=\underline{\omega_2}$, and
$5\underline{\omega_4}=\underline{\omega_5}=-\underline{\omega_4}$.
For either $W\equiv1$ or $W=W^\star$, spectral antisymmetry and symmetry of
the weight give
\begin{align*}
\mcalK_W(\omega_3,l)&=(-1)^l\mcalK_W(0,l),\\
\mcalK_W(\omega_4,l)=\mcalK_W(\omega_5,l)
&=(-1)^l\mcalK_W(\omega_1,l).
\end{align*}

Let
$$
\Lambda':=\{\alpha\in H_1(G,\mbbZ):
\chi_{\omega_4}(\alpha)=1\},
\qquad
\underline{\alpha_i'}:=i\alpha_1+\Lambda',
\quad 0\leq i\leq5.
$$
Then $Q_{\Lambda'}\simeq\mbbZ/6\mbbZ$,
$\widehat{Q_{\Lambda'}}$ is generated by
$\underline{\omega_4}$, and
$$
\underline{\alpha_i}
=\underline{\alpha_i'}\sqcup\underline{\alpha_{i+3}'},
\qquad i=0,1,2.
$$
Finite Fourier inversion now gives
\begin{align*}
N_W(\underline{\alpha_0'},l)
&=\frac{1+(-1)^l}{6}
\bigl(\mcalK_W(0,l)+2\mcalK_W(\omega_1,l)\bigr),\\
N_W(\underline{\alpha_1'},l)
=N_W(\underline{\alpha_5'},l)
&=\frac{1-(-1)^l}{6}
\bigl(\mcalK_W(0,l)-\mcalK_W(\omega_1,l)\bigr),\\
N_W(\underline{\alpha_2'},l)
=N_W(\underline{\alpha_4'},l)
&=\frac{1+(-1)^l}{6}
\bigl(\mcalK_W(0,l)-\mcalK_W(\omega_1,l)\bigr),\\
N_W(\underline{\alpha_3'},l)
&=\frac{1-(-1)^l}{6}
\bigl(\mcalK_W(0,l)+2\mcalK_W(\omega_1,l)\bigr).
\end{align*}
More compactly, for
$X_W\in\{N_W,\pi_W,\Pi_W\}$ and $0\leq i\leq5$,
\begin{equation}\label{E:K4-parity-refinement}
X_W(\underline{\alpha_i'},l)
=
\begin{cases}
X_W(\underline{\alpha_{i\bmod3}},l),&i+l\ \text{is even},\\
0,&i+l\ \text{is odd}.
\end{cases}
\end{equation}
Indeed, the two classes over
$\underline{\alpha_{i\bmod3}}$ are distinguished by the canonical
character, and Corollary~\ref{C:vanishing} gives the parity restriction.
Positivity of the weight transfers the same restriction from circuits to
prime cycles and cycles.

Since $\eta_G=\theta\in\widehat{Q_{\Lambda'}}$,
Theorem~\ref{T:counting}(c) gives
the corresponding parity-sensitive asymptotics.  If $i+l$ is odd, all
three sums in \eqref{E:K4-parity-refinement} vanish.  If $i+l$ is even,
then
$$
\pi_W(\underline{\alpha_i'},l)
\sim\Pi_W(\underline{\alpha_i'},l)
\sim
\begin{cases}
\displaystyle\frac{2^l}{3l},&W\equiv1,\\[6pt]
\displaystyle\frac{(1+\sqrt5)^l}{3l},&W=W^\star.
\end{cases}
$$
Thus the nontrivial weight changes the magnitude of the cycle sums but not
the canonical-character parity law.
 
\section{Proofs of Theorem~\ref{T:SpecDuality} and Proposition~\ref{P:period-character-algorithm}} \label{S:proof-antisymmetry}

Throughout this section $W:\vE(G)\to\mbbR_{>0}$ is a fixed positive
directed weight.  We prove Theorem~\ref{T:SpecDuality} directly in this
weighted setting.  The argument proceeds through a series of lemmas.  We
first identify the canonical character and establish spectral antisymmetry
and spectral rigidity.  We then develop the walk and edge-walk criteria for
strict spectral-radius decrease, analyze their equality cases, and obtain
the period-character rotation and the complete extremal-character
classification.  These results are assembled in the proof of the main
theorem near the end of the section.  We conclude with the proof of
Proposition~\ref{P:period-character-algorithm}.  The
spectral-antisymmetry statements allow full complex characters, whereas the
spectral-radius arguments use unitary characters and the identity
$|\chi_\omega|=1$.  The unweighted formulation is recovered by taking
$W\equiv1$.

\begin{lemma}\label{L:O-TO}
For every orientation $\mfrako$ and every spanning tree $T$, Construction~\ref{Cs:CanoForm2} also represents the canonical character:
$$\theta=\underline{\omega_{\mfrako}}=\underline{\omega_{(T,\mfrako)}}.$$
\end{lemma}

\begin{proof}
Let $\{\ve_1,\ldots,\ve_g\}=\vE_{\mfrako}(G)\setminus T$, and let
$$u_i=\ve_i+T(\ve_i(1),\ve_i(0))^{\ab},\qquad i=1,\ldots,g,$$
be the spanning-tree basis of $H_1(G,\mbbZ)$ from Proposition~\ref{P:basis_gen}.  It suffices to show that $\omega_{\mfrako}(u_i)-\omega_{(T,\mfrako)}(u_i)\in\mbbZ$ for every $i$.  If $\ve_i$ joins $V_1$ and $V_2$, then $\tau(T(\ve_i(1),\ve_i(0)))$ is odd, so $\omega_{\mfrako}(u_i)\in\mbbZ$ and $\omega_{(T,\mfrako)}(u_i)=0$.  Otherwise $\tau(T(\ve_i(1),\ve_i(0)))$ is even, so $\omega_{\mfrako}(u_i)-1/2\in\mbbZ$ and $\omega_{(T,\mfrako)}(u_i)=1/2$.  Hence the two $1$-forms have the same image in $\mcalX$.
\end{proof}

\begin{lemma}\label{L:spectral-neg1}
For every orientation $\mfrako$ and every $\omega\in\Omega(G)$,
$$
\chi_{\omega+\omega_{\mfrako}}(\ve)=-\chi_\omega(\ve)
\qquad\text{for every }\ve\in\vE(G),
$$
and consequently
$$A_{\omega+\omega_{\mfrako},W}=-A_{\omega,W},\qquad
B_{\omega+\omega_{\mfrako},W}=-B_{\omega,W}.$$
\end{lemma}

\begin{proof}
The character identity follows from $\chi_{\omega_{\mfrako}}(\ve)=-1$, and the two matrix identities then follow entrywise from Definition~\ref{D:weighted-adj}.
\end{proof}

\begin{lemma}\label{L:spectral-neg2}
Fix an orientation $\mfrako$ and a spanning tree $T$. Let $\gamma$ be a complex gain on $G$, let $\gamma_\theta(\ve):=\chi_{\omega_{(T,\mfrako)}}(\ve)$, and put
$$\gamma':=\gamma_\theta\cdot\gamma.$$
Then
$$A_{\gamma',W}\sim -A_{\gamma,W},\qquad B_{\gamma',W}\sim -B_{\gamma,W}.$$
Consequently, for every full character $\xi\in\mcalX_\mbbC$,
$$\spec A_{\theta+\xi,W}=-\spec A_{\xi,W},\qquad
\spec B_{\theta+\xi,W}=-\spec B_{\xi,W}.$$
If $\omega,\omega'\in\Omega(G)$ are unitary representatives satisfying
$$\underline{\omega'}=\theta+\underline{\omega},$$
then
$$\spec A_{\omega',W}=-\spec A_{\omega,W},\qquad
\spec B_{\omega',W}=-\spec B_{\omega,W}.$$
\end{lemma}

\begin{proof}
Let $\gamma_{\mfrako}:=\chi_{\omega_{\mfrako}}$, so
$\gamma_{\mfrako}(\ve)=-1$ for every oriented edge $\ve$.  By
Lemma~\ref{L:O-TO}, $\gamma_\theta$ and $\gamma_{\mfrako}$ represent the same
character and hence are switching equivalent.  Therefore
$\gamma_\theta\cdot\gamma$ and $\gamma_{\mfrako}\cdot\gamma$ are switching
equivalent, so Lemma~\ref{L:switching-similarity} gives
$$
A_{\gamma',W}\sim A_{\gamma_{\mfrako}\cdot\gamma,W}
=-A_{\gamma,W},
\qquad
B_{\gamma',W}\sim B_{\gamma_{\mfrako}\cdot\gamma,W}
=-B_{\gamma,W}.
$$
The spectral statements follow, with the unitary formulation obtained by
representing the characters by $1$-forms.
\end{proof}

\begin{lemma}\label{L:weighted-spectral-rigidity}
For every positive directed weight $W$ and every $\omega\in\Omega(G)$,
\begin{align*}
\spec A_{\omega,W}=\spec A_{0,W}
&\Longleftrightarrow \underline{\omega}=0
\Longleftrightarrow \spec B_{\omega,W}=\spec B_{0,W},\\
\spec A_{\omega,W}=-\spec A_{0,W}
&\Longleftrightarrow \underline{\omega}=\theta
\Longleftrightarrow \spec B_{\omega,W}=-\spec B_{0,W}.
\end{align*}
\end{lemma}

\begin{proof}
The reverse implications in the first line follow from
Lemma~\ref{L:switching-similarity}.

We first prove the converse for vertex matrices.  Suppose that $A_{\omega,W}$ and $A_{0,W}$ are isospectral.  Then their traces agree in every positive power.  Thus, for every $l\geq1$, the weighted trace expansion gives
$$0=\operatorname{tr}(A_{0,W}^l)-\operatorname{tr}(A_{\omega,W}^l)
=\sum_{\substack{\Delta\text{ closed}\\ \tau(\Delta)=l}}
W(\Delta)\left(1-\chi_\omega(\Delta)\right),$$
where the sum is over all closed walks of length $l$.  Taking real parts yields
$$\sum_{\substack{\Delta\text{ closed}\\ \tau(\Delta)=l}}
W(\Delta)\left(1-\mathrm{Re}\,\chi_\omega(\Delta)\right)=0.$$
Since $W(\Delta)>0$ and $\chi_\omega(\Delta)\in S^1$, each summand is nonnegative.  Hence $\chi_\omega(\Delta)=1$ for every closed walk $\Delta$.  In particular, $\chi_\omega$ is trivial on the spanning-tree basis of $H_1(G,\mbbZ)$ from Proposition~\ref{P:basis_gen}, so $\underline{\omega}=0$.

We next prove the converse for edge matrices.  Suppose that $B_{\omega,W}$ and $B_{0,W}$ are isospectral.  The edge determinant identity in Theorem~\ref{T:DertminantL}(a) gives
$$L_W(u,\chi_\omega)=L_W(u,\chi_0)=z_W(u).$$
Proposition~\ref{P:LFuncIden}(a) therefore gives $\underline{\omega}=0$.

For the second line, let $\omega_\theta$ represent $\theta$.  If $\underline{\omega}=\theta$, then Lemma~\ref{L:switching-similarity}, Lemma~\ref{L:spectral-neg2}, and the reality of $A_{0,W}$ and $B_{0,W}$ give
$$\spec A_{\omega,W}=-\spec A_{0,W},\qquad
\spec B_{\omega,W}=-\spec B_{0,W}.$$
Conversely, suppose first that $\spec A_{\omega,W}=-\spec A_{0,W}$.  Lemma~\ref{L:spectral-neg2}, applied to $\omega$, gives
$$\spec A_{\omega_\theta+\omega,W}
=-\spec A_{\omega,W}
=\spec A_{0,W}.$$
The first line gives $\underline{\omega_\theta+\omega}=0$, and therefore $\underline{\omega}=-\theta=\theta$.  The same argument, with $B$ in place of $A$, proves the converse from $\spec B_{\omega,W}=-\spec B_{0,W}$.
\end{proof}

\begin{lemma}\label{L:path}
Consider a graph $G$ and an arbitrary walk $\Delta$ on $G$.  There exists $a\in\mbbN$ such that for each pair of vertices $v$ and $v'$ (not necessarily distinct), and each nonnegative integer $b$, we have either
\begin{enumerate}[(i)]
\item there is no $x\in\mbbN$ for which a walk from $v$ to $v'$ has length $ax+b$; or
\item there exists some $x_0\in\mbbN$ such that for all $x>x_0$, there exists a walk from $v$ to $v'$ of length $ax+b$ which contains $\Delta$ as a sub-walk.
\end{enumerate}
\end{lemma}

\begin{proof}
First we choose a period which works simultaneously for all base vertices.  For each vertex $v$, choose a walk from $v$ to $\Delta(0)$ and a walk from $\Delta(1)$ back to $v$.  Since $G$ is connected, this gives a closed walk $\Upsilon'_v$ based at $v$ which contains $\Delta$ as a sub-walk.  Let $a$ be a common multiple of the finitely many lengths $\tau(\Upsilon'_v)$.  Replacing $\Upsilon'_v$ by the appropriate positive power, we obtain a closed walk $\Upsilon_v$ based at $v$, of length $a$, containing $\Delta$.

Now fix vertices $v,v'$ and a residue $b\geq0$.  If there is no $x\in\mbbN$ for which such a walk exists, then we are in the first alternative.  Otherwise choose $x_0\in\mbbN$ and a walk $\Gamma$ from $v$ to $v'$ of length $ax_0+b$.  For every $x>x_0$, the concatenation
$$\Upsilon_v^{x-x_0}\cdot\Gamma$$
is a walk from $v$ to $v'$ of length $ax+b$, and it contains $\Delta$ because $\Upsilon_v$ does.
\end{proof}

We say two walks $\Delta$ and $\Delta'$ of the same length with the same initial vertex and the same terminal vertex are \emph{coherent} with respect to $\chi_\omega$ if
$$\chi_\omega(\Delta)=\chi_\omega(\Delta'),$$
and \emph{incoherent} with respect to $\chi_\omega$ otherwise. 

\begin{lemma}\label{L:incoherent}
$\rho(A_{\omega,W})\leq\rho(A_{0,W})$ for all $\omega\in\Omega(G)$. In addition, $\rho(A_{\omega,W})<\rho(A_{0,W})$ if and only if there is a pair of incoherent walks with respect to $\chi_\omega$.
\end{lemma}

\begin{proof}
We will employ Gelfand's formula, which relates spectral radius to matrix norms:
\begin{quote}
\textbf{(Gelfand's formula.)} \emph{For any square matrix $M$ and any matrix norm $\|\cdot\|$,}
$$\rho(M)=\lim_{k\to\infty}\|M^k\|^{1/k}.$$
\end{quote}

Write $V(G)=\{v_1,\ldots,v_n\}$.  We use the $\infty$-norm
$$\|M\|:=\max_{1\leq i\leq n}\sum_{j=1}^n |M_{ij}|.$$
We also use the equality case of the triangle inequality: for positive real
numbers $c_1,\ldots,c_N$ and unit complex numbers $z_1,\ldots,z_N$,
$$
\left|\sum_{j=1}^N c_jz_j\right|\leq\sum_{j=1}^N c_j,
$$
with equality if and only if $z_1=\cdots=z_N$.
For vertices $v_i,v_j$, let $\Gamma_k(i,j)$ be the set of walks from $v_i$ to $v_j$ of length $k$.  Then
$$\left(A_{\omega,W}^k\right)_{ij}
=\sum_{\Delta\in\Gamma_k(i,j)}W(\Delta)\chi_\omega(\Delta),\qquad
\left(A_{0,W}^k\right)_{ij}
=\sum_{\Delta\in\Gamma_k(i,j)}W(\Delta).$$
Since $W(\Delta)>0$ and $|\chi_\omega(\Delta)|=1$, the triangle inequality gives
$$\left|\left(A_{\omega,W}^k\right)_{ij}\right|
\leq \left(A_{0,W}^k\right)_{ij}.$$
Thus $\|A_{\omega,W}^k\|\leq \|A_{0,W}^k\|$ for every $k$, and Gelfand's formula gives
$$\rho(A_{\omega,W})\leq \rho(A_{0,W}).$$

Now suppose that there are no pairs of incoherent walks with respect to $\chi_\omega$.  Then all walks in each nonempty set $\Gamma_k(i,j)$ are coherent.  Hence the weighted sum above has no cancellation:
$$\left|\left(A_{\omega,W}^k\right)_{ij}\right|
=\sum_{\Delta\in\Gamma_k(i,j)}W(\Delta)
=\left(A_{0,W}^k\right)_{ij}$$
for all $i,j,k$.  Consequently $\|A_{\omega,W}^k\|=\|A_{0,W}^k\|$ for every $k$, and Gelfand's formula implies
$$\rho(A_{\omega,W})=\rho(A_{0,W}).$$
This proves one direction of the strict statement.

It remains to prove that one incoherent pair forces a strict drop.  Suppose that $\Delta,\Delta'\in\Gamma_r(s,t)$ are incoherent.  Apply Lemma~\ref{L:path} to $\Delta$ with $b=0$.  There exists $a\in\mbbN$ such that, for each pair $v_i,v_j$, either
\begin{enumerate}[(i)]
\item $\Gamma_{ax}(i,j)=\emptyset$ for all $x\in\mbbN$; or
\item for all sufficiently large $x$, there exists a walk $\Upsilon\in\Gamma_{ax}(i,j)$ containing $\Delta$ as a sub-walk.
\end{enumerate}
For every pair $(i,j)$ satisfying the second alternative, choose a threshold $x_0(i,j)$ after which the stated walk exists.  Since there are only finitely many pairs of vertices, choose $x$ larger than all these thresholds, and set $y=ax$.

Now fix $i,j$.  If the first alternative holds for $(i,j)$, then $\Gamma_y(i,j)=\emptyset$, so $\left(A_{\omega,W}^{y}\right)_{ij}=\left(A_{0,W}^{y}\right)_{ij}=0$.  If the second alternative holds, choose a walk $\Upsilon\in\Gamma_y(i,j)$ containing $\Delta$ as a sub-walk.  Replace one chosen copy of $\Delta$ in $\Upsilon$ by $\Delta'$ and leave the rest of the walk unchanged.  The resulting walk $\Upsilon'\in\Gamma_y(i,j)$ has the same endpoints and length, but
$$\chi_\omega(\Upsilon')\cdot\chi_\omega(\Upsilon)^{-1}
=\chi_\omega(\Delta')\cdot\chi_\omega(\Delta)^{-1}\neq1.$$
Since the two corresponding weights $W(\Upsilon)$ and $W(\Upsilon')$ are positive, at least two summands in $\left(A_{\omega,W}^{y}\right)_{ij}$ have distinct arguments.  Therefore the triangle inequality is strict:
$$\left|\left(A_{\omega,W}^{y}\right)_{ij}\right|
<\left(A_{0,W}^{y}\right)_{ij}$$
whenever $\left(A_{0,W}^{y}\right)_{ij}\neq0$.  Since there are only finitely many pairs $(i,j)$, we may choose $0<c<1$ such that
$$\left|\left(A_{\omega,W}^{y}\right)_{ij}\right|
\leq c\left(A_{0,W}^{y}\right)_{ij}$$
for all $i,j$.

For every $i,j$ and every $k\geq1$,
\begin{align*}
\left|\left(A_{\omega,W}^{ky}\right)_{ij}\right|
&=\left|\sum_{i_1,\ldots,i_{k-1}}\left(A_{\omega,W}^{y}\right)_{ii_1}
\left(A_{\omega,W}^{y}\right)_{i_1i_2}\cdots
\left(A_{\omega,W}^{y}\right)_{i_{k-1}j}\right|\\
&\leq \sum_{i_1,\ldots,i_{k-1}}\left|\left(A_{\omega,W}^{y}\right)_{ii_1}\right|
\left|\left(A_{\omega,W}^{y}\right)_{i_1i_2}\right|\cdots
\left|\left(A_{\omega,W}^{y}\right)_{i_{k-1}j}\right|\\
&\leq c^k\sum_{i_1,\ldots,i_{k-1}}\left(A_{0,W}^{y}\right)_{ii_1}
\left(A_{0,W}^{y}\right)_{i_1i_2}\cdots
\left(A_{0,W}^{y}\right)_{i_{k-1}j}\\
&=c^k\left(A_{0,W}^{ky}\right)_{ij}.
\end{align*}
Hence
$$\|A_{\omega,W}^{ky}\|\leq c^k\|A_{0,W}^{ky}\|.$$
Since $\rho(M^y)=\rho(M)^y$, Gelfand's formula applied to $M^y$ gives
$$
\rho(M)=\lim_{k\to\infty}\|M^{ky}\|^{1/(ky)}.
$$
Applying this identity to $A_{\omega,W}$ and $A_{0,W}$, we obtain
$$\rho(A_{\omega,W})
=\lim_{k\to\infty}\|A_{\omega,W}^{ky}\|^{1/(ky)}
\leq \lim_{k\to\infty}\left(c^k\|A_{0,W}^{ky}\|\right)^{1/(ky)}
=c^{1/y}\rho(A_{0,W})<\rho(A_{0,W}).$$
\end{proof}

\begin{lemma}\label{L:spe-ineq}
For the weighted vertex matrices,
$$\rho(A_{\omega,W})<\rho(A_{0,W})$$
whenever $\underline{\omega}\notin\{0,\theta\}$.\footnote{For $W\equiv1$, this also follows from Theorem~4.4 of Mehatari--Kannan--Samanta~\cite{MKS2022On}, Lemma~\ref{L:O-TO}, and Lemma~\ref{L:switching-similarity}.  Their Perron--Frobenius argument concerns vertex adjacency; our walk proof treats arbitrary positive directed weights and extends to edge adjacency.}
\end{lemma}

\begin{proof}
Fix an orientation $\mfrako$ and a spanning tree $T$.  By Remark~\ref{R:X_TO}, represent $\underline{\omega}$ by the unique element of $\mcalX_{(T,\mfrako)}$,
$$\omega=\sum_{\ve\in\vE_{\mfrako}(G)\setminus T}\omega(\ve)d\ve,\qquad 0\leq\omega(\ve)<1.$$
For a cotree oriented edge $\ve$, let
$$C_\ve:=\ve\cdot T(\ve(1),\ve(0))$$
be the fundamental closed walk.  Since $\omega$ vanishes on tree edges, $\chi_\omega(C_\ve)=\chi_\omega(\ve)$.

Using Construction~\ref{Cs:CanoForm2}, let
$$E_0=\{\ve\in\vE_{\mfrako}(G)\setminus T:
\tau(T(\ve(1),\ve(0)))\text{ is odd}\}
=\{\ve\in\vE_{\mfrako}(G)\setminus T:\kappa(\ve)=0\}$$
and
$$E_1=\{\ve\in\vE_{\mfrako}(G)\setminus T:
\tau(T(\ve(1),\ve(0)))\text{ is even}\}
=\{\ve\in\vE_{\mfrako}(G)\setminus T:\kappa(\ve)=1/2\}.$$
If $\ve\in E_0$, then $C_\ve$ has even length; hence $\chi_\omega(\ve)\neq1$ gives an even closed walk with nontrivial character.  If $\ve\in E_1$, then $C_\ve$ has odd length; hence $C_\ve^2$ is even and has character $\chi_\omega(\ve)^2$.  Thus $\chi_\omega(\ve)\notin\{\pm1\}$ again gives an even closed walk with nontrivial character.

It remains to consider the case in which $\chi_\omega(\ve)=1$ for all $\ve\in E_0$ and $\chi_\omega(\ve)\in\{\pm1\}$ for all $\ve\in E_1$.  If two edges $\ve,\ve'\in E_1$ have different signs, then the closed walk
$$\ve\cdot T(\ve(1),\ve'(0))\cdot \ve'\cdot T(\ve'(1),\ve(0))$$
has even length and character $\chi_\omega(\ve)\cdot\chi_\omega(\ve')=-1$.  Therefore, if no even closed walk has nontrivial character, all edges in $E_1$ have the same sign and all edges in $E_0$ have sign $1$.  This says precisely that $\omega=0$ or $\omega=\omega_{(T,\mfrako)}$ in $\mcalX_{(T,\mfrako)}$.

Thus, whenever $\underline{\omega}\notin\{0,\theta\}$, there is an even closed walk $C$ with $\chi_\omega(C)\neq1$.  Compare $C$ with a back-and-forth closed walk of the same even length based at the same vertex; the latter has character $1$.  This gives an incoherent pair of walks, so Lemma~\ref{L:incoherent} applies.
\end{proof}

We now turn from the weighted vertex adjacency matrices to the weighted edge
adjacency matrices.  Recall from Subsection~\ref{SS:weighted-adjacency} that
an edge-walk begins and ends at oriented edges, while its associated walk
begins and ends at vertices.  The corresponding
spectral-radius analysis requires a connectivity notion for nonbacktracking
edge-walks.

We say that a graph $H$ with at least one edge is \emph{edge-walk-connected}
if, for every pair of oriented edges $\ve,\ve'\in\vE(H)$, there is a
nonbacktracking edge-walk from $\ve$ to $\ve'$.  In matrix form, for every
$\ve,\ve'$ there is an integer $k\geq0$ such that
$$
\left(B^k\right)_{\ve\ve'}>0.
$$

\begin{lemma}\label{L:edge-walk}
Consider an edge-walk-connected graph $G$ and an arbitrary non-backtracking edge-walk $\vDelta$ on $G$.  There exists $a\in\mbbN$ such that, for each pair of oriented edges $\ve,\ve'\in\vE(G)$ (not necessarily distinct) and each nonnegative integer $b$, either
\begin{enumerate}[(i)]
\item there is no $x\in\mbbN$ for which a non-backtracking edge-walk from $\ve$ to $\ve'$ has length $ax+b$; or
\item there exists $x_0\in\mbbN$ such that, for every $x>x_0$, there is a non-backtracking edge-walk from $\ve$ to $\ve'$ of length $ax+b$ which contains $\vDelta$ as a sub-edge-walk.
\end{enumerate}
\end{lemma}

\begin{proof}
As in the proof of Lemma~\ref{L:path}, we first construct a common period.  For every $\ve\in\vE(G)$, edge-walk-connectedness provides a non-backtracking edge-walk from $\ve$ to $\vDelta(0)$ and another from $\vDelta(1)$ to $\ve$.  Concatenating these two edge-walks with $\vDelta$ gives a closed non-backtracking edge-walk $\vUpsilon'_\ve$ based at $\ve$ and containing $\vDelta$.

Let $a$ be a common multiple of the finitely many positive integers $\tau(\vUpsilon'_\ve)$, and put
$$a_\ve:=\frac{a}{\tau(\vUpsilon'_\ve)},
\qquad
\vUpsilon_\ve:=(\vUpsilon'_\ve)^{a_\ve}.$$
Then $\vUpsilon_\ve$ is a closed non-backtracking edge-walk based at $\ve$, has length $a$, and contains $\vDelta$ as a sub-edge-walk.

Fix $\ve,\ve'$ and $b$.  If (i) does not hold, there are $x_0\in\mbbN$ and a non-backtracking edge-walk $\Gamma$ from $\ve$ to $\ve'$ of length $ax_0+b$.  For every $x>x_0$, the concatenation
$$\vUpsilon_\ve^{x-x_0}\cdot\Gamma$$
is a non-backtracking edge-walk from $\ve$ to $\ve'$ of length $ax+b$ and contains $\vDelta$.  Thus (ii) holds.
\end{proof}

The next lemma characterizes edge-walk-connected graphs and identifies the edge-walk-connected core of an arbitrary graph of genus at least two.

\begin{lemma}\label{L:genus-2-edge-walk}
A graph is edge-walk-connected if and only if it has genus at least two and has no degree-one vertices.  If $g\geq2$, then $G$ contains a unique subgraph $G^\circ$ of genus $g$ which is edge-walk-connected.
\end{lemma}

\begin{proof}
First consider the necessary conditions.  If $v$ has degree one and $\ve$ is its incident edge oriented toward $v$, then no non-backtracking edge-walk beginning at $\ve$ can continue beyond $\ve$.  Hence the graph is not edge-walk-connected.  If a connected graph with at least one edge has genus zero, it is a tree and has a degree-one vertex.  If it has genus one and no degree-one vertices, then
$$
\sum_{v\in V(G)}\deg(v)=2m=2n,
$$
so every vertex has degree two.  The graph is therefore a cycle.  Its two
directions of travel cannot be joined by a nonbacktracking edge-walk, so it is
again not edge-walk-connected.  Thus edge-walk-connectedness requires genus
at least two and the absence of degree-one vertices.

Now assume $g\geq2$.
Repeatedly perform the following trimming operation: whenever a vertex has degree one, remove that vertex and its unique incident edge.  When the process terminates, it produces a subgraph $G^\circ$ without degree-one vertices.  Each removal decreases both the numbers of vertices and edges by one, so it does not change the genus.

The result is independent of the order of trimming.  Indeed, no vertex or edge of a subgraph of minimum degree at least two can be removed at any stage, whereas the terminal subgraph itself has minimum degree at least two.  Thus $G^\circ$ is intrinsically the unique maximal subgraph of $G$ having minimum degree at least two.

Since $g\geq2$, the graph $G^\circ$ is not a cycle.  Its
edge-walk-connectedness also follows from Ortner--Woess
\cite[Lemma~3.1]{OW2007Nonbacktracking}: for a finite connected graph of
minimum degree at least two that is not a cycle, every oriented edge can be
joined to every other by a nonbacktracking edge-walk.  We retain the following
direct proof for completeness.

We now prove that $G^\circ$ is edge-walk-connected.  If $K\subseteq G^\circ$ is connected and has positive genus, then every vertex $v\in V(K)$ supports a non-backtracking closed walk: follow a path from $v$ to a cycle, traverse the cycle, and return along the path.  Denote the set of such walks by $\mcalW_v(K)$.

Let $\ve,\ve'\in\vE(G^\circ)$, and let $e,e'$ be their underlying unoriented edges.  The absence of degree-one vertices gives the following three facts:
\begin{enumerate}[(1)]
\item if $e$ is a cut edge, then $G^\circ\setminus e$ has two connected components, both of positive genus;
\item if both $e$ and $e'$ are cut edges, then $G^\circ\setminus\{e,e'\}$ has three connected components, the two outer components have positive genus, and both cut edges meet the middle component;
\item if $\{e,e'\}$ is a cut while neither $e$ nor $e'$ is a cut edge, then $G^\circ\setminus\{e,e'\}$ has two connected components, at least one of which has positive genus.
\end{enumerate}
Indeed, a genus-zero component excluded by any of these assertions would be a tree attached to the rest of the graph and would have disappeared during trimming.

We construct a non-backtracking edge-walk $\vDelta_{\ve,\ve'}$ from $\ve$ to $\ve'$.  Write $\vDelta_{\ve,\ve'}^\times$ for its interior walk.

First suppose that $\ve'=\ve$ or $\ve'=\ve^{-1}$.
\begin{enumerate}[(a)]
\item If $G^\circ\setminus e$ is connected, choose a spanning tree $T_0$ of $G^\circ\setminus e$.  We may take
$$\vDelta_{\ve,\ve}^\times=T_0(\ve(1),\ve(0)),$$
and we may take $\vDelta_{\ve,\ve^{-1}}^\times$ to be any element of $\mcalW_{\ve(1)}(G^\circ\setminus e)$.
\item If $e$ is a cut edge, write $G^\circ\setminus e=G_1\sqcup G_2$, with $\ve(0)\in G_1$ and $\ve(1)\in G_2$.  Choose
$$\Delta_1\in\mcalW_{\ve(0)}(G_1),
\qquad
\Delta_2\in\mcalW_{\ve(1)}(G_2).$$
Then we may take
$$\vDelta_{\ve,\ve}^\times=\Delta_2\cdot\ve^{-1}\cdot\Delta_1,
\qquad
\vDelta_{\ve,\ve^{-1}}^\times=\Delta_2.$$
\end{enumerate}

Now suppose that $\{\ve,\ve^{-1}\}\neq\{\ve',\ve'^{-1}\}$.
\begin{enumerate}[(a)]
\item If $G^\circ\setminus\{e,e'\}$ is connected, choose a spanning tree $T_0$ of that graph and set
$$\vDelta_{\ve,\ve'}^\times=T_0(\ve(1),\ve'(0)).$$

\item Suppose $e$ is a cut edge and $G^\circ\setminus e=G_1\sqcup G_2$, where $e'$ lies in $G_1$ and is not a cut edge of $G_1$.  Choose a spanning tree $T_1$ of $G_1\setminus e'$.  If $\ve(1)\in G_1$, take
$$\vDelta_{\ve,\ve'}^\times=T_1(\ve(1),\ve'(0)).$$
If $\ve(1)\in G_2$, choose $\Delta_2\in\mcalW_{\ve(1)}(G_2)$ and take
$$\vDelta_{\ve,\ve'}^\times
=\Delta_2\cdot\ve^{-1}\cdot T_1(\ve(0),\ve'(0)).$$

\item Suppose both $e$ and $e'$ are cut edges.  Write
$$G^\circ\setminus\{e,e'\}=G_1\sqcup G_2\sqcup G_3,$$
where $e$ meets $G_1$, $e'$ meets $G_3$, and both meet $G_2$.  Choose a spanning tree $T_2$ of $G_2$.  According to the orientations of $\ve$ and $\ve'$, the interior walk is obtained from the central path in $T_2$ and, when needed, one or both of the detours
$$\Delta_1\cdot\ve^{-1},
\qquad
\ve'^{-1}\cdot\Delta_3,$$
where $\Delta_1\in\mcalW_{\ve(1)}(G_1)$ and $\Delta_3\in\mcalW_{\ve'(0)}(G_3)$.  Explicitly, the four possibilities are
\begin{align*}
T_2(\ve(1),\ve'(0)),&\\
T_2(\ve(1),\ve'(1))\cdot\ve'^{-1}\cdot\Delta_3,&\\
\Delta_1\cdot\ve^{-1}\cdot T_2(\ve(0),\ve'(0)),&\\
\Delta_1\cdot\ve^{-1}\cdot T_2(\ve(0),\ve'(1))\cdot\ve'^{-1}\cdot\Delta_3.&
\end{align*}

\item Finally, suppose $\{e,e'\}$ is a cut while neither edge alone is a cut edge.  Write
$$G^\circ\setminus\{e,e'\}=G_1\sqcup G_2,$$
where $G_1$ has positive genus, and choose spanning trees $T_1,T_2$ of the two components.  A tree path in the component containing $\ve(1)$ and $\ve'(0)$ gives the desired interior walk whenever these vertices lie in the same component.  In either crossed orientation, insert a walk in $\mcalW_v(G_1)$ before traversing the inverse of the boundary edge which would otherwise be an immediate backtrack.  This gives a non-backtracking edge-walk in all four endpoint configurations.
\end{enumerate}
Thus $G^\circ$ is edge-walk-connected.  In particular, if $G$ has no degree-one vertices, then the trimming process removes nothing, so $G=G^\circ$ and $G$ is edge-walk-connected.  This proves the converse in the characterization.

Finally, let $H\subseteq G$ be an edge-walk-connected subgraph of genus $g$.  By the characterization just proved, $H$ has no degree-one vertices, so the maximality of $G^\circ$ gives $H\subseteq G^\circ$.  If this inclusion were strict, the equality of the genera would force every component of the omitted part to be a tree attached to $H$ at a single vertex; such a nonempty tree would have a degree-one vertex in $G^\circ$, a contradiction.  Hence $H=G^\circ$, proving the asserted uniqueness.
\end{proof}

We say that two non-backtracking edge-walks $\vDelta$ and $\vDelta'$ of the same length, with the same initial oriented edge and the same terminal oriented edge, are \emph{coherent} with respect to $\chi_\omega$ if
$$\chi_\omega(\vDelta)=\chi_\omega(\vDelta'),$$
and \emph{incoherent} with respect to $\chi_\omega$ if
$$\chi_\omega(\vDelta)\neq\chi_\omega(\vDelta').$$

\begin{lemma}\label{L:edge-incoherent}
For every $\omega\in\Omega(G)$,
$$\rho(B_{\omega,W})\leq\rho(B_{0,W}).$$
If $G$ is edge-walk-connected, then
$$\rho(B_{\omega,W})<\rho(B_{0,W})$$
if and only if $G$ contains a pair of incoherent non-backtracking edge-walks.
\end{lemma}

\begin{proof}
For a matrix indexed by oriented edges, use the row-sum norm
$$\|M\|:=\max_{\va\in\vE(G)}\sum_{\vb\in\vE(G)}|M_{\va\vb}|.$$
Gelfand's formula gives
$$\rho(B_{\omega,W})
=\lim_{k\to\infty}\|B_{\omega,W}^k\|^{1/k}.$$

Let $\Gamma_k(\va,\vb)$ be the set of all length-$k$ non-backtracking edge-walks from $\va$ to $\vb$.  By Definition~\ref{D:weighted-adj},
$$\left(B_{\omega,W}^k\right)_{\va\vb}
=\sum_{\vDelta\in\Gamma_k(\va,\vb)}
W(\vDelta)\chi_\omega(\vDelta),
\qquad
\left(B_{0,W}^k\right)_{\va\vb}
=\sum_{\vDelta\in\Gamma_k(\va,\vb)}W(\vDelta).$$
Therefore
$$\left|\left(B_{\omega,W}^k\right)_{\va\vb}\right|
\leq\left(B_{0,W}^k\right)_{\va\vb}$$
for all $\va,\vb$ and $k$.  Taking row-sum norms and applying Gelfand's formula proves
$$\rho(B_{\omega,W})\leq\rho(B_{0,W}).$$

Assume now that $G$ is edge-walk-connected.  If there is no incoherent pair in $G$, then all terms in each sum over $\Gamma_k(\va,\vb)$ have the same argument.  Hence
$$\left|\left(B_{\omega,W}^k\right)_{\va\vb}\right|
=\left(B_{0,W}^k\right)_{\va\vb}$$
for every $\va,\vb$ and $k$.  Thus $\|B_{\omega,W}^k\|=\|B_{0,W}^k\|$ for all $k$, and Gelfand's formula gives equality of the spectral radii.  This proves one direction of the strict statement.

Conversely, suppose that
$$\vDelta,\vDelta'\in\Gamma_r(\mathbf a,\mathbf b)$$
are incoherent.  Apply Lemma~\ref{L:edge-walk} to $\vDelta$ with $b=0$.  There exists $a\in\mbbN$ such that, for each pair $\va,\vb\in\vE(G)$, either
\begin{enumerate}[(i)]
\item $\Gamma_{ax}(\va,\vb)=\emptyset$ for every $x\in\mbbN$; or
\item for all sufficiently large $x$, there is an edge-walk in $\Gamma_{ax}(\va,\vb)$ containing $\vDelta$ as a sub-edge-walk.
\end{enumerate}
Choose $x$ large enough that (ii) holds simultaneously for every pair for which it can hold, and set $y=ax$.

In Case (i),
$$\left(B_{\omega,W}^y\right)_{\va\vb}
=\left(B_{0,W}^y\right)_{\va\vb}=0.$$
In Case (ii), choose $\vUpsilon\in\Gamma_y(\va,\vb)$ containing $\vDelta$ and replace one chosen copy of $\vDelta$ by $\vDelta'$.  The resulting edge-walk $\vUpsilon'$ has the same endpoints and length, while
$$\chi_\omega(\vUpsilon')\cdot\chi_\omega(\vUpsilon)^{-1}
=\chi_\omega(\vDelta')\cdot\chi_\omega(\vDelta)^{-1}\neq1.$$
Because $W(\vUpsilon)$ and $W(\vUpsilon')$ are positive, the corresponding summands have distinct arguments.  The triangle inequality is therefore strict:
$$\left|\left(B_{\omega,W}^y\right)_{\va\vb}\right|
<\left(B_{0,W}^y\right)_{\va\vb}$$
whenever the right-hand side is nonzero.

There are only finitely many pairs $(\va,\vb)$, so there exists $0<c<1$ such that
$$\left|\left(B_{\omega,W}^y\right)_{\va\vb}\right|
\leq c\left(B_{0,W}^y\right)_{\va\vb}$$
for every $\va,\vb\in\vE(G)$.  Consequently, for every $k\geq1$,
\begin{align*}
\left|\left(B_{\omega,W}^{ky}\right)_{\va\vb}\right|
&\leq \sum_{\mathbf c_1,\ldots,\mathbf c_{k-1}}
\left|\left(B_{\omega,W}^{y}\right)_{\va\mathbf c_1}\right|
\cdots
\left|\left(B_{\omega,W}^{y}\right)_{\mathbf c_{k-1}\vb}\right|\\
&\leq c^k\sum_{\mathbf c_1,\ldots,\mathbf c_{k-1}}
\left(B_{0,W}^{y}\right)_{\va\mathbf c_1}
\cdots
\left(B_{0,W}^{y}\right)_{\mathbf c_{k-1}\vb}\\
&=c^k\left(B_{0,W}^{ky}\right)_{\va\vb},
\end{align*}
where the intermediate oriented edges range over $\vE(G)$.  Hence
$$\|B_{\omega,W}^{ky}\|\leq c^k\|B_{0,W}^{ky}\|.$$
Applying Gelfand's formula to the powers of $B_{\omega,W}^y$ and $B_{0,W}^y$ gives
$$\rho(B_{\omega,W})
\leq c^{1/y}\rho(B_{0,W})
<\rho(B_{0,W}),$$
which completes the proof.
\end{proof}

\begin{lemma}[Length coherence and spectral rotation]
\label{L:length-coherence}
Let $G$ be edge-walk-connected.  If every pair of non-backtracking
edge-walks of the same length with the same initial and terminal oriented
edges is coherent with respect to $\chi_\omega$, then there is a number
$\zeta\in S^1$ such that
$$
\chi_\omega(P)=\zeta^{\tau(P)}
\qquad\text{for every prime cycle }[P]\in\bmcalP.
$$
Conversely, if such a $\zeta$ exists, then there is a unitary diagonal matrix
$D$, indexed by $\vE(G)$, for which
$$
B_{\omega,W}=\zeta\,D^{-1}\cdot B_{0,W}\cdot D.
$$
In particular,
$$
\spec B_{\omega,W}=\zeta\,\spec B_{0,W}.
$$
\end{lemma}

\begin{proof}
Fix an oriented edge $\ve_\ast$, and let $\mathcal L$ be the additive semigroup
of lengths of closed non-backtracking edge-walks based at $\ve_\ast$.
Coherence makes
$$
c(r):=\chi_\omega(\vDelta),\qquad r\in\mathcal L,
$$
independent of the choice of the based closed edge-walk $\vDelta$ of length
$r$.  Concatenation gives $c(r+s)=c(r)c(s)$.  The group generated by
$\mathcal L$ is $\nu_G\mbbZ$, because $\nu_G$ is the greatest common divisor
of the lengths of closed nonbacktracking edge-walks.  Hence $c$ extends
uniquely to a homomorphism
$\nu_G\mbbZ\to S^1$.  Choose $\zeta\in S^1$ whose $\nu_G$-th power is the
image of $\nu_G$.  Then
$$c(r)=\zeta^r\qquad(r\in\mathcal L).$$

Let $P$ be a prime cycle based at an arbitrary oriented edge $\ve$.  Choose
non-backtracking edge-walks $Q$ from $\ve_\ast$ to $\ve$ and $R$ from $\ve$
to $\ve_\ast$.  Comparing the based closed edge-walks
$Q\cdot P\cdot R$ and $Q\cdot R$ gives
$$
\chi_\omega(P)
=\frac{\chi_\omega(Q\cdot P\cdot R)}
       {\chi_\omega(Q\cdot R)}
=\zeta^{\tau(P)}.
$$

Conversely, attach to every allowed transition
$\ve\shortrightarrow\ve'$ the normalized factor
$\zeta^{-1}\chi_\omega(\ve')$.  The hypothesis says that the product of these
factors is $1$ along every closed nonbacktracking edge-walk.  Fix a base
oriented edge and define $d(\ve)$ by multiplying the factors along any
nonbacktracking edge-walk from the base edge to $\ve$.  Edge-walk-connectedness
and a return edge-walk show that this is independent of the chosen edge-walk.
Thus the normalized factor on $\ve\shortrightarrow\ve'$ is
$d(\ve)^{-1}d(\ve')$.  Taking $D_{\ve\ve}=d(\ve)$ gives the displayed
diagonal similarity and the spectral identity.
\end{proof}

\begin{lemma}\label{L:spe-eq-genus1}
If $g=1$, then $\rho(B_{\omega,W})=\rho(B_{0,W})$ for every $\omega$.
\end{lemma}

\begin{proof}
After pruning the attached trees, $G$ has a unique genus-one core $G^\circ$, which is a cycle graph with exactly two directed prime cycles $P$ and $P^{-1}$.  Put
$$r:=\tau(P)=\tau(P^{-1}).$$
Since the attached trees contain no prime cycles, the weighted $L$-function is
\begin{align*}
L_W(u,\chi_\omega)^{-1}
&=\left(1-W(P)\chi_\omega(P)u^r\right)
\left(1-W(P^{-1})\chi_\omega(P^{-1})u^r\right).
\end{align*}
By Theorem~\ref{T:DertminantL}(a), the characteristic polynomial of $B_{\omega,W}$ is therefore
\begin{align*}
p_{\omega,W}(\lambda)
=\lambda^{2m-2r}
\left(\lambda^r-W(P)\chi_\omega(P)\right)
\left(\lambda^r-W(P^{-1})\chi_\omega(P^{-1})\right).
\end{align*}
Because $|\chi_\omega(P)|=|\chi_\omega(P^{-1})|=1$, the moduli of the nonzero roots are
$$W(P)^{1/r}\qquad\text{and}\qquad W(P^{-1})^{1/r}.$$
They are independent of $\omega$, and hence
$$\rho(B_{\omega,W})=\rho(B_{0,W}).$$
\end{proof}

\begin{lemma}\label{L:spe-ineq-genus2}
Assume $g\geq2$.
\begin{enumerate}[(i)]
\item The period character, whenever it exists, is a nontrivial $2$-torsion
character depending only on $G$.  It always exists when $G$ is non-bipartite,
and then $\eta=\theta$.
\item If the period character $\eta$ exists, then for every full complex
character $\xi\in\mcalX_\mbbC$,
$$
\spec B_{\eta+\xi,W}
=e\left(\frac{1}{2\nu_G}\right)\spec B_{\xi,W}.
$$
In particular, $\rho(B_{\eta,W})=\rho(B_{0,W})$ but
$\spec B_{\eta,W}\neq\spec B_{0,W}$.
\item If $G$ is non-bipartite, then
$$
\rho(B_{\omega,W})=\rho(B_{0,W})
\quad\Longleftrightarrow\quad
\underline{\omega}\in\{0,\theta\}.
$$
If $G$ is bipartite and the period character does not exist, then
equality holds if and only if $\underline{\omega}=0$.  If the period
character $\eta$ exists, then equality holds if and only if
$\underline{\omega}\in\{0,\eta\}$.
\end{enumerate}
\end{lemma}

\begin{proof}
Let $G^\circ$ be the core in Lemma~\ref{L:genus-2-edge-walk}, and let $W^\circ$ be the restriction of $W$ to $G^\circ$.  Every circuit of $G$ is contained in $G^\circ$, so
$$L_{G,W}(u,\chi_\omega)=L_{G^\circ,W^\circ}(u,\chi_\omega).$$
Writing $B^\circ_{\omega,W^\circ}$ for the twisted weighted edge adjacency matrix of $G^\circ$, Theorem~\ref{T:DertminantL}(a) shows that $B_{\omega,W}$ and $B^\circ_{\omega,W^\circ}$ have the same nonzero spectrum.  Thus
$$\rho(B_{\omega,W})=\rho(B^\circ_{\omega,W^\circ}),
\qquad
\rho(B_{0,W})=\rho(B^\circ_{0,W^\circ}).$$

Assume these two radii are equal.  The core $G^\circ$ is
edge-walk-connected, so Lemma~\ref{L:edge-incoherent} implies that all
equal-length non-backtracking edge-walks with common endpoints are coherent.
Lemma~\ref{L:length-coherence} therefore gives $\zeta\in S^1$ such that
$$
\chi_\omega(P)=\zeta^{\tau(P)}
\qquad([P]\in\bmcalP).
$$
Applying the same identity to the inverse prime cycle $P^{-1}$ gives
$$
\chi_\omega(P)^{-1}=\chi_\omega(P^{-1})
=\zeta^{\tau(P)}=\chi_\omega(P),
$$
so every prime-cycle value is a sign.  In particular,
$\underline{\omega}$ is $2$-torsion because the fundamental cycles of a
spanning tree are prime cycles and form a basis of $H_1(G,\mbbZ)$.

Put $n_P:=\tau(P)/\nu_G$.  The integers $n_P$ have greatest common divisor
$1$.  Since $(\zeta^{\nu_G})^{2n_P}=1$ for every prime cycle, B\'ezout's
identity gives
$$
(\zeta^{\nu_G})^2=1.
$$
If $\zeta^{\nu_G}=1$, then $\chi_\omega(P)=1$ for every prime cycle and
$\underline{\omega}=0$.  If $\zeta^{\nu_G}=-1$, then
$$
\chi_\omega(P)=(-1)^{\tau(P)/\nu_G}.
$$
Such a character exists precisely when normalized length parity descends to
homology, and it is then the unique period character $\eta$ of
Definition~\ref{D:period-character}.  Conversely, if $\eta$ exists, take
$\zeta=e(1/(2\nu_G))$.  For every full complex character $\xi$, choose
complex gains $\gamma_\eta$ and $\gamma$ representing $\eta$ and $\xi$,
respectively.  For every prime cycle $[P]$, we have
$$
(\gamma_\eta\gamma)(P)
=\eta(P)\xi(P)=\zeta^{\tau(P)}\xi(P).
$$
We compare the two edge determinants by their Leibniz expansions.  Every
nonconstant contribution is a product over pairwise disjoint cyclic strings
of allowed nonbacktracking transitions.  If their total length is $k$,
multiplication by $\gamma_\eta$ multiplies that contribution by $\zeta^k$.
Hence
$$
\det(I-uB_{\gamma_\eta\gamma,W})
=\det(I-\zeta uB_{\gamma,W}),
$$
which proves the asserted spectral rotation over every full complex
character.  Taking $\xi=0$ also proves the equality of radii.
The two spectra are unequal by Lemma~\ref{L:weighted-spectral-rigidity}, since
$\eta$ is nontrivial.

If $G$ is non-bipartite, then $\nu_G$ is odd.  Thus normalized length parity
equals ordinary length parity, so the period character exists and is
$\eta=\theta$.  If $G$ is bipartite, $\nu_G$ is even, and the preceding
classification gives exactly the two alternatives in (iii).
\end{proof}

\long\def\CotreeStrictnessProof{%
\begin{proof}[Proof of Proposition~\ref{P:cotree-edge-rigidity}]
Let $G^\circ$ be the core in Lemma~\ref{L:genus-2-edge-walk}, and let
$W^\circ$ be the restriction of $W$ to $G^\circ$.  As in the proof of
Lemma~\ref{L:spe-ineq-genus2}, the weighted edge matrices of $G$ and
$G^\circ$ have the same nonzero spectra.  In particular,
$$
\rho(B_{\omega,W})=\rho(B^\circ_{\omega,W^\circ}),
\qquad
\rho(B_{0,W})=\rho(B^\circ_{0,W^\circ}).
$$

Fix an orientation $\mfrako$, and choose the spanning tree $T$ so that it contains all tree components attached to $G^\circ$; then every cotree edge lies in $G^\circ$.  Represent $\underline{\omega}$ by
$$\omega=\sum_{\ve\in\vE_{\mfrako}(G)\setminus T}\omega(\ve)d\ve\in\mcalX_{(T,\mfrako)}.$$
It is enough to prove the contrapositive.  Assume
$$\rho(B_{\omega,W})=\rho(B_{0,W}).$$
The preceding spectral-radius identities give
$$\rho(B^\circ_{\omega,W^\circ})=\rho(B^\circ_{0,W^\circ}).$$
Since $G^\circ$ is edge-walk-connected, the if-and-only-if assertion of Lemma~\ref{L:edge-incoherent} shows that every pair of non-backtracking edge-walks in $G^\circ$ with the same length and the same initial and terminal oriented edges is coherent.

For a cotree oriented edge $\ve$, write
$$C_\ve:=\ve\cdot T(\ve(1),\ve(0))$$
for its fundamental closed walk.  Fix two cotree oriented edges $\ve_a,\ve_b$ with distinct underlying edges, and put
$$T_a:=T(\ve_a(1),\ve_a(0)),
\qquad
T_b:=T(\ve_b(1),\ve_b(0)),
\qquad
C_a:=C_{\ve_a},
\qquad
C_b:=C_{\ve_b}.$$
For brevity, write
$$
\tau_a:=\tau(C_a),
\qquad
\tau_b:=\tau(C_b).
$$
After replacing $\ve_a$ or $\ve_b$ by its inverse when necessary, the tree paths $T_a$ and $T_b$ have exactly one of the following relative positions:
\begin{enumerate}[(a)]
\item they overlap in a maximal tree path $T(v,w)$;
\item they meet in exactly one vertex $v$;
\item they are disjoint and are joined by a minimal tree path $T(v,w)$.
\end{enumerate}

\begin{figure}[tbp]
\centering
\begin{tikzpicture}[>=to,x=2.2cm,y=2.2cm]

\begin{scope}[shift={(0,0)}]
\draw (-0.6,1.2) node {(a)};

\coordinate (v) at (0.5,0);
\coordinate (w) at (-0.5,0);
\coordinate (ea1) at (0.5,0.5);
\coordinate (ea0) at (0.2,0.8);
\coordinate (eb1) at (0.5,-0.7);
\coordinate (eb0) at (0.2,-1);
\coordinate (Ca) at (0,0.4);
\coordinate (Cb) at (0,-0.5);

\path[-,font=\scriptsize,line width=1.2pt,gray]
(ea0) edge[out=-45,in=135,line width=1.5pt,black]
node[pos=0.5,sloped,allow upside down]{\midarrow}
node[pos=0.5,anchor=south west]{\large $\ve_a$} (ea1)
(eb0) edge[out=-135,in=45,line width=1.5pt,black]
node[pos=0.5,sloped,allow upside down]{\midarrow}
node[pos=0.5,anchor=north west]{\large $\ve_b$} (eb1)
(v) edge[out=180,in=0] node[pos=0.5,sloped,allow upside down]{\midarrow} (w)
(ea1) edge[out=-45,in=45] (v)
(eb1) edge[out=45,in=-45] (v)
(ea0) edge[out=135,in=135] (w)
(eb0) edge[out=-135,in=-135] (w);

\fill[black] (v) circle (2.5pt);
\draw (v) node[anchor=west] {\large $v$};
\fill[black] (w) circle (2.5pt);
\draw (w) node[anchor=east] {\large $w$};
\fill[black] (ea1) circle (2.5pt);
\fill[black] (ea0) circle (2.5pt);
\fill[black] (eb1) circle (2.5pt);
\fill[black] (eb0) circle (2.5pt);

\draw[->,line width=1pt,dash pattern=on 2pt off 1pt] (.3,0.4) arc (0:-270:0.3);
\draw (Ca) node {\Large $C_a$};
\draw[->,line width=1pt,dash pattern=on 2pt off 1pt] (.3,-0.5) arc (0:270:0.3);
\draw (Cb) node {\Large $C_b$};
\end{scope}

\begin{scope}[shift={(2,0)}]
\draw (-0.6,1.2) node {(b)};

\coordinate (v) at (0,0);
\coordinate (ea1) at (0.46,0.46*2-0.3);
\coordinate (ea0) at (0.46-0.3,0.46*2);
\coordinate (eb1) at (0.56,-0.56*2+0.3);
\coordinate (eb0) at (0.56-0.3,-.56*2);
\coordinate (Ca) at (0,0.46);
\coordinate (Cb) at (0,-0.56);

\draw[line width=1.2pt,gray] (.46,0.46) arc (0:-270:0.46);
\draw[line width=1.2pt,gray] (.56,-.56) arc (0:270:0.56);

\path[-,font=\scriptsize,line width=1.2pt,gray]
(ea0) edge[out=-45,in=135,line width=1.5pt,black]
node[pos=0.5,sloped,allow upside down]{\midarrow}
node[pos=0.5,anchor=south west]{\large $\ve_a$} (ea1)
(eb0) edge[out=-135,in=45,line width=1.5pt,black]
node[pos=0.5,sloped,allow upside down]{\midarrow}
node[pos=0.5,anchor=north west]{\large $\ve_b$} (eb1)
(ea1) edge[out=-90,in=90] (0.46,0.46)
(eb1) edge[out=90,in=-90] (0.56,-0.56)
(ea0) edge[out=180,in=0] (0,.46*2)
(eb0) edge[out=180,in=0] (0,-.56*2);

\fill[black] (v) circle (2.5pt);
\draw (v) node[anchor=north] {\large $v$};
\fill[black] (ea1) circle (2.5pt);
\fill[black] (ea0) circle (2.5pt);
\fill[black] (eb1) circle (2.5pt);
\fill[black] (eb0) circle (2.5pt);

\draw[->,line width=1pt,dash pattern=on 2pt off 1pt] (.3,0.46) arc (0:-270:0.3);
\draw (Ca) node {\Large $C_a$};
\draw[->,line width=1pt,dash pattern=on 2pt off 1pt] (.3,-0.55) arc (0:270:0.3);
\draw (Cb) node {\Large $C_b$};
\end{scope}

\begin{scope}[shift={(4,0)}]
\draw (-0.6,1.2) node {(c)};

\coordinate (v) at (0,0.3);
\coordinate (w) at (0,-0.3);
\coordinate (ea1) at (0.46,0.3+0.46*2-0.3);
\coordinate (ea0) at (0.46-0.3,0.3+0.46*2);
\coordinate (eb1) at (0.56,-0.3-0.56*2+0.3);
\coordinate (eb0) at (0.56-0.3,-0.3-.56*2);
\coordinate (Ca) at (0,0.3+0.46);
\coordinate (Cb) at (0,-0.3-0.56);

\draw[line width=1.2pt,gray] (.46,.3+0.46) arc (0:-270:0.46);
\draw[line width=1.2pt,gray] (.56,-.3-.56) arc (0:270:0.56);

\path[-,font=\scriptsize,line width=1.2pt,gray]
(ea0) edge[out=-45,in=135,line width=1.5pt,black]
node[pos=0.5,sloped,allow upside down]{\midarrow}
node[pos=0.5,anchor=south west]{\large $\ve_a$} (ea1)
(eb0) edge[out=-135,in=45,line width=1.5pt,black]
node[pos=0.5,sloped,allow upside down]{\midarrow}
node[pos=0.5,anchor=north west]{\large $\ve_b$} (eb1)
(v) edge[out=-90,in=90] node[pos=0.5,sloped,allow upside down]{\midarrow} (w)
(ea1) edge[out=-90,in=90] (0.46,.3+0.46)
(eb1) edge[out=90,in=-90] (0.56,-.3-0.56)
(ea0) edge[out=180,in=0] (0,.3+.46*2)
(eb0) edge[out=180,in=0] (0,-.3-.56*2);

\fill[black] (v) circle (2.5pt);
\draw (v) node[anchor=north east] {\large $v$};
\fill[black] (w) circle (2.5pt);
\draw (w) node[anchor=south east] {\large $w$};
\fill[black] (ea1) circle (2.5pt);
\fill[black] (ea0) circle (2.5pt);
\fill[black] (eb1) circle (2.5pt);
\fill[black] (eb0) circle (2.5pt);

\draw[->,line width=1pt,dash pattern=on 2pt off 1pt] (.3,.3+0.46) arc (0:-270:0.3);
\draw (Ca) node {\Large $C_a$};
\draw[->,line width=1pt,dash pattern=on 2pt off 1pt] (.3,-.3-0.55) arc (0:270:0.3);
\draw (Cb) node {\Large $C_b$};
\end{scope}

\end{tikzpicture}
\caption{The three relative positions of the fundamental closed walks $C_a$ and $C_b$: (a) their tree parts overlap in a path; (b) they meet in one vertex; (c) they are disjoint and joined by a minimal tree path.}
\label{F:two-edges}
\end{figure}

\emph{Step 1: cotree character values are signs.}
Let $\ve_b$ be arbitrary.  Since $g\geq2$, choose $\ve_a$ with a different underlying edge.  In each of the three cases in Figure~\ref{F:two-edges}, decomposition at the common vertices $v,w$ gives two non-backtracking edge-walks $\vDelta,\vDelta'$ from $\ve_a$ to $\ve_a^{-1}$ whose interior walks are
$$\vDelta^\times
=T(\ve_a(1),\ve_b(0))\cdot\ve_b\cdot T(\ve_b(1),\ve_a(1))$$
and
$$\vDelta'^\times
=T(\ve_a(1),\ve_b(1))\cdot\ve_b^{-1}\cdot T(\ve_b(0),\ve_a(1)).$$
The two tree portions have the same total length in all three cases: common branches occur once on both sides, while the maximal overlap or connecting path occurs twice on both sides.  Thus $\vDelta$ and $\vDelta'$ have the same length.  Since $\omega$ vanishes on tree edges,
$$\chi_\omega(\vDelta)=\chi_\omega(\ve_b),
\qquad
\chi_\omega(\vDelta')=\chi_\omega(\ve_b^{-1})
=\chi_\omega(\ve_b)^{-1}.$$
Coherence implies
$$\chi_\omega(\ve_b)=\chi_\omega(\ve_b)^{-1},$$
and hence $\chi_\omega(\ve_b)=\pm1$.  This holds for every cotree edge.  Therefore equality of the two spectral radii can occur only at a $2$-torsion character, proving (ii).

\emph{Step 2: refinement when $G$ is non-bipartite.}
Assume now that $G$ is non-bipartite.  In the notation of Construction~\ref{Cs:CanoForm2}, set
$$E_0:=\{\ve\in\vE_{\mfrako}(G)\setminus T:
\tau(T(\ve(1),\ve(0)))\text{ is odd}\}$$
and
$$E_1:=\{\ve\in\vE_{\mfrako}(G)\setminus T:
\tau(T(\ve(1),\ve(0)))\text{ is even}\}.$$
The fundamental walk $C_\ve$ has even length for $\ve\in E_0$ and odd length for $\ve\in E_1$.  The set $E_1$ is empty exactly when $G$ is bipartite, so it is nonempty under the present hypothesis.

We compare the signs attached to $\ve_a$ and $\ve_b$ by constructing two closed non-backtracking edge-walks based at $\ve_a$.

\emph{Strategy (I).} Wind $k$ times around $C_a$:
$$\Delta_k:=C_a^k,
\qquad
\tau(\Delta_k)=k\tau_a,
\qquad
\chi_\omega(\Delta_k)=\chi_\omega(\ve_a)^k.$$

\emph{Strategy (II).} Wind $k'$ times around
$$D_{ab}:=\ve_a\cdot T(\ve_a(1),\ve_b(1))
\cdot\ve_b^{-1}\cdot T(\ve_b(0),\ve_a(0)).$$
Write $\tau_{ab}:=\tau(D_{ab})$.  In the three relative positions of Figure~\ref{F:two-edges},
$$\tau_{ab}=\begin{cases}
\tau_a+\tau_b-2\tau(T(v,w)),&\text{in Case (a),}\\
\tau_a+\tau_b,&\text{in Case (b),}\\
\tau_a+\tau_b+2\tau(T(v,w)),&\text{in Case (c).}
\end{cases}$$
In particular,
$$\tau_{ab}\equiv\tau_a+\tau_b\pmod2.$$

Choose
$$k:=\frac{\tau_{ab}}{\gcd(\tau_a,\tau_{ab})},
\qquad
k':=\frac{\tau_a}{\gcd(\tau_a,\tau_{ab})}.$$
Then $k\tau_a=k'\tau_{ab}$, so $\Delta_k$ and $D_{ab}^{k'}$ are coherent.  Since the cotree character values are signs,
$$\chi_\omega(\ve_a)^k
=\chi_\omega(\ve_a)^{k'}\cdot\chi_\omega(\ve_b)^{k'}.$$
We now distinguish the four parity cases.

\emph{Case (1):} If $\tau_a$ and $\tau_b$ are both odd, then $\tau_{ab}$ and $k$ are even while $k'$ is odd.  Hence
$$\chi_\omega(\ve_a)=\chi_\omega(\ve_b).$$

\emph{Case (2):} If $\tau_a$ is even and $\tau_b$ is odd, then $\tau_{ab}$ and $k$ are odd while $k'$ is even.  Hence
$$\chi_\omega(\ve_a)=1.$$

\emph{Case (3):} If $\tau_a$ is odd and $\tau_b$ is even, then $\tau_{ab}$, $k$, and $k'$ are odd.  Hence
$$\chi_\omega(\ve_b)=1.$$

\emph{Case (4):} If $\tau_a$ and $\tau_b$ are both even, the parity identity gives no further restriction.

Because $E_1$ is nonempty, Case (1) shows that all character values on $E_1$ have one common sign, while Cases (2) and (3) show that every character value on $E_0$ is $1$.  If the common sign on $E_1$ is $1$, then $\underline{\omega}=0$.  If it is $-1$, then $\underline{\omega}=\theta$ by Construction~\ref{Cs:CanoForm2} and Lemma~\ref{L:O-TO}.  Thus, whenever $G$ is non-bipartite, equality of the edge spectral radii can occur only at $0$ and $\theta$, proving (i).
\end{proof}
}

\begin{proof}[Proof of Theorem~\ref{T:SpecDuality}]
The spectral identities in (SA1) follow from Lemma~\ref{L:spectral-neg2}.  In particular the same character $\theta$ works for every positive directed weight $W$, because its construction is independent of $W$.  Part (SA2) is Lemma~\ref{L:unique} together with Lemma~\ref{L:O-TO}; the $2$-torsion assertion follows because $\theta(C)^2=1$ for every closed walk $C$.  Part (SA3) follows because $\theta=0$ is equivalent to every closed walk having even length, which is equivalent to $G$ being bipartite.  If $G$ is bipartite, then $\theta=0$, and (SA1) gives invariance under negation for every unitary character.  Part (SA4) is Lemma~\ref{L:weighted-spectral-rigidity}.

The period-character assertions (PC1)--(PC2) follow from
Lemma~\ref{L:spe-ineq-genus2}; the non-isospectrality in (PC2) also follows
from (SA4), since $\eta$ is nontrivial.  For the spectral-radius assertions,
(SR1) is Lemma~\ref{L:incoherent} for vertex matrices and
Lemma~\ref{L:edge-incoherent} for edge matrices.  Part (SR2) is
Lemma~\ref{L:spe-ineq}.  Part (SR3) is Lemma~\ref{L:spe-eq-genus1}, and
parts (SR4)--(SR5) are Lemma~\ref{L:spe-ineq-genus2}.
\end{proof}

We finish the section by proving Proposition~\ref{P:period-character-algorithm}.

\begin{proof}[Proof of Proposition~\ref{P:period-character-algorithm}]
First, every stage is finite.  Each trimming step deletes at least one vertex,
so the construction of $G^\circ$ terminates after at most $|V(G)|$ steps.
The cyclic core has finitely many oriented edges and feeding transitions, and
its period $\nu_G$, as the period of its finite edge adjacency matrix, can be
determined in finite time.  The predecessor choices require only finite
searches in this transition relation.  Step~3 then assigns one residue class
to each oriented edge and checks one congruence for each remaining feeding
transition.  Hence the algorithm terminates.

The fundamental cycles $C_1,\ldots,C_g$ form a basis of
$H_1(G,\mbbZ)$.  Each $C_i$ is a prime cycle, so
$\nu_G\mid\tau(C_i)$.  Consequently
\eqref{E:period-cotree-form} represents the unique candidate $2$-torsion
character $\eta_T$ whose values on this basis are
$$
\eta_T(C_i^{\ab})=(-1)^{q_i}=(-1)^{\tau(C_i)/\nu_G}.
$$
A gain representative of $\eta_T$ is
$$
\gamma_T(\ve):=(-1)^{c_T(\ve)};
$$
it equals $1$ on tree edges and has the prescribed sign on each cotree edge
and its inverse.

Put $\zeta:=e(1/(2\nu_G))$.  For every allowed transition
$\va\shortrightarrow\vb$, define the transition factor
$$
g_T(\va,\vb):=\zeta^{-1}\gamma_T(\vb).
$$
If
$$
P=\va_0\va_1\cdots\va_{l-1}
$$
is a closed nonbacktracking edge-walk, where
$\va_j\shortrightarrow\va_{j+1}$ and the indices are read modulo $l$, then
$$
\prod_{j=0}^{l-1}g_T(\va_j,\va_{j+1})
=\zeta^{-l}\prod_{j=0}^{l-1}\gamma_T(\va_j)
=\zeta^{-l}\eta_T(P^{\ab}).
$$
Since $\nu_G\mid l$, this product equals $1$ exactly when
$$
\eta_T(P^{\ab})=\zeta^l=(-1)^{l/\nu_G}.
$$
Every closed nonbacktracking edge-walk is a positive power of a prime cycle,
so the defining identity for the period character extends to all such $P$.
It follows that $\eta_T$ is the period character if and only if the product
of the transition factors is $1$ along every closed nonbacktracking
edge-walk.

We now express this closed-edge-walk condition as the finite congruence test
in Step~3.  By Lemma~\ref{L:genus-2-edge-walk}, $G^\circ$ is
edge-walk-connected.  Fix a base oriented edge $\va_\ast$.  If every closed
edge-walk has transition-factor product $1$, define $d(\va)$ as the product
of the transition factors along any nonbacktracking edge-walk from
$\va_\ast$ to $\va$.  This is independent of the chosen edge-walk: given two
such edge-walks, append to both an edge-walk from $\va$ back to
$\va_\ast$ and use the closed-edge-walk condition.  Hence
$$
g_T(\va,\vb)=d(\va)^{-1}d(\vb)
$$
for every allowed transition $\va\shortrightarrow\vb$.  Conversely, any
function $d$ satisfying these identities makes the product telescope to $1$
along every closed edge-walk.

Every transition factor is a power of $\zeta$, so the preceding construction
allows $d(\va)=\zeta^{p(\va)}$ with
$p(\va)\in\mbbZ/(2\nu_G)\mbbZ$.  Since
$\gamma_T(\vb)=\zeta^{\nu_Gc_T(\vb)}$, the transition identity becomes
$$
p(\vb)-p(\va)
\equiv\nu_Gc_T(\vb)-1\pmod{2\nu_G},
$$
which is precisely~\eqref{E:period-character-congruence}.  Assigning $p$
along the chosen predecessor steps determines the only possible values after
the normalization $p(\va_\ast)=0$; checking the remaining allowed
transitions is therefore equivalent to solvability of the full system.

When the test succeeds,
$\eta_T=\underline{\omega^{\mathrm{per}}_{(T,\mfrako)}}$ satisfies the
defining normalized-parity identity and is therefore the period character.
This intrinsic characterization proves both the asserted correctness and
the independence of the output character and the success or failure of the
test from $T$ and $\mfrako$.
\end{proof}

\bibliographystyle{alpha}

\newpage
\appendix

\section{Notation guide}
\label{A:notation-guide}

\subsubsection*{Section~\ref{S:intro}}\quad

\begin{longtable}{@{}p{.29\textwidth}p{.65\textwidth}@{}}
$G$ & a connected finite graph allowing multiple edges and loops \\
$g$ & the genus (first Betti number) of $G$ \\
$V(G)$ & the vertex set of $G$ \\
$E(G)$ & the edge set of $G$ \\
$\vE(G)$ & the set of all oriented edges of $G$ \\
$\vE_{\mfrako}(G)$ & the set of all positively oriented edges with respect to an orientation $\mfrako$ of $G$ \\
$\va(0)$ and $\va(1)$ & the initial and terminal vertices of an oriented edge $\va$\\
$\va\shortrightarrow\vb$ & $\va$ feeds into $\vb$, meaning $\va(1)=\vb(0)$ and $\vb\neq\va^{-1}$ \\
$\Delta(0)$ and $\Delta(1)$ & the initial and terminal vertices of a walk $\Delta$\\
$\vDelta(0)$ and $\vDelta(1)$ & the initial and terminal oriented edges of an edge-walk $\vDelta$\\
$\tau(\Delta)$ & the length of $\Delta$ \\
$z_W(u)$ & the weighted Ihara zeta function; $z(u)=z_1(u)$ \\
$\mcalT(G)$ or $\mcalT$ & the tangent space of $G$ (elements being vector fields) \\
$\Omega(G)$ or $\Omega$ & the cotangent space of $G$ (elements being $1$-forms) \\
$\Delta^{\ab}$ & the abelianization of a walk $\Delta$ \\
$T(x,y)$ & the unique path in a spanning tree $T$ from $x$ to $y$ \\
$\chi_\omega$ & the unitary character associated to a $1$-form $\omega$ \\
$L_W(u,\chi_\omega)$ & the weighted L-function; $L(u,\chi_\omega)=L_1(u,\chi_\omega)$
\end{longtable}

\subsubsection*{Section~\ref{S:prelim}}\quad

\begin{longtable}{@{}p{.29\textwidth}p{.65\textwidth}@{}}
$\mcalH^1(G)$ & the space of harmonic $1$-forms on $G$ \\
$ \imag(d)$ & the space of exact $1$-forms on $G$ \\
$\omega=\phi_1(\omega)+\phi_2(\omega)$ & orthogonal decomposition of a $1$-form $\omega$, with $\phi_1(\omega)\in\mcalH^1(G)$ and $\phi_2(\omega)\in\imag(d)$ \\
$H_1(G,\mbbR)$ & the first real homology group of $G$ \\
$H_1(G,\mbbZ)$ & the first integral homology group of $G$, which is a full-rank lattice in $H_1(G,\mbbR)$ \\
$H_1(G,\mbbZ)^\vee$ & the dual of $H_1(G,\mbbZ)$, which is a full-rank lattice in $\mcalH^1(G)$  \\
$e(\cdot)$ & $e(x):=\exp\big(2\pi\sqrt{-1}x\big)$ \\
$\mcalX(G)$ or $\mcalX$ &  $\mcalH^1(G)/H_1(G,\mbbZ)^\vee$, the unitary character group of $G$   \\
$\mcalX_\mbbC(G)$ or $\mcalX_\mbbC$ & $\Hom(H_1(G,\mbbZ),\mbbC^\times)$, the full complex character torus of $G$ \\
$\underline{\omega}$ & $\phi_1(\omega)+H_1(G,\mbbZ)^\vee$, the projection of a $1$-form $\omega$ to $\mcalX$ \\
$\gamma$ & a complex gain representing a full character $\xi\in\mcalX_\mbbC(G)$ \\
$\Lambda$ & a full-rank sublattice of  $H_1(G,\mbbZ)$\\
$\Lambda^\vee$ & the dual lattice of $\Lambda$ \\
$\mcalX_\Lambda$ & $\mcalH^1(G)/\Lambda^\vee$ \\
$\dunderline{\omega}$ & $\phi_1(\omega)+\Lambda^\vee$, the  projection of a $1$-form $\omega$ to $\mcalX_\Lambda$ \\
$Q_\Lambda$ & $H_1(G,\mbbZ)/\Lambda$ \\
$\underline{\alpha}$ & $\alpha+\Lambda\in Q_\Lambda$ for $\alpha\in H_1(G,\mbbZ)$ \\
$\widehat{Q_\Lambda}$  & $\Lambda^\vee/H_1(G,\mbbZ)^\vee$ \\
$A$ & the adjacency matrix of $G$ \\
$B$ & the edge adjacency matrix of $G$ \\
$A_\omega$ & the adjacency matrix of $G$ twisted by $\chi_\omega$ \\
$B_{\omega}$ & the edge adjacency matrix of $G$ twisted by $\chi_\omega$ \\
$W$ & a positive directed weight $W:\vE(G)\to\mbbR_{>0}$ \\
$W\equiv 1$ & the trivial weight, satisfying $W(\ve)=1$ for all $\ve\in\vE(G)$ \\
$W(\Delta)$ & the product of the directed weights along a walk $\Delta$ \\
$W(\vDelta)$ & $W(\Delta)$, where $\Delta$ is the associated walk of the edge-walk $\vDelta$ \\
$W^r$ & the pointwise powered weight $W^r(\ve)=W(\ve)^r$ \\
special weights & the symmetric and reciprocal cases of $W$ \\
$A_{\omega,W}$ & the twisted weighted vertex adjacency matrix \\
$B_{\omega,W}$ & the twisted weighted edge adjacency matrix \\
$A_{\gamma,W}$, $B_{\gamma,W}$ & the twisted weighted matrices associated to a complex gain $\gamma$ \\
$\xi$ & a full character in $\mcalX_\mbbC(G)$; in the Floquet subsection, a unitary character in $\mcalX(G)$ \\
$\spec A_{\xi,W}$, $\spec B_{\xi,W}$ & the spectra associated to a full character $\xi$ \\
$\spec M$ & the spectrum of a square matrix $M$ 
\end{longtable}

\subsubsection*{Section~\ref{S:main}}\quad

\begin{longtable}{@{}p{.29\textwidth}p{.65\textwidth}@{}}
$\rho(M)$ & the spectral radius of a square matrix $M$ \\
$\theta$ & the canonical character of $G$ \\
$\eta$ & the period character of $G$, when normalized length parity descends to homology \\
$\varepsilon_G$ & the normalized length-parity homomorphism defining $\eta$ \\
$\omega^{\mathrm{per}}_{(T,\mfrako)}$ & the cotree-supported candidate in Algorithm~\ref{A:period-character-algorithm}; a period \\& $1$-form when the test succeeds \\
$\Sigma_{\operatorname{Floq}}(W)$ & the Floquet spectrum of the symmetric weighted vertex family \\
$\lambda_j(\xi)$ & the $j$-th band function over the character torus \\
$\nu_G$ & the nonbacktracking period, or period, of $G$ \\
$\mathcal E_G(\Lambda)$ & the weight-independent extremal characters in $\widehat{Q_\Lambda}$ for $B_{\omega,W}$ \\
$N(l)$ & the number of circuits of length $l$ \\
$N(\alpha,l)$ & the number of circuits $C$ of length $l$ such that $C^{\ab}=\alpha\in H_1(G,\mbbZ)$ \\
$N(\underline{\alpha},l)$ & the number of circuits $C$ of length $l$ such that $\underline{C^{\ab}}=\underline{\alpha}\in Q_\Lambda$ \\
$N_W(\alpha,l)$ & the weighted sum of circuits $C$ of length $l$ such that $C^{\ab}=\alpha$ \\
$N_W(\underline{\alpha},l)$ & the weighted sum of circuits $C$ of length $l$ such that $\underline{C^{\ab}}=\underline{\alpha}$ \\
$\mcalK(\omega,l)$ & the trace distribution function of order $l$ \\
$\mcalK_W(\omega,l)$ & the weighted trace distribution function $\operatorname{tr}(B_{\omega,W}^l)$ \\
$\pi(l)$ & the number of prime cycles of length $l$ \\
$\Pi(l)$ & the number of cycles of length $l$ \\
$\pi(\alpha,l)$ & the number of prime cycles $P$ of length $l$ such that $P^{\ab}=\alpha\in H_1(G,\mbbZ)$ \\
$\Pi(\alpha,l)$ & the number of cycles $C$ of length $l$ such that $C^{\ab}=\alpha\in H_1(G,\mbbZ)$ \\
$\pi(\underline{\alpha},l)$ & the number of prime cycles $P$ of length $l$ such that $\underline{P^{\ab}}=\underline{\alpha}\in Q_\Lambda$ \\
$\Pi(\underline{\alpha},l)$ & the number of cycles $C$ of length $l$ such that $\underline{C^{\ab}}=\underline{\alpha}\in Q_\Lambda$ \\
$\pi_W(\alpha,l)$ & the $W$-weighted sum of prime cycles in the class $\alpha$ \\
$\Pi_W(\alpha,l)$ & the $W$-weighted sum of cycles in the class $\alpha$ \\
$\pi_W(\underline{\alpha},l)$ & the $W$-weighted sum of prime cycles in the class $\underline{\alpha}$ \\
$\Pi_W(\underline{\alpha},l)$ & the $W$-weighted sum of cycles in the class $\underline{\alpha}$
\end{longtable}

\subsubsection*{Section~\ref{S:examples}}\quad

\begin{longtable}{@{}p{.29\textwidth}p{.65\textwidth}@{}}
$G_1,G_2,G_3$ & the three genus-$2$ graphs in Subsection~\ref{SS:example_spec} \\
$W^\star$ & the nontrivial positive symmetric weight used in each example \\
$W_{a,b}$ & the two-parameter symmetric weight on $K_4$
\end{longtable}

\subsubsection*{Section~\ref{S:proof-antisymmetry}}\quad

\begin{longtable}{@{}p{.29\textwidth}p{.65\textwidth}@{}}
$\omega_{\mfrako}$ & a canonical $1$-form associated with an orientation $\mfrako$, as in \\& Lemma~\ref{L:unique} and Definition~\ref{D:canonical-character}  \\
$\omega_{(T,\mfrako)}$ & a canonical $1$-form with respect to an orientation $\mfrako$ and a spanning\\& tree $T$, which is the output of Construction~\ref{Cs:CanoForm2} \\
 $\mcalX_{(T,\mfrako)}$ & $\{\sum_{i=1}^g\omega_id\ve_i\mid 0\leq \omega_i< 1\}$ where  $\{\ve_1,\cdots,\ve_g\}=\vE_{\mfrako}(G)\setminus T$
\end{longtable}

\section{Proofs of preliminary facts}
\label{B:prelim-proofs}
This appendix contains the proofs of the statements collected in Section~\ref{S:prelim}.

\begin{proof}[Proof of Lemma~\ref{L:harmonic}]
Since $\Delta\omega=dd^*\omega$, we have $\Delta\omega=0$ if and only if
$d^*\omega$ is constant.  On the other hand,
$\sum_{v\in V(G)}d^*\omega(v)=0$, so a constant function of the form
$d^*\omega$ must vanish.  Hence $\omega$ is harmonic if and only if
$d^*\omega=0$.  Writing
$\omega=\sum_{\ve\in\vE_{\mfrako}(G)}\omega_\ve\,d\ve$, the value
$d^*\omega(v)$ is the sum of the coefficients on edges entering $v$ minus
the sum of the coefficients on edges leaving $v$.  Thus $d^*\omega=0$ is
equivalent to the stated balance condition at every vertex.
\end{proof}

\begin{proof}[Proof of Proposition~\ref{P:hodge}]
Consider the function $d^*\omega$.  Since
$\sum_{v\in V(G)}d^*\omega(v)=0$, the equation
$\Delta f=d^*df=d^*\omega$ has a solution, unique up to addition of a
constant.  Indeed, because $G$ is connected, the Laplacian has rank $n-1$
and its kernel consists of the constant functions.  Then $df$ is exact and
$\omega-df$ is harmonic by Lemma~\ref{L:harmonic}.  Write
$df=\sum_{\ve\in\vE_{\mfrako}(G)}\omega'_\ve\,d\ve$ and
$\omega-df=\sum_{\ve\in\vE_{\mfrako}(G)}\omega''_\ve\,d\ve$.  Since
$\omega'_\ve=f(\ve(1))-f(\ve(0))$, we have

\begin{align*}
&\sum_{\ve\in \vE_{\mfrako}(G)}\omega'_\ve\omega''_\ve = \sum_{\ve\in \vE_{\mfrako}(G)}\Big(f(\ve(1))-f(\ve(0))\Big)\omega''_\ve \\
&= \sum_{v\in V(G)}f(v)\Bigg(\sum_{\ve\in \vE_{\mfrako}(G),\ve(1)=v}\omega''_\ve-\sum_{\ve\in \vE_{\mfrako}(G),\ve(0)=v}\omega''_\ve\Bigg)=0
\end{align*}
Finally, because $G$ is connected, $\ker(d)$ consists of the constant
functions, and hence
$$
\dim\operatorname{Im}(d)=\operatorname{rank}(d)=n-1.
$$
Since $d$ and its adjoint $d^*$ have the same rank and
$\mathcal H^1(G)=\ker(d^*)$, the rank--nullity theorem gives
$$
\dim\mathcal H^1(G)=m-\operatorname{rank}(d^*)
=m-(n-1)=g.
$$
Thus these dimensions add to $m=\dim C^1(G,\mathbb R)$, as required for
the orthogonal direct-sum decomposition above.
\end{proof}

\begin{proof}[Proof of Lemma~\ref{L:GraphHomology}]
Let $T$ be a spanning tree of $G$. Then $G\setminus T$ contains $g$ edges. Fix an orientation of $G$. We may label the positively oriented cotree edges as $\ve_1,\ldots,\ve_g$ and the positively oriented tree edges as $\ve_{g+1},\ldots,\ve_m$. Let $K$ be the free $R$-module on $\{\ve_1,\ldots,\ve_g\}$. Clearly $K\simeq R^g$, and we claim that $H_1(G,R)\simeq K$. Consider the restriction homomorphism $\rho:H_1(G,R)\to K$ sending $\sum_{i=1}^m c_i \ve_i\in H_1(G,R)$ to $\sum_{i=1}^g c_i \ve_i\in K$. For each $i=1,\ldots,g$, let
$$u_i=\ve_i+T(\ve_i(1),\ve_i(0))^{\ab}.$$
Then $u_i\in H_1(G,\mbbZ)$, and
$\rho(\sum_{i=1}^g c_iu_i)=\sum_{i=1}^g c_i\ve_i$, so $\rho$ is
surjective.  If an element of $\ker\rho$ is written as
$\sum_{i=g+1}^m c_i\ve_i$, it is a $1$-cycle supported entirely on the
tree $T$.  A tree has no nonzero $1$-cycles, so all $c_i$ vanish.  Hence
$\rho$ is injective.

(b) is straightforward to verify. 

For (c), consider the canonical homomorphism
$$
\phi:H_1(G,\mbbZ)\longrightarrow H_1(G,\mbbZ/t\mbbZ),
\qquad
\sum_{i=1}^m c_i\ve_i\longmapsto
\sum_{i=1}^m \overline{c}_i\ve_i,
$$
where $\overline{c}_i$ is the reduction of $c_i$ modulo $t$. Then
$\ker\phi=tH_1(G,\mbbZ)$, and it remains to prove surjectivity.
As in the proof of (a), let $T$ be a spanning tree whose positively oriented
cotree edges are $\ve_1,\ldots,\ve_g$.  The isomorphism in (a) identifies
$H_1(G,R)$ with the free $R$-module on these cotree edges.  We may therefore
lift an element of $H_1(G,\mbbZ/t\mbbZ)$ by lifting its $g$ cotree
coefficients from $\mbbZ/t\mbbZ$ to $\mbbZ$.  This proves that $\phi$ is
surjective.
\end{proof}

\begin{proof}[Proof of Proposition~\ref{P:basis_gen}]
Using the orthogonal decomposition in Proposition~\ref{P:hodge}, we have $d\ve_i = \phi_1(d\ve_i)+\phi_2(d\ve_i)$ where $\phi_1(d\ve_i)\in\mcalH^1(G)$ is a harmonic $1$-form and $\phi_2(d\ve_i)\in \imag(d)$ is an exact form. 

For a spanning tree $T$ of $G$, let
$$u_i=\ve_i+T(\ve_i(1),\ve_i(0))^{\ab},\qquad i=1,\ldots,g,$$
as in the proof of Lemma~\ref{L:GraphHomology}. Then $\{u_1,\ldots,u_g\}$ is a basis of $H_1(G,\mbbZ)$. Since no cotree edge $\ve_i$ occurs in the tree path $T(\ve_j(1),\ve_j(0))$, we have
$$d\ve_i(u_j)=d\ve_i(\ve_j)+d\ve_i\bigl(T(\ve_j(1),\ve_j(0))^{\ab}\bigr)=d\ve_i(\ve_j)=\delta_{ij}$$
for all $i,j=1,\cdots,g$. 

Note that $d\ve_i(u)=\phi_1(d\ve_i)(u)$ for all $u\in H_1(G,\mbbZ)$. This implies $\phi_1(d\ve_i)(u_j)=\delta_{ij}$, i.e., $\{\phi_1(d\ve_1), \cdots, \phi_1(d\ve_g)\}$ is a basis of $H_1(G,\mbbZ)^\vee$ as desired, which is the dual to the basis $\{u_1,\cdots,u_g\}$ of $H_1(G,\mbbZ)$.

It remains to show that
$H_1(G,\mbbZ)^\vee=\phi_1(C^1(G,\mbbZ))$.  The cotree forms considered
above already show that
$H_1(G,\mbbZ)^\vee\subseteq\phi_1(C^1(G,\mbbZ))$.  For the reverse
inclusion, it is enough to prove that
$\phi_1(d\ve)\in H_1(G,\mbbZ)^\vee$ for every oriented edge $\ve$.  If the
underlying edge $e$ is a bridge, then $d\ve$ is exact and
$\phi_1(d\ve)=0$.  Otherwise, there is a spanning tree not containing $e$;
using that tree makes $\ve$ a cotree edge, and the first part of the proof
shows that $\phi_1(d\ve)$ belongs to $H_1(G,\mbbZ)^\vee$.  Recall that $e$
is a bridge exactly when $G\setminus e$ is disconnected, equivalently when
$e$ belongs to every spanning tree.  This completes the proof.
\end{proof}

\begin{proof}[Proof of Lemma~\ref{L:PontrDual}]
We make the standard facts about Pontryagin duals explicit in our setting:
the Pontryagin dual of a finite-dimensional real vector space is a real
vector space of the same dimension, and the Pontryagin dual of a rank-$g$
lattice is a $g$-dimensional real torus.

First note that since the unit circle $S^1\simeq \mbbR/\mbbZ$ is an injective $\mbbZ$-module, the sequence $H_1(G,\mbbZ)\hookrightarrow H_1(G,\mbbR) \hookrightarrow \mcalT(G)$ induces the sequence on their Pontryagin duals $\widehat{\mcalT(G)} \twoheadrightarrow \widehat{H_1(G,\mbbR)} \twoheadrightarrow \widehat{H_1(G,\mbbZ)}$. 

We next show that the homomorphism
$\Omega(G)\to\widehat{\mcalT(G)}$ defined by
$\omega\mapsto\chi_\omega$ is an isomorphism.  If
$\omega=\sum_{i=1}^m\omega_i\,d\ve_i$ is nonzero, say
$\omega_1\neq0$, choose $c\in\mbbR$ such that $c\omega_1\notin\mbbZ$.
Then $\chi_\omega(c\ve_1)=e(c\omega_1)\neq1$, proving injectivity.
Conversely, after identifying $\mcalT(G)$ with $\mbbR^m$, every continuous
character of $\mbbR^m$ has the form
$$
(c_1,\ldots,c_m)\longmapsto
e\left(\sum_{i=1}^m\omega_i c_i\right)
$$
for uniquely determined $\omega_1,\ldots,\omega_m\in\mbbR$.  Taking
$\omega=\sum_{i=1}^m\omega_i\,d\ve_i$ proves surjectivity.

By Proposition~\ref{P:hodge}, $\omega=\phi_1(\omega)+\phi_2(\omega)$ where $\phi_1(\omega)\in \mcalH^1(G)$ and $\phi_2(\omega)\in \imag(d)$. Then $\chi_\omega|_{H_1(G,\mbbR)}=\chi_{\phi_1(\omega)}|_{H_1(G,\mbbR)}$, which induces a compatible isomorphism between $\mcalH^1(G)$ and $\widehat{H_1(G,\mbbR)}$. 

It remains to show that there is a compatible isomorphism between $\mcalX(G)$ and $\widehat{H_1(G,\mbbZ)}$. For $\omega,\omega'\in\Omega(G)$,   $\chi_{\omega}|_{H_1(G,\mbbZ)}=\chi_{\omega'}|_{H_1(G,\mbbZ)}$, if and only if $e(\omega(\alpha))=e(\omega'(\alpha))$ for all $\alpha\in H_1(G,\mbbZ)$, if and only if $\omega(\alpha)-\omega'(\alpha)\in \mbbZ$ for all $\alpha\in H_1(G,\mbbZ)$, if and only if $\phi_1(\omega)-\phi_1(\omega')\in H_1(G,\mbbZ)^\vee$. Thus, we have $\widehat{H_1(G,\mbbZ)}\simeq \mcalH^1(G)/H_1(G,\mbbZ)^\vee$. 
\end{proof}

\begin{proof}[Proof of Proposition~\ref{P:CharacterOrth}]
We identify the domain of integration $\mcalX$ with the fundamental domain of the lattice $H_1(G,\mbbZ)^\vee$ in $\mcalH^1(G)$. 

Let $\{u_1,\cdots,u_g\}$ be a basis of $H_1(G,\mbbZ)$ and $\{u^1,\cdots,u^g\}$ be the corresponding dual basis of $H_1(G,\mbbZ)^\vee$, i.e., $u^i(u_j)=\delta_{ij}$. Then we can write $\alpha=\sum_{i=1}^g \alpha_iu_i$, $\beta=\sum_{i=1}^g\beta_iu_i$ and $\omega=\sum_{i=1}^g \omega_iu^i$. 

Now $$\chi_\omega(\alpha)\overline{\chi_\omega(\beta)}=e\left(\sum_{i=1}^g\omega_i(\alpha_i-\beta_i)\right)$$ and $$dV_\omega = \sqrt{|\det((\langle u^i,u^j\rangle)_{i,j})|}d\omega_1\cdots d\omega_g.$$ Since the integration is over the fundamental domain of the lattice $H_1(G,\mbbZ)^\vee$, we have
\begin{align*}
&\int_{\mcalX}\chi_\omega(\alpha)\overline{\chi_\omega(\beta)}dV_\omega \\
&=\int_{\omega_1=0}^1\cdots \int_{\omega_g=0}^1e\left(\sum_{i=1}^g\omega_i(\alpha_i-\beta_i)\right) \sqrt{|\det((\langle u^i,u^j\rangle)_{i,j})|}d\omega_1\cdots d\omega_g  \\
&=  \vol(\mcalX) \delta_{\alpha\beta}.
\end{align*}
\end{proof}

\begin{proof}[Proof of Lemma~\ref{L:PontrDualLambda}]
Since the unit circle $S^1\simeq\mbbR/\mbbZ$ is an injective
$\mbbZ$-module, the exact sequence
$0\to\Lambda\hookrightarrow H_1(G,\mbbZ)\twoheadrightarrow Q_\Lambda\to0$
induces an exact sequence on Pontryagin duals,
$0\to\widehat{Q_\Lambda}\hookrightarrow
\widehat{H_1(G,\mbbZ)}\twoheadrightarrow\widehat{\Lambda}\to0$.
Since $\mcalX=\mcalH^1(G)/H_1(G,\mbbZ)^\vee$ and
$\mcalX_\Lambda=\mcalH^1(G)/\Lambda^\vee$, we also have the exact sequence
$0\to\Lambda^\vee/H_1(G,\mbbZ)^\vee\hookrightarrow
\mcalX\twoheadrightarrow\mcalX_\Lambda\to0$.

The isomorphism $\mcalX\simeq \widehat{H_1(G,\mbbZ)}$ defined by $\underline{\omega}\mapsto \chi_{\underline{\omega}}$ is the same as in Lemma~\ref{L:PontrDual}, which naturally induces compatible isomorphisms $\Lambda^\vee/H_1(G,\mbbZ)^\vee\simeq \widehat{Q_\Lambda}$ and $\mcalX_\Lambda\simeq \widehat{\Lambda}$. In particular, for each $\omega\in\Lambda^\vee$, $\chi_{\underline{\omega}}$ is a character of $H_1(G,\mbbZ)$ such that for each $\alpha,\alpha'\in H_1(G,\mbbZ)$ such that $\alpha-\alpha'\in \Lambda$, we have $\chi_{\underline{\omega}}(\alpha) = \chi_{\underline{\omega}}(\alpha') $. This means that $\chi_{\underline{\omega}}$ can be considered as a character of $Q_\Lambda$. Moreover, for $\omega,\omega'\in \Omega(G)$, $\chi_{\omega}|_\Lambda = \chi_{\omega'}|_\Lambda$ if and only if $\phi_1(\omega)-\phi_1(\omega')\in \Lambda^\vee$. 
\end{proof}

\begin{proof}[Proof of Proposition~\ref{P:CharacterOrthLambda}]
(a) follows by the same argument as Proposition~\ref{P:CharacterOrth}.

Since
$\sum_{\underline{\omega}\in\widehat{Q_\Lambda}}
\chi_{\underline{\omega}}(\alpha)
\overline{\chi_{\underline{\omega}}(\beta)}
=\sum_{\underline{\omega}\in\widehat{Q_\Lambda}}
\chi_{\underline{\omega}}(\alpha-\beta)$,
part (b) is the standard character-orthogonality relation for finite abelian
groups.
\end{proof}

\begin{proof}[Proof of Lemma~\ref{L:switching-similarity}]
Suppose that
$$\gamma'(\ve)=s(\ve(0))^{-1}\cdot\gamma(\ve)\cdot s(\ve(1)),$$
for some $s:V(G)\to\mbbC^\times$.  Let $D_s$ be the diagonal matrix indexed
by vertices with $(D_s)_{vv}=s(v)$.  For $v,w\in V(G)$, we have
\begin{align*}
\left(D_s^{-1}\cdot A_{\gamma,W}\cdot D_s\right)_{vw}
&=s(v)^{-1}\cdot
\sum_{\substack{\ve\in\vE(G)\\ \ve(0)=v,\ \ve(1)=w}}
W(\ve)\gamma(\ve)\cdot s(w)\\
&=\sum_{\substack{\ve\in\vE(G)\\ \ve(0)=v,\ \ve(1)=w}}
W(\ve)\,
s(\ve(0))^{-1}\cdot\gamma(\ve)\cdot s(\ve(1))\\
&=\sum_{\substack{\ve\in\vE(G)\\ \ve(0)=v,\ \ve(1)=w}}
W(\ve)\gamma'(\ve)
=\left(A_{\gamma',W}\right)_{vw}.
\end{align*}
Hence
$$A_{\gamma',W}=D_s^{-1}\cdot A_{\gamma,W}\cdot D_s.$$

Let $E_s$ be the diagonal matrix indexed by oriented edges with
$(E_s)_{\va\va}=s(\va(1))$.  If $\va\shortrightarrow\vb$, then
$\va(1)=\vb(0)$, and therefore
\begin{align*}
\left(E_s^{-1}\cdot B_{\gamma,W}\cdot E_s\right)_{\va\vb}
&=s(\va(1))^{-1}\cdot W(\vb)\gamma(\vb)\cdot s(\vb(1))\\
&=W(\vb)\,
s(\vb(0))^{-1}\cdot\gamma(\vb)\cdot s(\vb(1))\\
&=W(\vb)\gamma'(\vb)
=\left(B_{\gamma',W}\right)_{\va\vb}.
\end{align*}
If $\va\not\shortrightarrow\vb$, both entries are zero.  Thus
$$B_{\gamma',W}=E_s^{-1}\cdot B_{\gamma,W}\cdot E_s.$$
Since $s(v)\neq0$ for every vertex $v$, both $D_s$ and $E_s$ are invertible,
which proves the two similarity assertions.  If $|s(v)|=1$ for every $v$,
then both diagonal matrices are unitary, so the similarities are unitary.
\end{proof}

\begin{proof}[Proof of Lemma~\ref{L:PropAdj}]
For (a), since $W$ is real,
\begin{align*}
(A_{-\omega,W})_{vw}
&=\sum_{\substack{\ve\in\vE(G)\\ \ve(0)=v,\ \ve(1)=w}}W(\ve)e(-\omega(\ve))
=\overline{(A_{\omega,W})_{vw}},
\end{align*}
and similarly $B_{-\omega,W}=\overline{B_{\omega,W}}$.  Taking complex conjugates of characteristic polynomials gives the asserted conjugation of spectra.

For (b), assume $W$ is symmetric.  Then
\begin{align*}
(A_{-\omega,W})_{vw}
&=\sum_{\substack{\ve\in\vE(G)\\ \ve(0)=v,\ \ve(1)=w}}W(\ve)e(-\omega(\ve))  \\
&=\sum_{\substack{\ve\in\vE(G)\\ \ve(0)=w,\ \ve(1)=v}}W(\ve)e(\omega(\ve))
=(A_{\omega,W})_{wv}.
\end{align*}
Thus $\overline{A_{\omega,W}}=A_{-\omega,W}=A_{\omega,W}^T$, so $A_{\omega,W}$ is Hermitian and has real eigenvalues.  Since every matrix is similar to its transpose, $A_{\omega,W}$ and $A_{-\omega,W}$ are similar.

For (c), assume again that $W$ is symmetric.  Let $D_{\omega,W}$ be the diagonal matrix indexed by $\vE(G)$ whose $\ve$-diagonal entry is
$$\sqrt{W(\ve)}\,e(\omega(\ve)/2).$$
Then
$$B_{\omega,W}=B\cdot D_{\omega,W}^2,\qquad
C_{\omega,W}:=D_{\omega,W}\cdot B\cdot D_{\omega,W}
=D_{\omega,W}\cdot B_{\omega,W}\cdot D_{\omega,W}^{-1},$$
so $C_{\omega,W}$ is similar to $B_{\omega,W}$.  Let $R$ be the permutation matrix on oriented edges induced by $\ve\mapsto\ve^{-1}$.  Since $W$ is symmetric, $D_{-\omega,W}=R\cdot D_{\omega,W}\cdot R^{-1}$, and the unweighted edge adjacency matrix satisfies $R\cdot B\cdot R^{-1}=B^T$.  Therefore
$$C_{-\omega,W}=R\cdot C_{\omega,W}^T\cdot R^{-1}.$$
As every matrix is similar to its transpose, $C_{-\omega,W}$ is similar to $C_{\omega,W}$.  Hence $B_{-\omega,W}$ is similar to $B_{\omega,W}$.  Together with (a), this gives
$$\spec B_{\omega,W}=\spec B_{-\omega,W}=\overline{\spec B_{\omega,W}},$$
as desired.
\end{proof}

\begin{proof}[Proof of Lemma~\ref{L:LFuncExp}]
Since a finite graph has only finitely many closed walks of each length,
every coefficient below receives only finitely many contributions.  Thus the
Euler product, logarithmic expansion, and regrouping by prime powers are all
well-defined coefficientwise as formal power series.
\begin{align*}
L_W(u,\chi_\omega)
&=\prod_{[P]\in\bmcalP}\left(1-W(P)\chi_\omega(P)u^{\tau(P)}\right)^{-1} \\
&=\exp\left(\sum_{[P]\in\bmcalP}\sum_{n=1}^\infty
\frac{W(P)^n\chi_\omega(P)^n u^{n\tau(P)}}{n}\right) \\
&=\exp\left(\sum_{[C]\in\bmcalC}
\frac{W(C)\chi_\omega(C)}{r(C)}u^{\tau(C)}\right).
\end{align*}
Setting $\omega=0$ gives the formula for $z_W(u)$.
\end{proof}

\begin{proof}[Proof of Lemma~\ref{L:L-symmetry}]
Assume that $W$ is symmetric.  Since $[P]\in\bmcalP$ if and only if $[P^{-1}]\in\bmcalP$, and since $W(P^{-1})=W(P)$, we have
\begin{align*}
L_W(u,\chi_\omega)
&=\prod_{[P^{-1}]\in\bmcalP}
\left(1-W(P^{-1})\chi_\omega(P^{-1})u^{\tau(P^{-1})}\right)^{-1} \\
&=\prod_{[P]\in\bmcalP}
\left(1-W(P)\chi_{-\omega}(P)u^{\tau(P)}\right)^{-1} \\
&=L_W(u,\chi_{-\omega}).
\end{align*}
\end{proof}

\begin{proof}[Proof of Proposition~\ref{P:LFuncIden}]
The class $\underline{\omega}$ is trivial if and only if $\chi_\omega$ is trivial on $H_1(G,\mbbZ)$.  Since the abelianizations of circuits generate $H_1(G,\mbbZ)$, this is equivalent to
$$\chi_\omega(C)=1\qquad\text{for every circuit }C.$$
Thus $\underline{\omega}=0$ immediately implies $L_W(u,\chi_\omega)=z_W(u)$.

Conversely, assume $L_W(u,\chi_\omega)=z_W(u)$.  Comparing the real parts of the coefficient of $u^l$ in their logarithms, using Lemma~\ref{L:LFuncExp}, gives
$$\sum_{\substack{[C]\in\bmcalC\\ \tau(C)=l}}
\frac{W(C)}{r(C)}\bigl(1-\operatorname{Re}\chi_\omega(C)\bigr)=0.$$
Every summand is nonnegative because $W(C)>0$ and $|\chi_\omega(C)|=1$.  Hence $\chi_\omega(C)=1$ for every circuit $C$, and therefore $\underline{\omega}=0$.  The remaining equivalent descriptions in (a) follow from the definition of $\mcalX(G)$ and the Hodge decomposition.

For (b), if $\underline{\omega}=\underline{\omega'}$, then $\chi_\omega$ and $\chi_{\omega'}$ agree on $H_1(G,\mbbZ)$ and hence on every prime circuit.  Their weighted Euler products are therefore equal.
\end{proof}

\begin{proof}[Proof of Theorem~\ref{T:DertminantL}]
For (a), the standard formal identity
$$-\log\det(I-uB_{\omega,W})
=\sum_{l\geq1}\frac{\operatorname{tr}(B_{\omega,W}^l)}{l}u^l$$
applies.  The trace is the sum over closed nonbacktracking edge-walks of length $l$, with a walk $C$ contributing $W(C)\chi_\omega(C)$.  A cycle $[C]$ of length $l$ has $l/r(C)$ distinct choices of initial oriented edge.  Hence
$$\operatorname{tr}(B_{\omega,W}^l)
=\sum_{\substack{[C]\in\bmcalC\\ \tau(C)=l}}
\frac{l}{r(C)}W(C)\chi_\omega(C).$$
Together with Lemma~\ref{L:LFuncExp}, this proves
$$L_W(u,\chi_\omega)^{-1}=\det(I-uB_{\omega,W}).$$
The characteristic-polynomial identity follows immediately from the fact that $B_{\omega,W}$ has size $2m$.

For (b), let $O$ and $H$ be the $n\times2m$ origin- and terminal-incidence matrices, respectively:
$$O_{v\ve}=\begin{cases}1,&\ve(0)=v,\\0,&\text{otherwise,}\end{cases}
\qquad
H_{v\ve}=\begin{cases}1,&\ve(1)=v,\\0,&\text{otherwise.}\end{cases}$$
Let $R$ be the reversal permutation matrix on $\vE(G)$, and put
$$Z:=\operatorname{diag}\bigl(W(\ve)\chi_\omega(\ve)\bigr)_{\ve\in\vE(G)}.$$
The transition convention in Definition~\ref{D:weighted-adj} gives
$$B_{\omega,W}=(H^T\cdot O-R)\cdot Z.$$
Set $M:=I+uR\cdot Z$.  Then
$$I-uB_{\omega,W}=M-uH^T\cdot O\cdot Z,$$
and the matrix determinant lemma yields
\begin{align*}
\det(I-uB_{\omega,W})
&=\det(M)\det\bigl(I-uO\cdot Z\cdot M^{-1}\cdot H^T\bigr).
\end{align*}

Order the oriented edges in inverse pairs.  On the pair $\{\ve,\ve^{-1}\}$, the corresponding block of $M$ is
$$\begin{pmatrix}
1&uW(\ve^{-1})\chi_\omega(\ve^{-1})\\
uW(\ve)\chi_\omega(\ve)&1
\end{pmatrix},$$
whose determinant is $1-u^2r_W(\ve)$.  It follows that
$$\det(M)=\prod_{\ve\in\vE_{\mfrako}(G)}\bigl(1-u^2r_W(\ve)\bigr).$$
Inverting these $2\times2$ blocks and multiplying by the incidence matrices gives
$$I-uO\cdot Z\cdot M^{-1}\cdot H^T
=I-\mathsf A_{\omega,W}(u)+\mathsf D_W(u).$$
This proves the general vertex determinant formula.

If $W$ is reciprocal, then $r_W(\ve)=1$ for every oriented edge, and therefore
$$\mathsf A_{\omega,W}(u)=\frac{u}{1-u^2}A_{\omega,W},
\qquad
\mathsf D_W(u)=\frac{u^2}{1-u^2}(Q+I).$$
Since $m-n=g-1$, substitution into the general formula gives
$$\det(I-uB_{\omega,W})
=(1-u^2)^{g-1}\det\bigl(I-uA_{\omega,W}+u^2Q\bigr),$$
as claimed.
\end{proof}

\section{Regular graph specializations}
\label{C:regular-specializations}

For a regular graph, Theorem~\ref{T:SpecDuality} and the twisted
Ihara--Bass formula give the following concrete description of the
unweighted twisted vertex and edge spectra.

\begin{theorem}\label{T:SpecDualityReg}
Let $G$ be a connected $(q+1)$-regular graph of genus $g$, and let $\theta$
be its canonical character.  For $\omega\in\Omega(G)$,
$$
\rho(A_\omega)\leq q+1,
\qquad
\rho(B_\omega)\leq q.
$$
More precisely:
\begin{enumerate}[(a)]
\item if $\underline{\omega}=0$, then $q+1\in \spec A_\omega$ and
$q\in \spec B_\omega$, both with multiplicity $1$;
\item if $\underline{\omega}=\theta$, then $-(q+1)\in \spec A_\omega$ and
$-q\in \spec B_\omega$, both with multiplicity $1$;
\item if $\underline{\omega}\notin\{0,\theta\}$, then
$\rho(A_\omega)<q+1$, and if in addition $g\geq2$, then
$\rho(B_\omega)<q$;
\item if $G$ is bipartite, then $\theta=0$; if in addition $g\geq2$, then
the period character does not exist, and $|\lambda|<q$ for every
$\lambda\in\spec B_\omega\setminus\{\pm q\}$ with
$\underline{\omega}=0$;
\item if $G$ is non-bipartite, then $\theta\neq0$, and:
\begin{enumerate}[(i)]
\item if $\underline{\omega}=0$, then $-(q+1)\notin\spec A_\omega$; if in
addition $g\geq2$, then $|\lambda|<q$ for every
$\lambda\in\spec B_\omega\setminus\{q\}$;
\item if $\underline{\omega}=\theta$, then $q+1\notin\spec A_\omega$; if
in addition $g\geq2$, then $|\lambda|<q$ for every
$\lambda\in\spec B_\omega\setminus\{-q\}$.
\end{enumerate}
\end{enumerate}
\end{theorem}

\begin{proof}
We first consider the twisted adjacency matrices $A_\omega$.  By
Lemma~\ref{L:PropAdj}(b), $A_\omega$ is Hermitian.  Since $G$ is connected
and $(q+1)$-regular, the standard Perron--Frobenius facts give
\begin{enumerate}[(1)]
\item $\rho(A)=q+1$;
\item $q+1$ is an eigenvalue of $A$ of multiplicity $1$;
\item $-(q+1)\in\spec A$ if and only if $G$ is bipartite.
\end{enumerate}

Choose $\omega,\omega'\in\Omega(G)$ with $\underline{\omega}=0$ and
$\underline{\omega'}=\theta$.  Since
$\underline{\omega+\omega'}=\theta$, Theorem~\ref{T:SpecDuality}(SA1)
gives
$$
\spec A_\omega=\spec A=-\spec A_{\omega'}.
$$
It follows that $q+1\in\spec A_\omega$ and
$-(q+1)\in\spec A_{\omega'}$, both with multiplicity $1$.  Moreover,
$-(q+1)\in\spec A_\omega$ if and only if
$q+1\in\spec A_{\omega'}$, if and only if $G$ is bipartite, if and only if
$\theta=0$ by Theorem~\ref{T:SpecDuality}(SA3).  Theorem~
\ref{T:SpecDuality}(SR2) also gives
$\rho(A_\omega)<q+1$ whenever
$\underline{\omega}\notin\{0,\theta\}$.

We next consider $B_\omega$.  Since $G$ is $(q+1)$-regular, $Q=qI_n$.
The twisted Ihara--Bass formula in Theorem~\ref{T:DertminantL}(b) gives
$$
\prod_{\kappa\in\spec B_\omega}(1-\kappa u)
=(1-u^2)^{g-1}
\prod_{\lambda\in\spec A_\omega}(1-\lambda u+qu^2).
$$
Write
$$
(1-\kappa_1u)(1-\kappa_2u)=1-\lambda u+qu^2,
\qquad |\kappa_1|\geq|\kappa_2|.
$$
The quadratic formula gives:
\begin{enumerate}[(1)]
\item if $\lambda=q+1$, then $\kappa_1=q$ and $\kappa_2=1$;
\item if $\lambda=-(q+1)$, then $\kappa_1=-q$ and $\kappa_2=-1$;
\item if $2\sqrt q<|\lambda|<q+1$, then
$$
\kappa_{1,2}=\frac{2q}{\lambda\pm\sqrt{\lambda^2-4q}},
$$
with $|\kappa_1|\in(\sqrt q,q)$ and
$|\kappa_2|\in(1,\sqrt q)$;
\item if $|\lambda|\leq2\sqrt q$, then
$$
\kappa_{1,2}
=\frac{2q}{\lambda\pm\sqrt{4q-\lambda^2}\sqrt{-1}},
$$
and $|\kappa_1|=|\kappa_2|=\sqrt q$.
\end{enumerate}
Therefore $q\in\spec B_\omega$ with multiplicity $1$ whenever
$q+1\in\spec A_\omega$ with multiplicity $1$, and similarly
$-q\in\spec B_\omega$ with multiplicity $1$ whenever
$-(q+1)\in\spec A_\omega$ with multiplicity $1$.  If
$\rho(A_\omega)<q+1$ and $g\geq2$, then $\rho(B_\omega)<q$.  The asserted
edge-spectral properties now follow from those of $A_\omega$.

Finally, suppose that $G$ is bipartite and $g\geq2$.  If the period
character $\eta$ existed, then Theorem~\ref{T:SpecDuality}(PC1) would give
$\eta\neq0$, whereas (PC2) would give
$\rho(B_\eta)=\rho(B)=q$.  This contradicts part~(c), so the period
character does not exist.
\end{proof}

The regular bipartite case of Corollary~
\ref{C:weighted-t-circulations} consequently takes a particularly simple
unweighted form.

\begin{corollary}[\textbf{Unweighted $t$-circulations on regular bipartite
graphs}]\label{C:regular-t-circulations}
Let $G$ be a connected bipartite $(q+1)$-regular graph of genus $g\geq2$,
let $t\geq1$, and let
$\underline{\alpha}\in H_1(G,\mbbZ/t\mbbZ)$.  Then the period character
does not exist and, along even integers $l$,
$$
\pi(\underline{\alpha},l)
\sim\Pi(\underline{\alpha},l)
\sim\frac{N(\underline{\alpha},l)}l
\sim\frac{2q^l}{lt^g}.
$$
\end{corollary}

\begin{proof}
Theorem~\ref{T:SpecDualityReg}(d) shows that the period character does not
exist.  Moreover, $\nu_G=2$ and $\rho(B)=q$.  Corollary~
\ref{C:weighted-t-circulations}(a)--(b), specialized to $W\equiv1$, now
gives the result for every $t$.
\end{proof}

\section{Supplementary cotree arguments}
\label{D:supplementary-cotree}

This appendix develops a cotree viewpoint on equality in the weighted edge
spectral-radius bound.  Subsection~\ref{D:cotree-strictness} gives a direct
combinatorial comparison of equal-length nonbacktracking edge-walks.  Its
conclusion is weaker than the complete bipartite classification in
Theorem~\ref{T:SpecDuality}(SR5), but the comparison exposes arithmetic
information that the classification itself does not display.  In
Subsection~\ref{SS:2adic-refinement} returns to the bipartite Case~(4) of
that comparison and extracts a finite $2$-adic obstruction to the existence
of the period character.

\subsection{A cotree argument for edge spectral-radius rigidity}
\label{D:cotree-strictness}

\begin{proposition}[Cotree rigidity for extremal edge twists]
\label{P:cotree-edge-rigidity}
Let $G$ be a connected graph of genus $g\geq2$, let $W$ be an arbitrary
positive directed weight, and let $\underline{\omega}\in\mcalX(G)$.  If
$$
\rho(B_{\omega,W})=\rho(B_{0,W}),
$$
then $\underline{\omega}$ is a $2$-torsion character.  If, in addition, $G$
is non-bipartite, then
$$
\underline{\omega}\in\{0,\theta\}.
$$
\end{proposition}

For bipartite graphs, Proposition~\ref{P:cotree-edge-rigidity} is weaker than
Theorem~\ref{T:SpecDuality}(SR5): it leaves all $2$-torsion characters as
possible equality points, whereas (SR5) leaves only $0$ and, when it exists,
the period character $\eta$.  Its value is the purely combinatorial cotree
proof, whose explicit equal-length comparisons lead to the local and
$2$-adic obstructions in Subsection~\ref{SS:2adic-refinement}.

\CotreeStrictnessProof

\subsection{A \texorpdfstring{$2$-adic}{2-adic} cotree obstruction in the bipartite case}
\label{SS:2adic-refinement}

Assume throughout this subsection that $G$ is bipartite and has genus
$g\geq2$.  Then every
fundamental cycle has even length, so every pair of cotree edges falls under
Case~(4) in the proof of Proposition~\ref{P:cotree-edge-rigidity}; ordinary
parity gives no restriction.  The exact $2$-adic valuations retain useful
information.

Fix an orientation $\mfrako$ and a spanning tree $T$, and write
$\ve_1,\ldots,\ve_g$ for the positively oriented cotree edges.  Let
$$
C_i:=\ve_i\cdot T(\ve_i(1),\ve_i(0)),
\qquad
\tau_i:=\tau(C_i).
$$
For each pair $i<j$, replace $\ve_i$ or $\ve_j$ by its inverse when
necessary, as in the proof of Proposition~\ref{P:cotree-edge-rigidity}, and
form the closed nonbacktracking walk
$$
D_{ij}:=\ve_i\cdot T(\ve_i(1),\ve_j(1))
\cdot\ve_j^{-1}\cdot T(\ve_j(0),\ve_i(0)).
$$
Put $\tau_{ij}:=\tau(D_{ij})$ and, for convenience,
$\tau_{ji}:=\tau_{ij}$.  Thus $D_{ij}^{\ab}=C_i^{\ab}\pm C_j^{\ab}$;
the sign is immaterial for a $2$-torsion character.

\begin{algorithm}[Combined $2$-adic cotree obstruction]
\label{A:2adic-cotree-obstruction}
\textbf{Input.}\quad The period $\nu_G$ and the lengths
$\tau_i,\tau_{ij}$ above.

\smallskip
\noindent\textbf{Step 1: normalized-parity compatibility.}\quad
For every pair $i<j$, test
\begin{equation}\label{E:2adic-pair-compatibility}
\frac{\tau_{ij}}{\nu_G}
\equiv
\frac{\tau_i}{\nu_G}+\frac{\tau_j}{\nu_G}
\pmod2.
\end{equation}
Record every pair for which this congruence fails.

\smallskip
\noindent\textbf{Step 2: forced-sign propagation.}\quad
Initialize a marking vector
$$
\mathbf m=(m_1,\ldots,m_g)=(0,\ldots,0).
$$
Here $m_i=1$ records that the cotree comparison forces the character value on
$\ve_i$ to be $+1$.  For every ordered pair $i\neq j$, apply the following
rules:
\begin{enumerate}[(a)]
\item if $v_2(\tau_i)>v_2(\tau_{ij})$, set $m_i:=1$;
\item if $v_2(\tau_i)=v_2(\tau_{ij})$, set $m_j:=1$;
\item if $v_2(\tau_i)<v_2(\tau_{ij})$, record the relation $i\sim j$.
\end{enumerate}
For every recorded relation $i\sim j$, whenever one of $m_i,m_j$ is $1$,
set the other equal to $1$.  Repeat this propagation until no entry changes.

\smallskip
\noindent\textbf{Decision and output.}\quad
The algorithm has two possible outputs.  It certifies that the period
character does not exist if at least one of the following three obstruction
certificates is obtained:
\begin{enumerate}[(i)]
\item a pair tested in Step~1 fails
      \eqref{E:2adic-pair-compatibility};
\item the stabilized marking vector has an index $i$ such that $m_i=1$ and
      $$
      \frac{\tau_i}{\nu_G}\equiv1\pmod2;
      $$
\item the stabilized marking vector is $\mathbf m=(1,\ldots,1)$.
\end{enumerate}
If none of these certificates is obtained, the algorithm reports that the
test is \emph{inconclusive}.
\end{algorithm}

We explain why the algorithm is correct.  If the period character $\eta$
exists, then
$$
\eta(C_i^{\ab})=(-1)^{\tau_i/\nu_G},
\qquad
\eta(D_{ij}^{\ab})=(-1)^{\tau_{ij}/\nu_G}.
$$
Since $D_{ij}^{\ab}=C_i^{\ab}\pm C_j^{\ab}$ and $\eta$ is a $2$-torsion
character, multiplicativity gives exactly
\eqref{E:2adic-pair-compatibility}.

For Step~2, write $s_i$ for the character value on $\ve_i$.  The equal-length
comparison in the proof of Proposition~\ref{P:cotree-edge-rigidity} gives
$$
s_i^k=s_i^{k'}s_j^{k'},
\qquad
k=\frac{\tau_{ij}}{\gcd(\tau_i,\tau_{ij})},
\qquad
k'=\frac{\tau_i}{\gcd(\tau_i,\tau_{ij})}.
$$
If $v_2(\tau_i)>v_2(\tau_{ij})$, then $k$ is odd and $k'$ is even, so
$s_i=1$.  If the two valuations are equal, both exponents are odd, so
$s_j=1$.  If $v_2(\tau_i)<v_2(\tau_{ij})$, then $k$ is even and $k'$ is
odd, so $s_i=s_j$.  These are precisely the marking and propagation rules.
The invariant $m_i=1\Rightarrow s_i=1$ is therefore preserved throughout.
The propagation terminates after finitely many updates because each of the
$g$ entries can change only once, from $0$ to $1$.
Certificate~(ii) detects a contradiction between this forced value and
$\eta(C_i^{\ab})=-1$.  Under certificate~(iii), all cotree values are
$+1$, so the resulting character is trivial, whereas the period character,
when it exists in the present bipartite setting, is nontrivial.

For the bipartite theta graphs in Subsection~\ref{SS:example_spec}, this
obstruction separates the two examples immediately.  For $G_2$, one may take
$$
\bigl(\tau_1,\tau_2,\tau_{12}\bigr)=(4,6,8),
\qquad \nu_{G_2}=2.
$$
For $G_2$, the compatibility test in Step~1 produces certificate~(i), so the
period character cannot exist.  For $G_3$, Step~1 passes and Step~2 marks only
the index whose normalized length is even; here
$(\tau_1,\tau_2,\tau_{12})=(4,6,6)$.  Hence none of the three obstruction
conditions is met.  The combined algorithm is still only an
obstruction strategy:
Algorithm~\ref{A:period-character-algorithm} remains the complete decision
procedure.

\section{Weighted transform formulas}
\label{E:weighted-transforms}
Fix a positive directed weight $W$ and a full-rank sublattice
$\Lambda\subset H_1(G,\mbbZ)$.  For $\omega\in\Omega(G)$ and
$\alpha\in H_1(G,\mbbZ)$, retain the notation
$\underline{\omega}\in\mcalX$, $\dunderline{\omega}\in\mcalX_\Lambda$, and
$\underline{\alpha}=\alpha+\Lambda\in Q_\Lambda$.

Define the homology-refined weighted zeta functions by
\begin{align*}
z_{W,\alpha}(u)&:=\exp\left(
\sum_{\substack{[C]\in\bmcalC\\C^{\ab}=\alpha}}
\frac{W(C)}{r(C)}u^{\tau(C)}\right),\\
z_{W,\underline{\alpha}}(u)&:=\exp\left(
\sum_{\substack{[C]\in\bmcalC\\
\underline{C^{\ab}}=\underline{\alpha}}}
\frac{W(C)}{r(C)}u^{\tau(C)}\right).
\end{align*}
For $\alpha\in H_1(G,\mbbZ)$ and
$\dunderline{\omega}\in\mcalX_\Lambda$, define
\begin{align*}
L_{W,(\Lambda,\alpha)}(u,\chi_\omega):=
\exp\left(
\sum_{\substack{[C]\in\bmcalC\\C^{\ab}-\alpha\in\Lambda}}
\frac{W(C)\chi_\omega(C^{\ab}-\alpha)}{r(C)}u^{\tau(C)}
\right).
\end{align*}
Finally, put
$$
z_{W,\Lambda}(u):=z_{W,\underline0}(u),\qquad
L_{W,\Lambda}(u,\chi_\omega):=L_{W,(\Lambda,0)}(u,\chi_\omega).
$$
If $\alpha-\alpha'\in\Lambda$, then
\begin{align*}
\log L_{W,(\Lambda,\alpha)}(u,\chi_\omega)
=\chi_\omega(\alpha'-\alpha)
\log L_{W,(\Lambda,\alpha')}(u,\chi_\omega).
\end{align*}
When $W\equiv1$, we omit $W$ from all the notation above.

The cycle expansion and Definition~\ref{D:weighted-circuit-sums} give
\begin{align}
\log z_W(u)&=\sum_{l\geq1}\frac{N_W(l)}l u^l,\nonumber\\
\log z_{W,\alpha}(u)&=
\sum_{l\geq1}\frac{N_W(\alpha,l)}l u^l,
&
\log z_{W,\underline{\alpha}}(u)&=
\sum_{l\geq1}\frac{N_W(\underline{\alpha},l)}l u^l.
\label{E:weighted-refined-zeta}
\end{align}

\begin{proposition}[Weighted zeta and $L$-transforms]\label{T:LTransform}
The following identities hold as formal power series in $u$.
\begin{enumerate}[(a)]
\item For $\dunderline{\omega}\in\mcalX_\Lambda$ and $\alpha\in\Lambda$,
\begin{align*}
\log L_{W,\Lambda}(u,\chi_\omega)
&=\sum_{\beta\in\Lambda}\chi_\omega(\beta)
\log z_{W,\beta}(u),\\
\log z_{W,\alpha}(u)
&=\frac1{\vol(\mcalX_\Lambda)}
\int_{\mcalX_\Lambda}\chi_{-\omega}(\alpha)
\log L_{W,\Lambda}(u,\chi_\omega)dV_\omega.
\end{align*}
In particular,
$\log z_{W,\Lambda}(u)=\sum_{\beta\in\Lambda}\log z_{W,\beta}(u)$.

\item For $\underline{\omega}\in\mcalX$ and
$\alpha\in H_1(G,\mbbZ)$,
\begin{align*}
\log L_W(u,\chi_\omega)
&=\sum_{\beta\in H_1(G,\mbbZ)}\chi_\omega(\beta)
\log z_{W,\beta}(u),\\
\log z_{W,\alpha}(u)
&=\frac1{\vol(\mcalX)}
\int_{\mcalX}\chi_{-\omega}(\alpha)
\log L_W(u,\chi_\omega)dV_\omega.
\end{align*}
In particular,
$\log z_W(u)=\sum_{\beta\in H_1(G,\mbbZ)}\log z_{W,\beta}(u)$.

\item For $\dunderline{\omega}\in\mcalX_\Lambda$,
$\underline{\omega'}\in\widehat{Q_\Lambda}$, and
$\alpha\in H_1(G,\mbbZ)$,
\begin{align*}
\log L_W(u,\chi_{\omega+\omega'})
&=\sum_{\underline{\beta}\in Q_\Lambda}
\chi_{\omega+\omega'}(\beta)
\log L_{W,(\Lambda,\beta)}(u,\chi_\omega),\\
\log L_{W,(\Lambda,\alpha)}(u,\chi_\omega)
&=\frac1{|Q_\Lambda|}
\sum_{\underline{\omega'}\in\widehat{Q_\Lambda}}
\chi_{-(\omega+\omega')}(\alpha)
\log L_W(u,\chi_{\omega+\omega'}).
\end{align*}

\item For $\underline{\omega}\in\widehat{Q_\Lambda}$ and
$\underline{\alpha}\in Q_\Lambda$,
\begin{align*}
\log L_W(u,\chi_\omega)
&=\sum_{\underline{\beta}\in Q_\Lambda}
\chi_{\underline{\omega}}(\underline{\beta})
\log z_{W,\underline{\beta}}(u),\\
\log z_{W,\underline{\alpha}}(u)
&=\frac1{|Q_\Lambda|}
\sum_{\underline{\omega}\in\widehat{Q_\Lambda}}
\chi_{-\underline{\omega}}(\underline{\alpha})
\log L_W(u,\chi_\omega).
\end{align*}

\item The specializations at the zero class are
\begin{align*}
\log L_W(u,\chi_\omega)
&=\sum_{\underline{\alpha}\in Q_\Lambda}
\chi_\omega(\alpha)
\log L_{W,(\Lambda,\alpha)}(u,\chi_\omega),\\
\log L_{W,\Lambda}(u,\chi_\omega)
&=\frac1{|Q_\Lambda|}
\sum_{\underline{\omega'}\in\widehat{Q_\Lambda}}
\log L_W(u,\chi_{\omega+\omega'}),\\
\log z_W(u)&=\sum_{\underline{\alpha}\in Q_\Lambda}
\log z_{W,\underline{\alpha}}(u),\\
\log z_{W,\Lambda}(u)&=\frac1{|Q_\Lambda|}
\sum_{\underline{\omega}\in\widehat{Q_\Lambda}}
\log L_W(u,\chi_\omega).
\end{align*}
\end{enumerate}
\end{proposition}

\begin{proof}
All identities in this proof are identities of formal power series.  For
each power of $u$, only finitely many cycles contribute, so the groupings and
the passage of the resulting coefficientwise finite sums through character
integrals are legitimate.
Part (a), first identity, and part (c), first identity, follow by grouping the
weighted cycle expansion in Lemma~\ref{L:LFuncExp} by the relevant homology
class.  For the inverse formula in (a), character orthogonality gives
\begin{align*}
&\frac1{\vol(\mcalX_\Lambda)}
\int_{\mcalX_\Lambda}\chi_{-\omega}(\alpha)
\log L_{W,\Lambda}(u,\chi_\omega)dV_\omega\\
&\quad=\sum_{\substack{[C]\in\bmcalC\\C^{\ab}\in\Lambda}}
\frac{W(C)u^{\tau(C)}}{r(C)\vol(\mcalX_\Lambda)}
\int_{\mcalX_\Lambda}\chi_\omega(C^{\ab}-\alpha)dV_\omega\\
&\quad=\sum_{\substack{[C]\in\bmcalC\\C^{\ab}=\alpha}}
\frac{W(C)}{r(C)}u^{\tau(C)}
=\log z_{W,\alpha}(u).
\end{align*}
Part (b) is the case $\Lambda=H_1(G,\mbbZ)$.

For the inverse formula in (c), finite character orthogonality gives
\begin{align*}
&\frac1{|Q_\Lambda|}
\sum_{\underline{\omega'}\in\widehat{Q_\Lambda}}
\chi_{-(\omega+\omega')}(\alpha)
\log L_W(u,\chi_{\omega+\omega'})\\
&\quad=\sum_{[C]\in\bmcalC}
\frac{W(C)\chi_\omega(C^{\ab}-\alpha)}{r(C)}u^{\tau(C)}
\left(\frac1{|Q_\Lambda|}
\sum_{\underline{\omega'}\in\widehat{Q_\Lambda}}
\chi_{\omega'}(C^{\ab}-\alpha)\right)\\
&\quad=\sum_{\substack{[C]\in\bmcalC\\
C^{\ab}-\alpha\in\Lambda}}
\frac{W(C)\chi_\omega(C^{\ab}-\alpha)}{r(C)}u^{\tau(C)}
=\log L_{W,(\Lambda,\alpha)}(u,\chi_\omega).
\end{align*}
Part (d) is the specialization of (c) to
$\underline{\omega}\in\widehat{Q_\Lambda}$.  The identities in (e) follow
from (c) and (d) by taking the indicated zero character or zero homology
class.
\end{proof}

We finish by deriving Theorem~\ref{T:Fourier} from the weighted transforms,
without first using the circuit expansion of the matrix trace.
\begin{proof}[Alternative proof of Theorem~\ref{T:Fourier}]
The manipulations below are coefficientwise: at each power of $u$, only
finitely many circuits occur.
By \eqref{E:weighted-refined-zeta} and Proposition~\ref{T:LTransform}(b),
\begin{align*}
\sum_{l\geq1}N_W(\alpha,l)u^l
&=u\frac{d}{du}\log z_{W,\alpha}(u)\\
&=\frac1{\vol(\mcalX)}\int_{\mcalX}
\chi_{-\omega}(\alpha)
u\frac{d}{du}\log L_W(u,\chi_\omega)dV_\omega.
\end{align*}
The weighted determinant formula in Theorem~\ref{T:DertminantL}(a) gives
\begin{align*}
u\frac{d}{du}\log L_W(u,\chi_\omega)
&=\sum_{l\geq1}
\left(\sum_{\lambda\in\spec B_{\omega,W}}\lambda^l\right)u^l
=\sum_{l\geq1}\mcalK_W(\omega,l)u^l.
\end{align*}
Comparing coefficients proves the continuous inverse transform
\begin{align*}
N_W(\alpha,l)=\frac1{\vol(\mcalX)}
\int_{\mcalX}\chi_{-\omega}(\alpha)
\mcalK_W(\omega,l)dV_\omega.
\end{align*}
Fourier inversion on the compact torus then gives
\begin{align*}
\mcalK_W(\omega,l)=
\sum_{\alpha\in H_1(G,\mbbZ)}
\chi_\omega(\alpha)N_W(\alpha,l).
\end{align*}

Similarly, \eqref{E:weighted-refined-zeta} and
Proposition~\ref{T:LTransform}(d) give
\begin{align*}
N_W(\underline{\alpha},l)
&=\frac1{|Q_\Lambda|}
\sum_{\underline{\omega}\in\widehat{Q_\Lambda}}
\chi_{-\underline{\omega}}(\underline{\alpha})
\mcalK_W(\underline{\omega},l).
\end{align*}
Finite Fourier inversion yields
\begin{align*}
\mcalK_W(\underline{\omega},l)
=\sum_{\underline{\alpha}\in Q_\Lambda}
\chi_{\underline{\omega}}(\underline{\alpha})
N_W(\underline{\alpha},l),
\end{align*}
which proves the remaining assertions.
\end{proof}

\end{document}